\documentclass{amsart}
\usepackage{amssymb,euscript,tikz,units}
\usepackage[colorlinks,citecolor=blue,linkcolor=blue]{hyperref}
\usepackage{verbatim}
\usepackage{booktabs, array, boldline}
\usepackage{color}

\usepackage[nameinlink]{cleveref}
\usepackage{colonequals}
\usepackage{tikz-cd}
\usepackage[titletoc]{appendix}
\usepackage{yhmath}
\usepackage{enumerate}
\usepackage{hyperref}
\usepackage{url}
\usepackage{enumitem}
\usepackage{fancyhdr} 
\usepackage{lastpage}
\usepackage{enumitem}
\newcommand{\R}{\mathbb{R}}
\newcommand{\C}{\mathbb{C}}

\renewcommand{\H}{\mathbb{H}}

\newcommand{\abs}[1]{\lvert #1 \rvert}

\newcommand{\dvol}{d\mathit{vol}}

\def\dvol{\mathop{\rm dvol}}

\theoremstyle{plain}
\newtheorem{theorem}{Theorem}[section]
\newtheorem{corollary}[theorem]{Corollary}
\newtheorem{proposition}[theorem]{Proposition}
\newtheorem{lemma}[theorem]{Lemma}

\newtheorem{remark}[theorem]{Remark}
\newtheorem{conjecture}[theorem]{Conjecture}

\newcommand{\be}{\begin{equation}}
\newcommand{\ene}{\end{equation}}
\newcommand{\br}{\begin{remark}}
\newcommand{\er}{\end{remark}}
\newcommand{\bl}{\begin{lem}}
\newcommand{\el}{\end{lem}}
\newcommand{\bcor}{\begin{cor}}
\newcommand{\ecor}{\end{cor}}
\newcommand{\bpro}{\begin{pro}}
\newcommand{\epro}{\end{pro}}
\newcommand{\ben}{\begin{enumerate}}
\newcommand{\een}{\end{enumerate}}
\newcommand{\bp}{\begin{proof}}
\newcommand{\ep}{\end{proof}}
\newcommand{\bpo}{\begin{pro}}
\newcommand{\epo}{\end{pro}}
\newcommand{\beq}{\begin{equation*}}
\newcommand{\eeq}{\end{equation*}}
\newcommand{\bear}{\begin{eqnarray}}
\newcommand{\eear}{\end{eqnarray}}
\newcommand{\beqar}{\begin{eqnarray*}}
\newcommand{\eeqar}{\end{eqnarray*}}
\newcommand{\bt}{\begin{theorem}}
\newcommand{\et}{\end{theorem}}
\newcommand{\bex}{\begin{excer}}
\newcommand{\eex}{\end{excer}}

\theoremstyle{definition}

\theoremstyle{remark}

\newtheorem*{con*}{Construction}
\newtheorem*{rem*}{Remark}
\newtheorem*{exam*}{Example}
\newtheorem*{exams*}{Examples}
\newtheorem*{thm*}{\bf Theorem}
\newtheorem*{que*}{Question}
\newtheorem{que}{Question}
\newtheorem*{Def*}{Definition}
\newtheorem*{Cons*}{Construction}
\newtheorem*{Algo*}{Algorithm}
\newtheorem*{Lem*}{Lemma}
\newtheorem*{Conj*}{\bf Conjecture}

\newtheorem*{def*}{Definition}


\begin{document}

\title[Nearly optimal spectral gaps]{Nearly optimal spectral gaps for random Belyi surfaces}

\author{Yang Shen}
\address{
Fudan University, Shanghai, China}
\email[(Y.~S.)]{shenwang@fudan.edu.cn}

\author{Yunhui Wu}
\address{Tsinghua University, Beijing, China}
\email[(Y.~W.)]{yunhui\_wu@tsinghua.edu.cn}

\maketitle
\begin{abstract}
In this paper, we show that a random hyperbolic surface in the Brooks-Makover model has a spectral gap greater than $\left(\frac{1}{4}-\frac{c}{\log n}\right)$ for some universal constant $c>0$ , confirming the nearly optimal spectral gap conjecture in this model.
\end{abstract}

\section{Introduction}
Let $X_g$ be a closed hyperbolic surface of genus $g$, and let $\lambda_1(X_g)$ be its first Laplacian eigenvalue. It is known from e.g. Huber \cite{Huber74} or Cheng \cite{Cheng75} that 
$$\limsup\limits_{g\to\infty}\lambda_1(X_g)\leq\frac{1}{4}$$
for any sequence of closed hyperbolic surfaces $\{X_g\}_{g\geq 2}$. A long-standing conjecture of Buser \cite{Buser84} predicted the sharpness of this bound, i.e., the existence of a sequence $\{X_g\}$ of closed hyperbolic surfaces with $\lambda_1(X_g)\to \frac{1}{4}$ as $g\to \infty$. This was first achieved by Hide-Magee \cite{HM23}, who essentially proved the nearly optimal spectral gap for random covers of a cusped hyperbolic surface. For closed surfaces, research has focused on three prominent randomness models: the Brooks-Makover model, the Weil-Petersson model, and the covering model. It is widely conjectured that \emph{a random closed hyperbolic surface in all these three models has nearly optimal spectral gap.} In what follows, we briefly review the progress toward this conjecture across the three models.\\

\noindent \emph{Brooks-Makover Model}: Brooks and Makover \cite{BM04} proved that a random closed hyperbolic surface in their model has a uniform spectral gap $c>0$, where the constant $c$ is implicit. This was not only the first proof of a uniform spectral gap for random closed hyperbolic surfaces but also remains the only such result within their specific model.\\

\noindent \emph{Weil-Petersson Model}: The study of spectral gaps for random closed hyperbolic surfaces in this model began with Mirzakhani \cite{Mir13}, who proved a lower bound $0.00247$. This bound was significantly improved to $\left(\frac{3}{16}-\epsilon\right)$ for any $\epsilon>0$ by Wu-Xue \cite{WX22} using a completely different method, a result independently obtained by Lipnowski-Wright \cite{LW24}. In a series of subsequent works, Anantharaman-Monk raised the lower bound further to $\left(\frac{2}{9}-\epsilon\right)$ in \cite{AM23} and eventually to $\left(\frac{1}{4}-\epsilon\right)$ in the remarkable work \cite{AM25}, resolving the nearly optimal spectral gap conjecture for the Weil-Petersson model. Most recently, Hide-Macera-Thomas \cite{HMT25} provided an alternative proof via the polynomial method, additionally refining the $\epsilon$ term to a significant explicit polynomial error of order $\frac{1}{g^\delta}$ for some uniform $\delta>0$.\\

\noindent \emph{Covering Model}: Magee-Naud-Puder \cite{MNP22} proved that for any $\epsilon>0$ and any fixed closed hyperbolic surface $X$, the eigenvalues (with multiplicity) of a random covering surface of $X$ within $\left(0, \frac{3}{16}-\epsilon\right)$ coincide with those of $X$. In a recent remarkable work \cite{MPvH25}, Magee-Puder-van Handel used the polynomial method to improve the bound from $\left(\frac{3}{16}-\epsilon\right)$ to $\left(\frac{1}{4}-\epsilon\right)$, resolving the nearly optimal spectral gap conjecture for the covering model. Hide-Macera-Thomas \cite{HMTcov} refined the term $\epsilon$ to a polynomial error of order $\frac{1}{g^b}$ for some uniform $b>0$.\\

Recall that for all $n\geq 1$, set
$$\mathcal{F}_n^\star=\left\{(\Gamma,\mathcal{O});\ \begin{matrix}
    \Gamma\text{ is a $3-$regular graph with $2n$ vertices,}\\ \mathcal{O}\text{ is an orientation on }\Gamma
\end{matrix}\right\}.$$
As introduced by Brooks-Makover \cite{BM04}, there are two associated hyperbolic surfaces for each $(\Gamma,\mathcal{O})\in\mathcal{F}_n^\star$: a cusped hyperbolic surface $S^O(\Gamma,\mathcal{O})$ which is obtained by gluing $2n$ hyperbolic ideal triangles in a certain way, and a closed hyperbolic surface  $S^C(\Gamma,\mathcal{O})$ which is the conformal compactification of $S^O(\Gamma,\mathcal{O})$ by filling in all the cusps. It is known that such
constructions produce all the so-called Belyi surfaces that are dense in the space of all Riemann surfaces in a certain sense (see e.g. \cite{Be79, Ga06}). For each $n$, the set $\mathcal{F}_n^\star$ is finite, and hence there is a classical probability measure denoted by $\textnormal{Prob}_{\textnormal{BM}}^n$ on $\mathcal{F}_n^\star$, which was introduced by Bollab\'as \cite{Bollobas-iso}.

As mentioned previously, the Brooks-Makover model has stood as the final piece of the puzzle among the three models of random closed hyperbolic surfaces, being the only one for which the nearly optimal spectral gap conjecture was still unproven. Our work completes this picture by resolving the conjecture (as e.g. formulated in \cite[Page 71]{Petri-er}). Specifically, we prove

\begin{theorem}\label{mt-1}
   There exists a universal constant $c>0$ such that the following limit holds:
\[
\lim\limits_{n\to\infty}\textnormal{Prob}_{\textnormal{BM}}^n\left((\Gamma,\mathcal{O})\in\mathcal{F}_n^\star; \ \lambda_1\left(S^C(\Gamma,\mathcal{O})\right)>\frac{1}{4}-\frac{c}{\log n}\right)=1.
\]
\end{theorem}
\noindent We will first show that the spectral gap of a cusped hyperbolic surface becomes larger after compactification (see Proposition \ref{p-compare}). Therefore, to prove Theorem \ref{mt-1}, it suffices to show that a random cusped hyperbolic surface $S^O(\Gamma,\mathcal{O})$ has no eigenvalue in $\left(0,\frac{1}{4}-\frac{c}{\log n}\right)$. Our approach is inspired by the work \cite{MPvH25} of Magee-Puder-van Handel that builds upon the polynomial method in e.g. \cite{CGVTvH24, CGVvH24, MdlS24}. Meanwhile, we believe that the novel ingredients developed in this work will find further applications in studying the asymptotics of other geometric quantities in this model of random hyperbolic surfaces.
\\

For more related results of spectral gaps on random hyperbolic surfaces, one may refer to e.g. \cite{Hide23,HT25,SW25} for the Weil-Petersson model and to e.g. \cite{MN20,HM23,LM25} for the covering model. Similar results on spectral gaps also hold for random graphs. Alon \cite{Alon86} conjectured and Friedman \cite{Fr08} proved that a random $d-$regular $(d\geq 3)$ graph has nearly optimal spectral gap, as the number $N$ of vertices goes to infinity.  A recent breakthrough of Huang-Mckenzie-Yau \cite{HMY25} proved that as $N\to\infty$, about $69\%$ $d-$regular graphs with $N$ vertices are Ramanujan graphs, that is, have optimal spectral gaps. It is natural to ask whether an analogous result holds for all these three models of random closed hyperbolic surfaces. That is, 
\begin{que}
Does there exist a constant $c\in (0,1)$ such that 
$$\lim\limits_{g\to\infty}\textnormal{Prob}\left(\lambda_1(X)\geq\frac{1}{4}\right)=c?$$
\noindent where the constant $c>0$ may be different for these three models of random closed hyperbolic surfaces. 
\end{que}
The existence of infinite families of Ramanujan graphs was first proved in \cite{LPS88, Mar88, MSS15}. We remark here that it remains even \emph{open} that there exists a closed hyperbolic surface $X_g$ of genus $g$ such that $\lambda_1(X_g)\geq \frac{1}{4}$ for all $g\geq 2$. It is known from \cite{SU2013,Cook2018,BP2023} that such so-called \emph{Ramanujan surfaces} exist in genera $2$ to $7$, $14$ and $17$. In this paper, based on \cite{BLS2020} we show that Ramanujan surfaces also exist in genera \begin{equation*}
    g(N) = 1 + \frac{N^2 (N-6)}{24} \prod_{p\mid N}\Big(1 - \frac{1}{p^2}\Big)
\end{equation*}
whenever $7\leq N\leq 226$. See Theorem \ref{rs-220}.\\

\noindent\textbf{Strategy on the proof of Theorem \ref{mt-1}}.
We first give a new geometric description for each cusped hyperbolic surface $S^O(\Gamma,\mathcal{O})$ in the Brooks-Makover model.

For $n\geq 1$, assume that $S_{6n}$ is the permutation group of $6n$ elements. Set
$$\mathcal{E}_n=\left\{(\sigma,\tau);\ \begin{matrix}\sigma,\tau\in S_{6n},\text{ $\sigma$ has order $2$ and no fixed point,}\\ \text{$\tau$ has order $3$ and no fixed point}.\end{matrix}\right\}.$$
For each $(\sigma,\tau)\in\mathcal{E}_n$, it induces a group homomorphism
$$\Phi(\sigma,\tau):\textnormal{PSL}(2,\mathbb{Z})\to S_{6n},$$
which corresponds to a degree $6n$ covering hyperbolic surface $S^O(\sigma,\tau)$ of the modular surface $\mathbb{H}/\textnormal{PSL}(2,\mathbb{Z})$. For each cusped hyperbolic surface $S^O(\Gamma,\mathcal{O})$ in the Brooks-Makover model, it is isometric to certain $S^O(\sigma,\tau)$ (see e.g. Theorem \ref{t-bmrep}). One major difference between the Brooks-Makover model and the covering model of cusped hyperbolic surfaces (see e.g. \cite{HM23}) is as following: All degree $n$ covering hyperbolic surfaces of a fixed cusped hyperbolic surface $X$ correspond to all group homomorphisms $\phi\in\textnormal{Hom}(\pi_1(X),S_n)$, while the set of all hyperbolic surfaces $S^O(\sigma,\tau)$ in the Brooks-Makover model corresponds to an ``asymptotic measure zero" subset of $$\textnormal{Hom}(\textnormal{PSL}(2,\mathbb{Z}),S_{6n})$$ as $n\to\infty$.

Next, we adapt the polynomial method for the strong convergence of random matrices, introduced by Chen, Garza-Vargas, Tropp and van Handel in \cite{CGVTvH24}, and its generalization to the strong convergence of surface groups by Magee, Puder and van Handel \cite{MPvH25}, to study the strong convergence property of $\textnormal{PSL}(2,\mathbb{Z})$ in the Brooks-Makover model. We prove two important propositions: Proposition \ref{p-count} and \ref{p-bound} that correspond to Assumption 1.3 and 1.4 in \cite{MPvH25}, respectively. Recall that the group $\textnormal{PSL}(2,\mathbb{Z})$ has a unitary representation called the regular representation:
    $$\lambda:\textnormal{PSL}(2,\mathbb{Z})\to\mathcal{U}(\ell^2(\textnormal{PSL}(2,\mathbb{Z}))),\ \lambda(\gamma)[f](x)=f(\gamma^{-1}x)$$
    for any $f\in\ell^2\left(\textnormal{PSL}(2,\mathbb{Z})\right)$ and $x\in \textnormal{PSL}(2,\mathbb{Z})$.
    The regular representation could also be linearly extended to $\mathbb{C}\left[\textnormal{PSL}(2,\mathbb{Z})\right]$ where
    $$\mathbb{C}\left[\textnormal{PSL}(2,\mathbb{Z})\right]=\left\{\sum\limits_{\gamma\in\textnormal{PSL}(2,\mathbb{Z})}a_{\gamma}\gamma;\ a_\gamma\in\mathbb{C},\ a_{\gamma}\neq 0\text{ for finitely many }\gamma's\right\}.$$

\noindent For any $n\geq 1$, denote by $\mathcal{U}(n)$ the group of $n\times n$ complex unitary matrices. The $6n-$dimensional linear representation of $S_{6n}$ by permutation matrices has a $(6n-1)-$dimensional irreducible subrepresentation 
    $$\textnormal{std}:S_{6n}\to\mathcal{U}(6n-1)$$
    obtained by removing non-zero invariant vectors. Define
    \[\textnormal{tr }\Pi_\gamma:\mathcal{E}_n\to\mathbb{R}\]
    as
    \[\textnormal{tr }\Pi_\gamma(\sigma,\tau)\overset{\mathrm{def}}{=}\frac{1}{6n}\textnormal{Tr }\left(\textnormal{std}\circ\Phi(\sigma,\tau)(\gamma)\right),\]
    where $\textnormal{Tr }(\cdot)$ is the trace of a unitary matrix in $\mathcal{U}(6n-1)$. Denote by $\mathbb{E}_n\left[\cdot\right]$ the expectation on $\mathcal{E}_n$. The first proposition is as follows. 
    \begin{proposition}[=Proposition \ref{p-count}]\label{i-p-count}
    There exists a sequence $\{u_i(\gamma)\}_{i\geq 0} \subset \R$ for all $\gamma\in\textnormal{PSL}(2,\mathbb{Z})$ such that $u_0(\gamma)=\rho(\lambda(\gamma))$ and $$\left|\mathbb{E}_n\left[\textnormal{tr }\Pi_\gamma\right]-\sum\limits_{i=0}^{m-1}\frac{u_i(\gamma)}{(6n)^i}\right|\leq\frac{(18m)^{18m}}{(6n)^m}$$
    for all $m\geq |\gamma|+1$ and $n\geq 18m^{18}$.
    \end{proposition}
\noindent Assume $(\sigma,\tau)\in\mathcal{E}_n$, then there are two related group homomorphisms
$$\Phi(\sigma,\tau):\textnormal{PSL}(2,\mathbb{Z})\to S_{6n}\text{ and }\tilde{\Phi}(\sigma,\tau):\textbf{F}_2\to S_{6n}.$$
To prove Proposition \ref{i-p-count}, it suffices to count the number of fixed points of permutation $\tilde{\Phi}(\sigma,\tau)(\omega)\ (\omega\in\textbf{F}_2)$. Inspired by the work of Puder-Parzanchevski \cite{PP15}, we construct two graphs $\Gamma(\sigma,\tau)$ and $\overline{\Gamma_X(\omega)}$, then the set of fixed points of $\tilde{\Phi}(\sigma,\tau)(\omega)$ corresponds to the set of certain graph homomorphisms $\overline{f}:\overline{\Gamma_X(\omega)}\to\Gamma(\sigma,\tau)$. Such graph homomorphisms could be decomposed into two parts: surjective homomorphisms $g:\overline{\Gamma_X(\omega)}\to \Gamma$ and injective homomorphisms $\iota:\Gamma\to\Gamma(\sigma,\tau)$. In Section \ref{subsec-inj}, we deal with the part of injective homomorphisms, and calculate the expected value of the number of such homomorphisms (see Proposition \ref{l-count} and \ref{l-poly-1}). In Section \ref{subsec-surj}, we introduce an algorithm to analyze the properties of such surjective homomorphisms. Together with these two parts, we calculate the expected value $\mathbb{E}_n[\textnormal{fix}_\omega]$ for any fixed $\omega\in\textbf{F}_2$ (see Proposition \ref{p-poly-free}). Using all these properties, we will complete the proof of Proposition \ref{i-p-count}.    

Recall that $\alpha \in \mathbb{C}[\textnormal{PSL}(2,\mathbb{Z})]$ is called self-adjoint if
\[\alpha=\sum\limits_{\gamma\in\textnormal{PSL}(2,\mathbb{Z})}a_\gamma \gamma=\sum\limits_{\gamma\in\textnormal{PSL}(2,\mathbb{Z})}\overline{a_{\gamma}}\gamma^{-1}.\]
From Proposition \ref{i-p-count}, the coefficient $u_1(\cdot)$ gives a functional on $\mathbb{C}[\textnormal{PSL}(2,\mathbb{Z})]$. The second proposition is as follows.    
    \begin{proposition}[=Proposition \ref{p-bound}]\label{i-p-bound}
        For every self-adjoint $\alpha\in\mathbb{C}[\textnormal{PSL}(2,\mathbb{Z})]$, we have
        $$\limsup\limits_{p\to\infty}\left|u_1(\alpha^p)\right|^{\frac{1}{p}}\leq ||\lambda(\alpha)||,$$
        where the right hand side is the operator norm.
    \end{proposition}

It is known from e.g. \cite[Theorem 3.38]{Be-book} that $\lambda_1(\mathbb{H}/\textnormal{PSL}(2,\mathbb{Z}))>\frac{1}{4}$. Finally, based on Proposition \ref{i-p-count} and \ref{i-p-bound}, we finish the proof of Theorem \ref{mt-1} following the arguments in \cite{HMTcov, Moy25, MPvH25}.\\

\noindent\textbf{Plan of the paper.}
In Section \ref{s-bm}, we revisit the construction of the Brooks–Makover model and furnish it with a novel geometric characterization; we further establish a useful inequality relating the spectral gap of a cusped hyperbolic surface to that of its conformal compactification. In Section \ref{sec-f2}, we study certain group homomorphisms and calculate the expected value $\mathbb{E}_n[\textnormal{fix}_\omega]$ for any fixed $\omega\in\textbf{F}_2$, which is essential in the proof of Proposition \ref{i-p-count}. In Section \ref{s-scp}, we prove the strong convergence property of $\textnormal{PSL}(2,\mathbb{Z})$ and complete the proof of Theorem \ref{mt-1}. In the Appendix, we prove an analogue of Theorem \ref{mt-1} for the cusped case, building on Propositions \ref{i-p-count} and \ref{i-p-bound}. \\

\noindent\textbf{Acknowledgment.} We are very grateful to Bram Petri for invaluable discussions on the new description of the Brooks-Makover model and his comment on this
paper. We also thank Yuxin He, Yuhao Xue and Haohao Zhang for many helpful discussions and suggestions on this project. We are also grateful to Doron Puder, Joe Thomas and Ramon van Handel for their interests and comments on this paper. The first named author is supported by the NSFC grant No. 12401081, and the second named author is partially supported by the NSFC grant No. 12425107 and 12361141813.

 \tableofcontents

 \section{Brooks-Makover Model}\label{s-bm}
In this section, we recall the construction of random hyperbolic surfaces by Brooks-Makover in \cite{BM04}, and introduce an alternative description. Moreover, we also provide a useful inequality relating the spectral gap of a cusped hyperbolic surface to that of its conformal compactification. 
\subsection{Brooks-Makover random hyperbolic surfaces}
Brooks and Makover \cite{BM04} established a model of random hyperbolic surfaces. For each $n\geq 1$, denote by
$$\mathcal{F}_n\overset{\mathrm{def}}{=}\{\text{all partitions of }\{1,2,...,6n\}\text{ into pairs} \}.$$
Then the cardinality satisfies $$|\mathcal{F}_n|=    \frac{1}{(3n)!}\times\binom{6n}{2}\times\binom{6n-2}{2}\times\cdots\times\binom{2}{2}=\frac{(6n)!}{2^{3n}(3n)!}.$$
Assume that there are $2n$ vertices $\{v_1,...,v_{2n}\}$, and for each vertex $v_i\ (1\leq i\leq 2n)$, there are $3$ half-edges $\{3i-2,\ 3i-1,\ 3i\}$ emanating from it. For any element 
 $$\mathcal{P}=\langle j_1k_1\rangle...\langle j_{3n}k_{3n}\rangle\in\mathcal{F}_n,$$
 gluing the half-edges $j_t$ and $k_t$ $(1\leq t\leq 3n)$ together, one may obtain a $3-$regular graph $\Gamma(\mathcal{P})$ with $2n$ vertices. Hence $\mathcal{F}_n$ is a set of certain $3-$regular graphs. 
 
 Let $\Gamma$ be a $3-$regular graph and $V(\Gamma)$ be the set of vertices. An orientation $\mathcal{O}$ on $\Gamma$ is an assignment, for each vertex $v\in V(\Gamma)$, of a cyclic order for the three edges emanating from $v$. Denote by 
 \[\mathcal{F}_n^\star\overset{\mathrm{def}}{=}\left\{(\Gamma,\mathcal{O});\ \begin{matrix}\Gamma \text{ is a $3-$regular graph with $2n$ vertices},\\
 \mathcal{O} \text{ is an orientation on }\Gamma 
 \end{matrix}\right\}.\]
 Then we have
 \begin{align}\label{e-num}
     |\mathcal{F}_n^\star|=|\mathcal{F}_n|\times2^{2n}=\frac{(6n)!}{2^{n}(3n)!}.
 \end{align}

\noindent Now we recall the construction of Brooks and Makover in \cite{BM04}. Up to a M\"obius transformation, one may assume that a hyperbolic ideal triangle $T$ always has vertices $0,1$ and $\infty$. The ideal triangle $T$ has a natural clockwise orientation $$\left(\textbf{i},\ 1+\textbf{i}, \ \frac{1+\textbf{i}}{2}\right).$$
As in Figure \ref{fig:tri-1}, the clockwise orientation induces a direction on each side of $T$. The points $\left\{\textbf{i},\ 1+\textbf{i},\ \frac{1+\textbf{i}}{2}\right\}$ are called the \emph{mid-points} of three sides of $T$. 
\begin{figure}[h]
\includegraphics[width=0.5\textwidth]{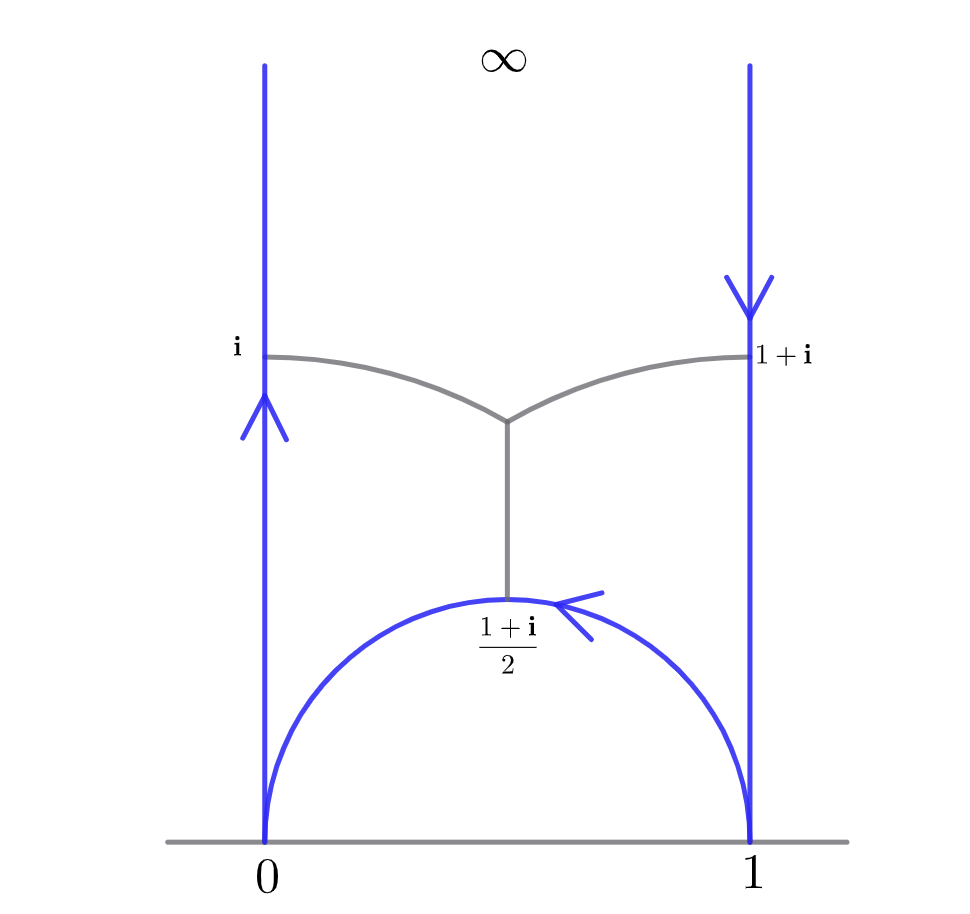}
{\caption{The standard hyperbolic ideal triangle $T$}}
\label{fig:tri-1}
\end{figure}
Given an element $(\Gamma,\mathcal{O})\in\mathcal{F}_n^\star$, we replace each vertex $v\in V(\Gamma)$ by a copy of $T$, three half-edges emanating from $v$ correspond to three sides of $T$ such that the natural clockwise orientation of $T$ coincides with the orientation of $\Gamma$ at the vertex $v$. If two vertices of $\Gamma$ are joined by an edge, we glue the two copies of $T$ along the corresponding sides subject to the following conditions:
\begin{enumerate}
\item the mid-points of two sides are glued together;
\item the gluing preserves the orientations of two copies of $T$.
\end{enumerate}
As in \cite{BM04}, the surface $S^O(\Gamma,\mathcal{O})$ is uniquely determined by the two conditions above and is a cusped hyperbolic surface with area equal to $2\pi n$.
We remark here that orientation-preserving means that two sides are glued together via reverse directions. As in Figure \ref{fig:ori}, the ideal triangle $T$ has three sides $a,\ b,\ c$ with orientation $(abc)$, and the ideal triangle $T^\prime$ has three sides $a^\prime,\ b^\prime,\ c^\prime$ with orientation $(a^\prime b^\prime c^\prime)$. If we glue the sides $a$ and $a^\prime$ together such that the orientations of $T$ and $T^\prime$ are preserved, then we obtain a quadrilateral.
\begin{figure}[h]
\includegraphics[width=\textwidth]{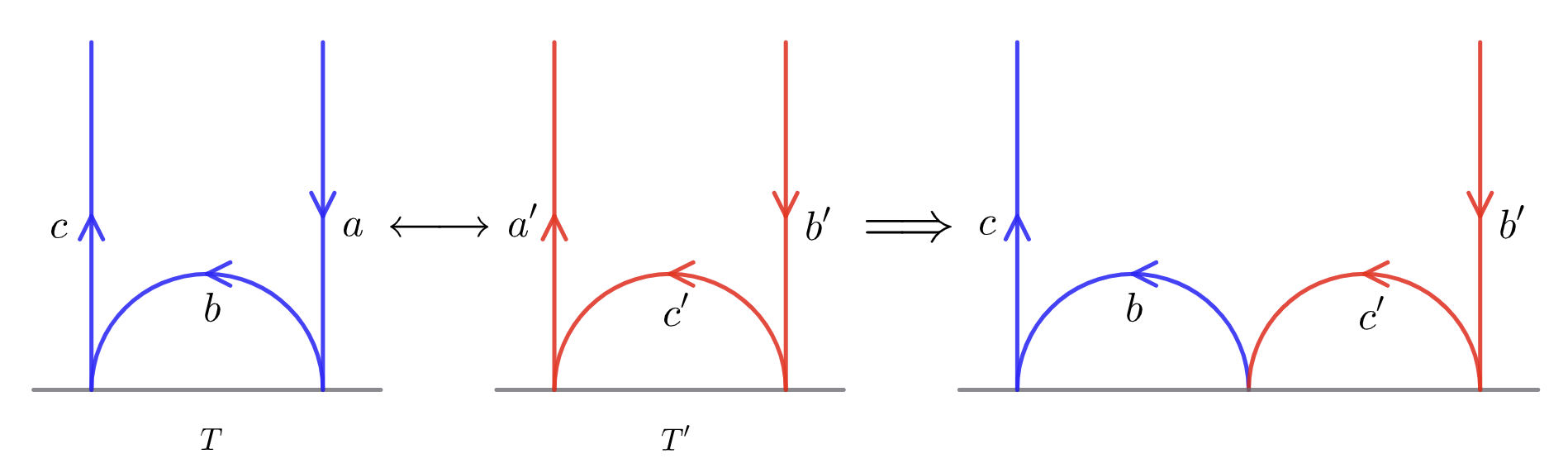}
{\caption{Orientation-preserving}}
\label{fig:ori}
\end{figure}

\begin{Def*}
The compact Riemann surface $S^C(\Gamma,\mathcal{O})$ is defined as the conformal compactification of $S^O(\Gamma,\mathcal{O})$ by filling in all the punctures.
\end{Def*}
For any fixed $n\in\mathbb{N}$, $\mathcal{F}_n^\star$ is finite, hence there is a classic probability measure $\textnormal{Prob}_{\textnormal{BM}}^n$ on it. For any random variable $f:\mathcal{F}_n^\star\to\mathbb{R}$, let $\mathbb{E}_{\textnormal{BM}}^n[f]$ be its expected value. More precisely,
$$\textnormal{Prob}_{\textnormal{BM}}^n(\mathcal{A})=\frac{|\mathcal{A}|}{|\mathcal{F}_n^\star|}\text{ \ and \  }\mathbb{E}_{\textnormal{BM}}^n[f]=\frac{\sum\limits_{(\Gamma,\mathcal{O})\in\mathcal{F}_n^\star}f(\Gamma,\mathcal{O})}{|\mathcal{F}_n^\star|},$$
where $\mathcal{A}\subset \mathcal{F}_n^\star$ is a subset. It is known from \cite[Theorem 2.1 \& Theorem 4.3]{BM04} that there exists a uniform constant $c>0$ such that
\begin{equation}\label{pro-genus}
\lim\limits_{n\to \infty} \textnormal{Prob}_{\textnormal{BM}}^n \left(\mathrm{genus(S^O(\Gamma,\mathcal{O}))}\geq c\cdot n\right)=1.
\end{equation}
One may also refer to e.g. \cite{Ga06,Pe17,BCP21,SW23,LP25,GWW25} for related results on this random surface model.

\subsection{Comparison of eigenvalues}
In this subsection, we show that the spectral gap becomes larger after conformal compactification.  More precisely,
\begin{proposition}\label{p-compare}
    Let $S^O$ be a non-compact hyperbolic surface of finite area with genus $\geq 2$, and let $S^C$ be the conformal compactification of $S^O$ obtained by filling in all the cusps of $S^O$, then 
    \[ \lambda_1(S^C) > \textnormal{RayQ}(S^O)\overset{\mathrm{def}}{=}\inf\limits_{f\in L^2(S^O),\ \int_{S^O}f=0}\frac{\int_{S^O} |\nabla f|^2\dvol_O}{\int_{S^O} f^2\dvol_O}, \] 
    where $\dvol_O$ is the volume form on $S^O$.
\end{proposition}
\begin{proof}
    Denote by $\dvol_C$ the volume form on $S^C$. There is a natural holomorphic inclusion 
    \[ (S^O,\mathrm{dvol_O}) \to (S^C,\mathrm{dvol_C}). \] 
    By the Ahlfors--Schwarz Lemma, 
    \[ \mathrm{dvol_O} > \mathrm{dvol_C}. \] 
    Let $\varphi\in C^\infty(S^C)$ be an eigenfunction w.r.t the first eigenvalue $\lambda_1(S^C)$. Define 
    \[ \widetilde{\varphi} \overset{\mathrm{def}}{=}\varphi - \frac{1}{\mathrm{vol}(S^O)}\int_{S^O}\varphi\mathrm{dvol_O}. \] 
    Then we have $\int_{S^O}\widetilde{\varphi}\mathrm{dvol_O}=0$. Moreover, we have 
    \begin{equation}\label{eqn:prop-pf-1}
        \begin{aligned}
            \int_{S^O}(\widetilde{\varphi})^2 \mathrm{dvol_O} 
            &> \int_{S^C}(\widetilde{\varphi})^2 \mathrm{dvol_C}\\
            &  =  \int_{S^C}\varphi^2 \mathrm{dvol_C} + \frac{\mathrm{vol}(S^C)}{\mathrm{vol}(S^O)^2}\Big( \int_{S^O}\varphi \mathrm{dvol_O} \Big)^2\\
            &\geq \int_{S^C}\varphi^2 \mathrm{dvol_C}, 
        \end{aligned}
    \end{equation}
    where in the second line we use $\int_{S^C}\varphi\mathrm{dvol_C}=0$. 
    By the conformal invariance of two-dimensional Dirichlet integral, we have 
    \begin{equation}\label{eqn:prop-pf-2}
        \int_{S^O}\abs{\nabla \widetilde{\varphi}}^2 \mathrm{dvol_O} = \int_{S^C}\abs{\nabla \widetilde{\varphi}}^2 \mathrm{dvol_C} = \int_{S^C}\abs{\nabla \varphi}^2 \mathrm{dvol_C}.
    \end{equation}
    It then follows from \eqref{eqn:prop-pf-1} and \eqref{eqn:prop-pf-2} that 
    \begin{equation*}
        \lambda_1(S^C) = \frac{\int_{S^C}\abs{\nabla\varphi}^2\mathrm{dvol_C}}{\int_{S^C}\varphi^2\mathrm{dvol_C}} >\frac{\int_{S^O}\abs{\nabla\widetilde{\varphi}}^2\mathrm{dvol_O}}{\int_{S^O}(\widetilde{\varphi})^2\mathrm{dvol_O}} \geq \textnormal{RayQ}(S^O).
    \end{equation*}
    The proof is complete. 
\end{proof}

The following fundamental question still remains open (see e.g. \cite{Sar03}).
\begin{que}\label{in-q-1}
Does a complete non-compact hyperbolic surface of finite area always have a non-zero eigenvalue?
\end{que}

Combined with the standard result for spectral theory (see e.g. \cite[Theorem XIII.1]{RS-book}), we have the following consequence of Proposition \ref{p-compare} that gives a new insight into Question \ref{in-q-1}.
\begin{theorem}
Let $S^O$ be a non-compact hyperbolic surface of finite area with genus $\geq 2$ such that its conformal compactification $S^C$ has first eigenvalue $\leq \frac{1}{4}$. Then $S^O$ has  a non-zero eigenvalue.
\end{theorem}

Following the terminology of \cite{BP2023}, we call a finite-area hyperbolic surface $X$ a \emph{Ramanujan surface} if $\lambda_1(X) \geq \frac{1}{4}$.
As introduced in the introduction, the following question is important and widely open (see e.g. \cite{BP2023,Hand25}).
\begin{que}
Does there exist a closed Ramanujan surface of genus $g$ for all $g\geq 2$?
\end{que}
\noindent Such surfaces are known to exist in genera $2$, $3$, and $4$, as established in \cite{SU2013} and \cite{Cook2018}, where they made numerical computations of $\lambda_1$ for the Bolza surface, the Klein quartic and the Bring's curve, respectively. More recently, in \cite{BP2023}, Bourque--Petri using linear programming methods proved the existence of closed Ramanujan surfaces in genera $5$ to $7$, $14$ and $17$.

Now we recall some arithmetics. The \emph{principal congruence subgroup} $\Gamma(N)$ of the modular group $\mathrm{PSL}(2,\mathbb{Z})$ of level $N\geq2$ is defined as 
\begin{equation*}
    \Gamma(N) \overset{\mathrm{def}}{=} 
    \left\{ 
        \begin{pmatrix}
            a & b\\ c & d
        \end{pmatrix} \in \mathrm{PSL}(2,\mathbb{Z});\ 
        \begin{pmatrix}
            a & b\\ c & d
        \end{pmatrix} \equiv 
        \begin{pmatrix}
            1 & 0\\ 0 & 1
        \end{pmatrix} \mod N
    \right\}.
\end{equation*}
The quotient $Y(N):=\mathbb{H}/\Gamma(N)$ is a non-compact hyperbolic surface of finite area. 
The genus $g(N)$ of $Y(N)$ can be computed explicitly by using the Riemann--Hurwitz formula, see e.g. \cite{Shimura1971} or \cite[Page 108]{DS2005}. For $N \geq 3$, an explicit formula is
\begin{equation*}
    g(N) = 1 + \frac{N^2 (N-6)}{24} \prod_{p\mid N}\Big(1 - \frac{1}{p^2}\Big).
\end{equation*}
This is always a non-negative integer.

The celebrated \emph{Selberg $\frac{1}{4}$-eigenvalue conjecture} states that 
\begin{conjecture}
For each integer $N\geq1$, 
\begin{equation*}
    \lambda_1(Y(N))\geq \frac14. 
\end{equation*}
\end{conjecture}
\noindent We refer the reader to \cite{Selberg1965,GJ1978,LRS95,KS2002a,KS2002b,Kim03} for the history and developments on this conjecture. Huxley \cite{Huxley1985} proved the Selberg eigenvalue conjecture for $N\leq 18$.  Recently, Booker-Lee-Str{\"o}mbergsson \cite{BLS2020} greatly improved Huxley's result. 
\begin{theorem}[{\cite[Theorem 1.1]{BLS2020}}]
    The Selberg eigenvalue conjecture holds for the principal congruence subgroups $\Gamma(N)$ with $N\leq 226$. 
\end{theorem}
\noindent Indeed, they proved that the conjecture also holds for another family of congruence subgroups $\Gamma_1(N)$ with $N\leq 880$, improving the previous result of Booker--Str{\"o}mbergsson \cite{BS2007}. The genera of the corresponding surfaces $\H/\Gamma_1(N)$ can be also explicitly computed (see e.g. \cite[Page 108]{DS2005}); however, the resulting table for $N\leq 880$ is too long and is therefore omitted here.

Let $X(N)$ be the conformal compactification of $Y(N)$ obtained by filling in all the cusps. Then by Proposition \ref{p-compare} and noting that $g(N)>1$ when $N\geq7$, we have 
\begin{theorem}\label{rs-220}
    For each $7\leq N\leq 226$, we have 
    \begin{equation*}
        \lambda_1(X(N))>\frac14. 
    \end{equation*}
\end{theorem}
\noindent In particular, there exist closed Ramanujan surfaces in all the corresponding genera (see Table \ref{tab:genera}). 

    \begin{table}[htbp]
    \centering
    \footnotesize
    \begin{tabular}{ @{} cc !{\vrule width 1pt} cc !{\vrule width 1pt} cc !{\vrule width 1pt} cc !{\vrule width 1pt} cc !{\vrule width 1pt} cc @{} }
    \toprule
    $N$ & $g(N)$ & $N$ & $g(N)$ & $N$ & $g(N)$ & $N$ & $g(N)$ & $N$ & $g(N)$ & $N$ & $g(N)$ \\
    \midrule
    3 & 0 & 42 & 1729 & 81 & 18226 & 120 & 43777 & 159 & 143209 & 198 & 207361 \\
    4 & 0 & 43 & 2850 & 82 & 15961 & 121 & 69576 & 160 & 118273 & 199 & 318451 \\
    5 & 0 & 44 & 2281 & 83 & 22100 & 122 & 53941 & 161 & 163681 & 200 & 232801 \\
    6 & 1 & 45 & 2809 & 84 & 14977 & 123 & 65521 & 162 & 113725 & 201 & 291721 \\
    7 & 3 & 46 & 2641 & 85 & 22753 & 124 & 56641 & 163 & 173800 & 202 & 249901 \\
    8 & 5 & 47 & 3773 & 86 & 18481 & 125 & 74376 & 164 & 132721 & 203 & 330961 \\
    9 & 10 & 48 & 2689 & 87 & 22681 & 126 & 51841 & 165 & 152641 & 204 & 228097 \\
    10 & 13 & 49 & 4215 & 88 & 19681 & 127 & 81313 & 166 & 137761 & 205 & 334321 \\
    11 & 26 & 50 & 3301 & 89 & 27391 & 128 & 62465 & 167 & 187083 & 206 & 265201 \\
    12 & 25 & 51 & 4321 & 90 & 18145 & 129 & 75769 & 168 & 124417 & 207 & 318385 \\
    13 & 50 & 52 & 3865 & 91 & 28561 & 130 & 62497 & 169 & 192830 & 208 & 271489 \\
    14 & 49 & 53 & 5500 & 92 & 22705 & 131 & 89376 & 170 & 141697 & 209 & 365401 \\
    15 & 73 & 54 & 3889 & 93 & 27841 & 132 & 60481 & 171 & 178201 & 210 & 235009 \\
    16 & 81 & 55 & 5881 & 94 & 24289 & 133 & 91441 & 172 & 153385 & 211 & 380276 \\
    17 & 133 & 56 & 4801 & 95 & 32041 & 134 & 71809 & 173 & 208250 & 212 & 289225 \\
    18 & 109 & 57 & 6121 & 96 & 23041 & 135 & 83593 & 174 & 141121 & 213 & 347761 \\
    19 & 196 & 58 & 5461 & 97 & 35673 & 136 & 74881 & 175 & 202801 & 214 & 297649 \\
    20 & 169 & 59 & 7686 & 98 & 27049 & 137 & 102443 & 176 & 163201 & 215 & 386233 \\
    21 & 241 & 60 & 5185 & 99 & 33481 & 138 & 69697 & 177 & 198361 & 216 & 272161 \\
    22 & 241 & 61 & 8526 & 100 & 28201 & 139 & 107066 & 178 & 170281 & 217 & 405121 \\
    23 & 375 & 62 & 6721 & 101 & 40376 & 140 & 77185 & 179 & 230956 & 218 & 314821 \\
    24 & 289 & 63 & 8209 & 102 & 27649 & 141 & 99361 & 180 & 150337 & 219 & 378289 \\
    25 & 476 & 64 & 7425 & 103 & 42875 & 142 & 85681 & 181 & 238876 & 220 & 308161 \\
    26 & 421 & 65 & 9913 & 104 & 32929 & 143 & 115081 & 182 & 177409 & 221 & 433441 \\
    27 & 568 & 66 & 7201 & 105 & 38017 & 144 & 79489 & 183 & 219481 & 222 & 295489 \\
    28 & 529 & 67 & 11408 & 106 & 35101 & 145 & 116761 & 184 & 187969 & 223 & 449625 \\
    29 & 806 & 68 & 8929 & 107 & 48178 & 146 & 93241 & 185 & 244873 & 224 & 334849 \\
    30 & 577 & 69 & 11089 & 108 & 33049 & 147 & 110545 & 186 & 172801 & 225 & 394201 \\
    31 & 1001 & 70 & 9217 & 109 & 50986 & 148 & 97129 & 187 & 260641 & 226 & 351121 \\
    32 & 833 & 71 & 13651 & 110 & 37441 & 149 & 132276 & 188 & 200929 &  &  \\
    33 & 1081 & 72 & 9505 & 111 & 47881 & 150 & 86401 & 189 & 237169 &  &  \\
    34 & 1009 & 73 & 14875 & 112 & 40705 & 151 & 137751 & 190 & 198721 &  &  \\
    35 & 1393 & 74 & 11629 & 113 & 56925 & 152 & 105121 & 191 & 281201 &  &  \\
    36 & 1081 & 75 & 13801 & 114 & 38881 & 153 & 127009 & 192 & 190465 &  &  \\
    37 & 1768 & 76 & 12601 & 115 & 57553 & 154 & 106561 & 193 & 290225 &  &  \\
    38 & 1441 & 77 & 17041 & 116 & 46201 & 155 & 143041 & 194 & 221089 &  &  \\
    39 & 1849 & 78 & 12097 & 117 & 55945 & 156 & 100801 & 195 & 254017 &  &  \\
    40 & 1633 & 79 & 18981 & 118 & 48721 & 157 & 155078 & 196 & 223441 &  &  \\
    41 & 2451 & 80 & 14209 & 119 & 65089 & 158 & 118561 & 197 & 308848 &  &  \\

    \bottomrule
     \vspace{.1in}
    \end{tabular}
    \caption{The values of $g(N)$ for $3\leq N\leq 226$ \label{tab:genera}}  
    \end{table}

\subsection{An alternative description of Brooks-Makover model}\label{sub-descri}
For $n\geq 1$, assume that $S_{6n}$ is the permutation group of $6n$ elements. Now we consider a special subset of $S_{6n}\times S_{6n}$, which is defined as
$$\mathcal{E}_n\overset{\mathrm{def}}{=}\left\{(\sigma,\tau);\ \begin{matrix}\sigma,\tau\in S_{6n},\text{ $\sigma$ has order $2$ and no fixed point,}\\ \text{$\tau$ has order $3$ and no fixed point}.\end{matrix}\right\}.$$
Now we give the construction of a hyperbolic surface from $(\sigma,\tau) \in \mathcal{E}_n$.

\begin{Cons*}[Construction of $S^O(\sigma,\tau)$]
For any $(\sigma,\tau)\in\mathcal{E}_n$, assume
$$\sigma=\left(j_1^1j_1^2\right)\cdots(j_{3n}^1j_{3n}^2)\text{ and }\tau=\left(i_1^1i_1^2i_1^3\right)\cdots\left(i_{2n}^1i_{2n}^2i_{2n}^3\right).$$
We first take $2n$ copies $\{T_1,...,T_{2n}\}$ of the hyperbolic ideal triangle $T$. For each $1\leq k\leq 2n$, label the three sides of $T_k$ by $i_k^1,\ i_k^2$ and $i_k^3$ such that the natural orientation 
$$\left(\textbf{i},1+\textbf{i},\frac{1+\textbf{i}}{2}\right)$$
of $T_k$ coincides with the cyclic orientation $(i_k^1 i_k^2 i_k^3)$.
Then for each $1\leq t\leq 3n$, glue the sides which are labeled by $j_t^1$ and $j_t^2$ together such that
\begin{enumerate}
    \item the mid-points of two sides are glued together;
    \item the gluing preserves the orientations of ideal triangles.
\end{enumerate}
Then we obtain a non-compact hyperbolic surface $S^O(\sigma,\tau)$ (see Figure \ref{fig:quo} for two examples), and hence $\mathcal{E}_n$ could be regarded as a set of certain cusped hyperbolic surfaces.
\end{Cons*}

Now we study the relationship between $\mathcal{E}_n$ and $\mathcal{F}_n^\star$.

For any $(\sigma,\tau)\in\mathcal{E}_n$, one may assume
$$\sigma=\left(j_1^1j_1^2\right)\cdots(j_{3n}^1j_{3n}^2)\text{ and }\tau=\left(i_1^1i_1^2i_1^3\right)\cdots\left(i_{2n}^1i_{2n}^2i_{2n}^3\right),$$
with $i_1^1<i_2^1<\cdots<i_{2n}^1$ and $i_k^1<i_k^2,\ i_k^3$ for all $1\leq k\leq 2n$.
Set $[6n]=\{1,2,...,6n\}$. There is a natural bijection $$h:[6n]\to [6n]$$ associated with $(\sigma,\tau)$ such that for all $1\leq k\leq 2n$,
\begin{enumerate}
    \item $h(i_{k}^1)=3k-2$;
    \item $h\left(\min\left\{i_k^2,i_{k}^3\right\}\right)=3k-1$ and $h\left(\max\left\{i_k^2,i_{k}^3\right\}\right)=3k$.
\end{enumerate}
Then we define a map $\varphi_n: \mathcal{E}_n\to\mathcal{F}_n^\star$ by 
$$\varphi_n(\sigma,\tau)\overset{\mathrm{def}}{=}(\Gamma,\mathcal{O}),$$
where
\begin{itemize}
    \item $\Gamma=\Gamma(\mathcal{P})$ for 
    $$\mathcal{P}=\langle h(j_1^1)h(j_1^2)\rangle\cdots\langle h(j_{3n}^1)h(j_{3n}^2)\rangle;$$
    \item $\mathcal{O}$ is an orientation on $\Gamma$ such that at vertex $v_k\ (1\leq k\leq 2n)$ the cyclic ordering is $\left(h(i_k^1)h(i_k^2)h(i_k^3)\right).$
\end{itemize}

\noindent It is clear that $\phi_n$ is surjective. Moreover, we have
\begin{lemma}\label{l-des}Assume $n\geq 1$, then 
\begin{enumerate}
    \item $$|\mathcal{E}_n|=\frac{(6n)!}{2^{3n}(3n)!}\times \frac{(6n)!}{3^{2n} (2n)!}.$$
    \item For any oriented graph $(\Gamma,\mathcal{O})\in\mathcal{F}_n^\star$,
    $$|\varphi_n^{-1}(\Gamma,\mathcal{O})|=\frac{(6n)!}{6^{2n}(2n)!}.$$
    \item Assume $(\sigma,\tau)\in\mathcal{E}_n$ and $(\Gamma,\mathcal{O})\in\mathcal{F}_n^\star$ satisfy $$\varphi_n(\sigma,\tau)=(\Gamma,\mathcal{O}),$$  then $S^O(\sigma,\tau)$ is isometric to $S^O(\Gamma,\mathcal{O})$.
    \end{enumerate}
    \end{lemma}
\begin{proof}
(1). The number of different $\sigma'$s is
\begin{align*}
    \frac{1}{(3n)!}\times\binom{6n}{2}\times\binom{6n-2}{2}\times\cdots\times\binom{2}{2}=\frac{(6n)!}{2^{3n}(3n)!}.
\end{align*}
The number of different $\tau'$s is
$$\frac{1}{(2n)!}\times\binom{6n}{3}\times\binom{6n-3}{3}\times\cdots\binom{3}{3}\times 2^{2n}=\frac{(6n)!}{3^{2n}(2n)!}.$$
It follows that 
$$|\mathcal{E}_n|=\frac{(6n)!}{2^{3n}(3n)!}\times \frac{(6n)!}{3^{2n} (2n)!}.$$

(2). Assume that $(\Gamma,\mathcal{O})\in\mathcal{F}_n^\star$ is fixed. Let 
$$\tau=\left(i_1^1i_1^2i_1^3\right)\cdots\left(i_{2n}^1i_{2n}^2i_{2n}^3\right),$$
with $i_1^1<i_2^1<\cdots<i_{2n}^1$ and $i_k^1<i_k^2,\ i_k^3$ $(1\leq k\leq 2n)$ such that 
$$\varphi_n(\sigma,\tau)=(\Gamma,\mathcal{O})$$
for some $\sigma$ that is uniquely determined by $\tau$ and $\Gamma$. From the definition of map $\varphi_n$, we have that for any $1\leq k\leq 2n$, the cyclic ordering $\left(h(i_k^1)h(i_k^2)h(i_k^3)\right)$ is equal to the cyclic ordering at vertex $v_k$ in $\Gamma$. From the definition of $h$, one may check that for any such $\tau$, the comparison of $i_k^2$ and $i_k^3$ $(1\leq k\leq 2n)$ is determined, i.e.
\begin{enumerate}
    \item if the cyclic ordering at $v_k$ in $\Gamma$ is $(3k-2,3k-1,3k)$, then $i_k^2<i_k^3$;
    \item if the cyclic ordering at $v_k$ in $\Gamma$ is $(3k-2,3k,3k-1)$, then $i_k^3<i_k^2$.
\end{enumerate}
Recall that the number of all $\tau$'s is $ \frac{(6n)!}{3^{2n}(2n)!}$. It follows that 
\begin{align*}
\left|\varphi_n^{-1}(\Gamma,\mathcal{O})\right|= \frac{1}{2^{2n}} \times \frac{(6n)!}{3^{2n}(2n)!}
=\frac{(6n)!}{6^{2n}(2n)!}.
\end{align*}

(3). Assume
$$\sigma=\left(j_1^1j_1^2\right)\cdots(j_{3n}^1j_{3n}^2)\text{ and }\tau=\left(i_1^1i_1^2i_1^3\right)\cdots\left(i_{2n}^1i_{2n}^2i_{2n}^3\right),$$
with $i_1^1<i_2^1<\cdots<i_{2n}^1$ and $i_k^1<i_k^2,\ i_k^3$ for all $1\leq k\leq 2n$. Then from the definition of $\varphi_n$, we have 
\begin{enumerate}[label=(\roman*)]
    \item $\Gamma=\Gamma(\mathcal{P})$ for
    $$\mathcal{P}=\langle h(j_1^1)h(j_1^2)\rangle\cdots\langle h(j_{3n}^1)h(j_{3n}^2)\rangle;$$
    \item the cyclic ordering at vertex $v_k\ (1\leq k\leq 2n)$ in $\Gamma$ is $\left(h(i_k^1)h(i_k^2)h(i_k^3)\right).$
\end{enumerate}
Now we recall the constructions of $S^O(\Gamma,\mathcal{O})$ and $S^O(\sigma,\tau)$ as follows.
\begin{itemize}
    \item $S^O(\sigma,\tau)$: take $2n$ copies $\{T_1,...,T_{2n}\}$ of the ideal triangle $T$. Label the three sides of $T_k\ (1\leq k\leq 2n)$ by $i_1^k,\ i_k^2$ and $i_k^3$ such that the cyclic orientation $\left(i_k^1i_k^2i_k^3\right)$ coincides with the natural orientation $$\left(\textbf{i},1+\textbf{i},\frac{1+\textbf{i}}{2}\right)$$
of $T_k$. Then glue the sides $j_s^1$ and $j_s^2\ (1\leq s\leq 3n)$ together subject to two precise conditions;
\item $S^O(\Gamma,\mathcal{O})$: replace the vertex $v_k\ (1\leq k\leq 2n)$ by a copy $T_k^\prime$ of the ideal triangle $T$, the three sides of $T_k^\prime$ correspond to the three half-edges $$\{3k-2,3k-1,3k\}$$ emanating from $v_k$ such that the cyclic orientation $\left(h(i_k^1)h(i_k^2)h(i_k^3)\right)$ coincides with the natural orientation 
$$\left(\textbf{i},1+\textbf{i},\frac{1+\textbf{i}}{2}\right)$$
of $T_k^\prime$. Then glue the sides $h(j_s^1)$ and $h(j_s^2)\ (1\leq s\leq 3n)$ together subject to two precise conditions.
\end{itemize}
Identify $T_k$ with $T_k^\prime\ (1\leq k\leq 2n)$ and the three sides $i_k^1,\ i_k^2,\ i_k^3$ of $T_k$ correspond to the three sides $h(i_k^1),\ h(i_k^2),\ h(i_k^3)$ of $T_k^\prime$, respectively. It is clear that the construction of $S^O(\sigma,\tau)$ coincides with the construction of $S^O(\Gamma,\mathcal{O})$. Hence $S^O(\sigma,\tau)$ is isometric to $S^O(\Gamma,\mathcal{O})$.

The proof is complete.
\end{proof}

Since $\mathcal{E}_n\ (n\geq 1)$ is a finite set, there is a classic probability $\textnormal{Prob}_n$ on it. For any random variable $f:\mathcal{E}_n\to\mathbb{R}$, let $\mathbb{E}_n[f]$ be its expected value. More precisely, 
$$\textnormal{Prob}_n(\mathcal{A})=\frac{|\mathcal{A}|}{|\mathcal{E}_n|}\text{ \ and \ }\mathbb{E}_n[f]=\frac{\sum\limits_{(\sigma,\tau)\in\mathcal{E}_n}f(\sigma,\tau)}{|\mathcal{E}_n|},$$
where $\mathcal{A}$ is a subset of $\mathcal{E}_n$. 

Then we have the following proposition.
\begin{proposition}\label{p-prob}
    Let $P$ be a geometric property of hyperbolic surfaces, denote by
    $$\mathcal{A}_{\textnormal{BM}}^n(P)=\left\{(\Gamma,\mathcal{O})\in\mathcal{F}_n^\star;\ S^O(\Gamma,\mathcal{O})\text{ satisfies property }P\right\}$$
    and 
    $$\mathcal{A}_n(P)=\{(\sigma,\tau)\in\mathcal{E}_n;\ S^O(\sigma,\tau)\text{ satisfies property }P\}.$$
    Then for any $n\geq 1$, 
    \begin{align*}
    \textnormal{Prob}_{\textnormal{BM}}^n\left(\mathcal{A}_{\textnormal{BM}}^n(P)\right)
        =\textnormal{Prob}_n\left(\mathcal{A}_n(P)\right).
    \end{align*}
\end{proposition}
\begin{proof}
For any $(\sigma,\tau)\in\mathcal{E}_n$ and $(\Gamma,\mathcal{O})\in\mathcal{F}_n^\star$ with $\varphi_n(\sigma,\tau)=(\Gamma,\mathcal{O})$, it follows from Part $(3)$ of Lemma \ref{l-des} that $S^O(\Gamma,\mathcal{O})$ is isometric to $S^O(\sigma,\tau)$. Hence
    $$\mathcal{A}_n(P)=\{(\sigma,\tau)\in\varphi_n^{-1}(\Gamma,\mathcal{O}); \ (\Gamma,\mathcal{O})\in\mathcal{A}_{\textnormal{BM}}^n(P)\}.$$
Then from Part $(2)$ of Lemma \ref{l-des} we have that
    $$|\mathcal{A}_n(P)|=|\mathcal{A}_{\textnormal{BM}}^n(P)|\times\frac{(6n)!}{6^{2n}(2n)!}.$$
This, together with \eqref{e-num} and  Part (1) of Lemma \ref{l-des}, implies that    \begin{align*}
        \textnormal{Prob}_n\left(\mathcal{A}_n(P)\right)&=\frac{|\mathcal{A}_n(P)|}{|\mathcal{E}_n|}=\frac{|\mathcal{A}_{\textnormal{BM}}^n(P)|\times\frac{(6n)!}{6^{2n}(2n)!}}{|\mathcal{F}_n^\star|\times\frac{(6n)!}{6^{2n}(2n)!}}\\
        &=\frac{|\mathcal{A}_{\textnormal{BM}}^n(P)|}{|\mathcal{F}_n^\star|}=
        \textnormal{Prob}_{\textnormal{BM}}^n\left(\mathcal{A}_{\textnormal{BM}}^n(P)\right).
    \end{align*}
    The proof is complete.
\end{proof}
 
 On the other hand, $\mathcal{E}_n$ could also be regarded as a space of certain special group homomorphisms. It is well-known that $\textnormal{PSL}(2,\mathbb{Z})\simeq\mathbb{Z}_2\star\mathbb{Z}_3$ and matrices
$$b=\begin{pmatrix}
  0 & 1\\ -1 & 0
\end{pmatrix}\text{ and }c=\begin{pmatrix}
0 & 1\\ -1 & 1
\end{pmatrix},$$
are generators of $\mathbb{Z}_2$ and $\mathbb{Z}_3$ respectively. For any $(\sigma,\tau)\in\mathcal{E}_n$, there is a unique group homomorphism $\Phi(\sigma,\tau):\textnormal{PSL}(2,\mathbb{Z})\to S_{6n}$ such that
$$\Phi(\sigma,\tau)(b)=\sigma\text{ and }\Phi(\sigma,\tau)(c)=\tau.$$
The group $\textnormal{PSL}(2,\mathbb{Z})$ acts on $\mathbb{H}\times [6n]$ by 
$$\gamma(x,i)=(\gamma(x),\Phi(\sigma,\tau)(\gamma)(i))$$
 for any $\gamma\in\textnormal{PSL}(2,\mathbb{Z}),\ x\in\mathbb{H}$ and $i\in [6n]$. A Belyi surface can be viewed as a covering surface of the modular space (see \cite[Section 1]{BM01}) and \cite[Subsection 2.6]{LP25}. We provide a detailed proof in our setting here.

\begin{theorem}\label{t-bmrep}
 The quotient space $\textnormal{PSL}(2,\mathbb{Z})\setminus
_{\Phi(\sigma,\tau)}\mathbb{H}\times [6n]$ is a covering surface of the modular surface $\mathbb{H}/\textnormal{PSL}(2,\mathbb{Z})$ and is isometric to $S^O(\sigma,\tau)$.
\end{theorem}
\begin{proof}
    We write $\Phi=\Phi(\sigma,\tau)$ for simplicity. Also write $$S=\textnormal{PSL}(2,\mathbb{Z})\setminus_{\Phi}\mathbb{H}\times[6n].$$
    Consider a fundamental domain $F$ of the modular surface $\mathbb{H}/\textnormal{PSL}(2,\mathbb{Z})$ defined as
    $$F\overset{\mathrm{def}}{=}\left\{x+\textbf{i}y;\ \begin{matrix}0\leq x\leq \frac{1}{2},\  (x-1)^2+y^2\geq 1\end{matrix}\right\}.$$
    Then $F$ is a geodesic triangle with three sides and
    $$\mathbb{H}=\bigcup\limits_{\gamma\in\textnormal{PSL}(2,\mathbb{Z})}\gamma F,$$
    where for $\gamma_1\neq \gamma_2\in\textnormal{PSL}(2,\mathbb{Z})$, $\gamma_1F$ and $\gamma_2F$ have no interior intersection (see the first one of Figure \ref{fig:tri}). 
    \begin{figure}[h]
\includegraphics[width=0.8\textwidth]{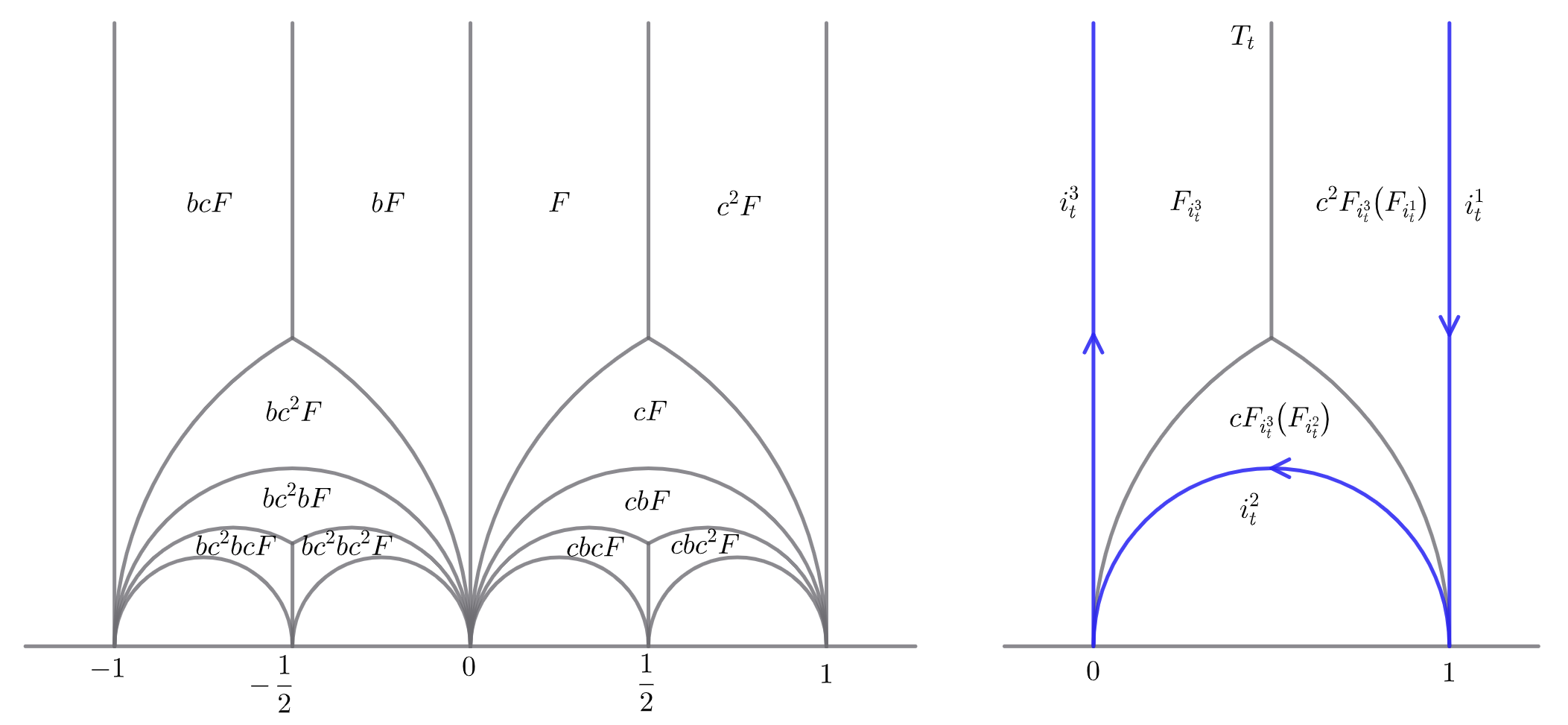}
{\caption{Tilings}}
\label{fig:tri}
\end{figure}
    For $1\leq i\leq 6n$ and $\gamma\in\textnormal{PSL}(2,\mathbb{Z})$, denote by
    $$ F_{\gamma, i}\overset{\mathrm{def}}{=}\{(\gamma p,i);\ p\in F\},$$
    which is a subset of $\mathbb{H}\times [6n]$, and write $F_{i}=F_{e,i}$ for simplicity, where $e$ is the unitary element in $\textnormal{PSL}(2,\mathbb{Z})$. Then we have
    \begin{enumerate}
        \item for any $ \gamma\in\textnormal{PSL}(2,\mathbb{Z})$ and $1\leq i\leq 6n$, $F_{\gamma,i}$ is identical to $F_{\Phi(\gamma^{-1})(i)}$ in $S$;
        \item for any $1\leq i\neq j\leq 6n$, $F_i$ and $F_j$ have no interior intersection in $S$.
    \end{enumerate}
    Hence, the quotient space $S$ could be obtained by gluing all $F_i's$ along their sides. For any $1\leq i\leq 6n$, the domain $F_i$ has a unique side, which is a geodesic line joining $0$ and $\infty$. On such a side, there is a natural direction from $0$ to $\infty$ with a mid-point $\textbf{i}$. We also label such a side by $i$. Now we assume 
    $$\sigma=(j_1^1j_1^2)\cdots(j_{3n}^1j_{3n}^2)\text{ and }\tau=(i_1^1i_1^2i_1^3)\cdots(i_{2n}^1i_{2n}^2i_{2n}^3).$$
    Recall that 
    $$b=\begin{pmatrix}
        0 & 1\\ -1 & 0
    \end{pmatrix}\text{ and }c=\begin{pmatrix}
        0 & 1\\ -1 &1
    \end{pmatrix}.$$
    Then we have
    \begin{enumerate}
        \item for $1\leq t\leq 2n$, as the second one of Figure \ref{fig:tri}, $F_{i_t^1}$ is identical to $ F_{c^2, i_t^3}$ in $S$, and $F_{i_t^2}$ is identical to $ F_{c, i_t^3}$ in $S$. It follows that $F_{i_t^1},\ F_{i_t^2},\ F_{i_t^3}$ are glued together to form an ideal triangle $T_t$ in $S$. The three sides of the ideal triangle $T_t\ (1\leq t\leq 2n)$ have been labeled by $i_t^1,\ i_t^2,\ i_t^3$ such that the cyclic orientation  $(i_t^1i_t^2i_t^3)$ coincides with the natural orientation 
     $$\left(\textbf{i},1+\textbf{i},\frac{1+\textbf{i}}{2}\right)$$
     of $T_k$. Moreover, the directions on these three sides $i_t^1,\ i_t^2$ and $i_t^3$ coincide with the ones that are induced by the orientation of $T_t$.
        \item For $1\leq s\leq 3n$, $F_{j_s^1}$ is identical to $F_{b,j_s^2}$ in $S$. It follows that the sides $j_s^1$ and $j_s^2$ are glued together via reverse directions in $S$. Moreover, the mid-points of sides $j_s^1$ and $j_s^2$ are glued together.
    \end{enumerate}
     Then from the construction of $S^O(\sigma,\tau)$, one may conclude that
     $S$ is isometric to $S^O(\sigma,\tau)$.
The proof is complete.
\end{proof}

\begin{figure}[h]
\includegraphics[width=.8\textwidth]{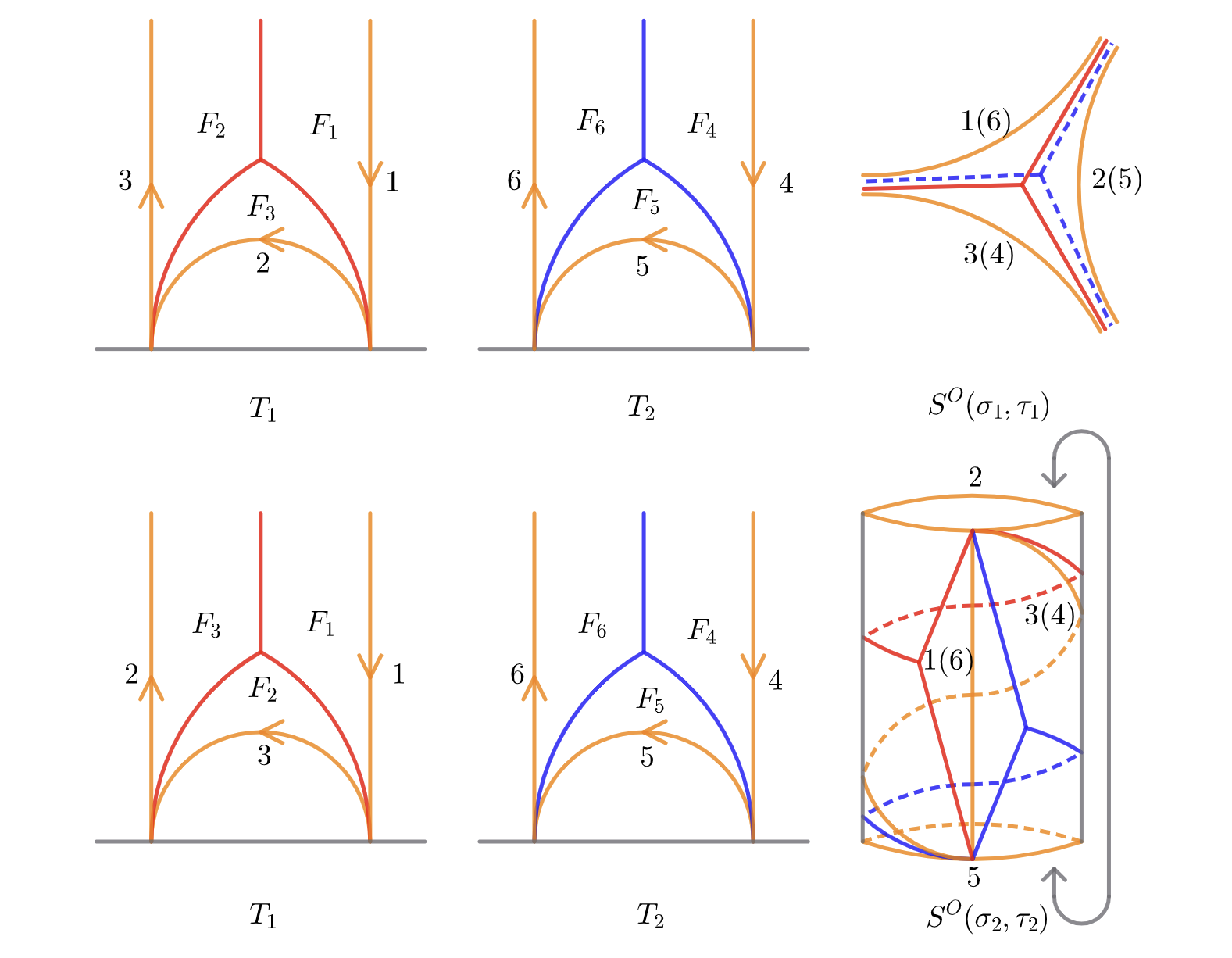}
{\caption{Two examples for $S^O(\sigma,\tau)$}}
\label{fig:quo}
\end{figure}

Now we give two examples (see Figure \ref{fig:quo}). 
\begin{enumerate}
    \item Assume $\sigma_1=(16)(25)(34)$ and $\tau_1=(123)(456)$. Then $F_1,\ F_2$ and $F_3$ are glued together to form an ideal triangle $T_1$ with orientation $(123)$, $F_4,\ F_5$ and $F_6$ are glued together to form an ideal triangle $T_2$ with orientation $(456)$. Gluing the geodesic boundaries of $T_1$ and $T_2$ according to $\sigma_1$, we obtain a non-compact hyperbolic surface $S^O(\sigma_1,\tau_1)\simeq S_{0,3}$, a hyperbolic three-punctured sphere.
    \item Assume $\sigma_2=(16)(25)(34)$ and $\tau_2=(132)(456)$. Then $F_1,\ F_2$ and $F_3$ are glued together to form an ideal triangle $T_1$ with orientation $(132)$, $F_4,\ F_5$ and $F_6$ are glued together to form an ideal triangle $T_2$ with orientation $(456)$. Gluing the geodesic boundaries of $T_1$ and $T_2$ according to $\sigma_2$, we obtain a non-compact hyperbolic surface $S^O(\sigma_2,\tau_2)\simeq S_{1,1}$, a hyperbolic once-punctured torus.
\end{enumerate}

\section{Group representations of $\textbf{F}_2$}\label{sec-f2}
Assume $X=\{x_1,x_2\}$ is a basis for the free group $\textbf{F}_2$ of rank $2$.  For $n\geq 1$ and $(\sigma,\tau)\in\mathcal{E}_n$, there exists a unique group homomorphism $$\tilde\Phi(\sigma,\tau):\textbf{F}_2\to S_{6n}$$ such that
$$\tilde\Phi(\sigma,\tau)(x_1)=\sigma\text{ and }\tilde\Phi(\sigma,\tau)(x_2)=\tau.$$
This gives a one-to-one correspondence between $\mathcal{E}_n$ and the set of all such $\tilde\Phi:\textbf{F}_2\to S_{6n}$. Hence $\mathcal{E}_n$ could be regarded as a space of certain special group homomorphisms from $\textbf{F}_2$ to $S_{6n}$. In this section, we study their related properties.
\subsection{Labeled graphs and homomorphisms}\label{sec-label} 
Assume $X=\{x_1,x_2\}$ is a basis for the free group $\textbf{F}_2$. As in \cite{PP15}, a graph $\Gamma$ is called \emph{$X-$labeled} if 
\begin{enumerate}
    \item each edge in $E(\Gamma)$ is directed and labeled by $x_1\text{ or } x_2$;
    \item any two different edges emanating from (terminating at) the same vertex have different labels.
\end{enumerate}
We remark here that an $X-$labeled graph $\Gamma$ may be disconnected, moreover it may have loops and multi-edges. For each edge in $E(\Gamma)$, it could be represented as $$u\overset{x_i}{\to}v\text{ or }v\overset{x_i^{-1}}{\to}u.$$

Now we introduce two useful constructions of $X-$labeled graphs.

\begin{Cons*}[graph associated with a word]
    For any reduced word $\omega=e_1e_2\cdots e_k\in\textbf{F}_2$ where $e_i\in\left\{x_1^{\pm},x_2^{\pm}\right\}$, denote by $\Gamma_X(\omega)$ a closed path with its edges labeled as follows:
$$u_1\overset{e_k}{\to}u_2\overset{e_{k-1}}{\to}\cdots\overset{e_2}{\to} u_k\overset{e_1}{\to} u_1.$$ 
See the first one of Figure \ref{fig:00} where $\omega=x_1x_2x_1x_2^2$.
\end{Cons*} 

\begin{figure}[h]
\includegraphics[width=.8\textwidth]{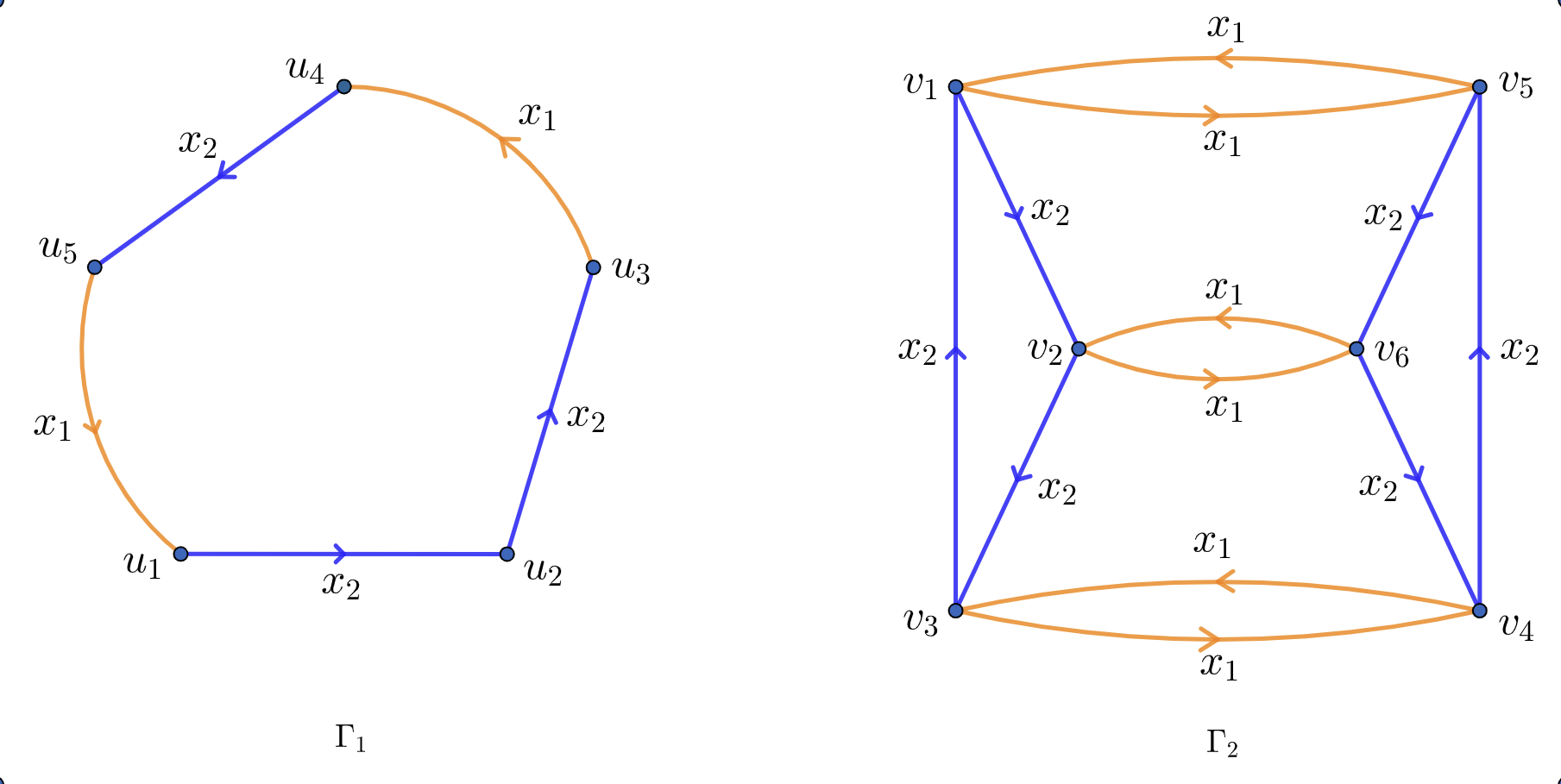}
{\caption{Two $X-$labeled graphs}}
\label{fig:00}
\end{figure}

\begin{Cons*}[graph associated with $(\sigma,\tau)$]
    for any $n\geq 1$ and $(\sigma,\tau)\in\mathcal{E}_n$, the associated $X-$labeled graph $\Gamma(\sigma,\tau)$ is constructed as follows,
\begin{enumerate}
    \item  firstly, there are $6n$ vertices $v_1,...,v_{6n}$;
    \item for $1\leq i\leq 6n$, join $v_i$ and $v_{\sigma(i)}$ with a directed and labeled edge 
    $$v_i\overset{x_1}{\to}v_{\sigma(i)};$$
    \item for $1\leq i\leq 6n$, join $v_i$ and $v_{\tau(i)}$ with a directed and labeled edge 
    $$v_i\overset{x_2}{\to}v_{\tau(i)}.$$
\end{enumerate}
See the second one of Figure \ref{fig:00}, where $\Gamma_2=\Gamma(\sigma,\tau)$ for
$$\sigma=(15)(26)(34)\text{ and }\tau=(123)(456).$$
\end{Cons*}

\noindent By construction, it is easy to see that for any $(\sigma,\tau)\in\mathcal{E}_n\ (n\geq 1)$, the $X-$labeled graph $\Gamma(\sigma,\tau)$ satisfies 
\begin{enumerate}
    \item 
    the vertex set $$V(\Gamma(\sigma,\tau))=\{ v_1,...,v_{6n}\},$$ and for each vertex $v_i\ (1\leq i\leq 6n)$, there are exactly four different edges based at it: two outgoing edges $x_{1},\ x_{2}$ and two incoming edges $x_1,\ x_{2}$;
    \item the edge set $E(\Gamma(\sigma,\tau))$ consists of cycles that are of the following two forms,
    $$v_i\overset{x_1}{\to}v_{\sigma(i)}\overset{x_1}{\to}v_i\text{ for } 1\leq i\leq 6n$$
    or
    $$v_i\overset{x_2}{\to}v_{\tau(i)}\overset{x_2}{\to}v_{\tau^2(i)}\overset{x_2}{\to}v_i\text{ for } 1\leq i\leq 6n.$$
\end{enumerate}

\noindent Assume that $G$ and $H$ are both $X-$labeled graphs. Recall that as in \cite{PP15}, an $X-$labeled graph homomorphism $$f:G\to H$$ is a pair of maps $V(G)\to V(H)$ and $E(G)\to E(H)$ such that  
 $$f(u\overset{e}{\to}v)=f(u)\overset{e}{\to}f(v)$$
 for any edge $u\overset{e}{\to}v\in E(G)$ where $u,v\in V(G)$ and $e\in\{x_1,x_2\}$. It is called an $X-$labeled graph isomorphism if $f$ is a bijective $X-$labeled graph homomorphism. Now we give an example for an $X-$labeled graph homomorphism as follows: Let $\Gamma_1, \Gamma_2$ be as in Figure \ref{fig:00}, the graph homomorphism $h:\Gamma_1\to \Gamma_2$ is defined by
\begin{enumerate}
    \item for $1\leq i\leq 5$, set 
     $$h(u_i)=v_i;$$
    \item for $1\leq i\leq 4$, set
    $$h(u_i\overset{e_i}{\to}u_{i+1})=v_i\overset{e_i}{\to}v_{i+1},$$ where $e_i\in\{x_1,x_2\}$ is uniquely determined by $i$, and set
    $$h(u_5\overset{x_1}{\to}u_1)=v_5\overset{x_1}{\to}v_1.$$
\end{enumerate}
For this particular example, the $X-$labeled graph homomorphism $h$ defined above gives an embedding of $\Gamma_1$ into $\Gamma_2$. However, in general, an $X-$labeled graph homomorphism $h$ defined on $\Gamma_X(\omega)$ is not necessarily an embedding.

 \subsection{Counting of fixed points}
Assume $\omega=e_1e_2\cdots e_k\ (e_i\in\{x_1^{\pm},x_2^{\pm}\})$ is a reduced word in $\textbf{F}_2$ and $(\sigma,\tau)\in\mathcal{E}_n\ (n\geq 1)$, then $\tilde{\Phi}(\sigma,\tau)(\omega)$ is a permutation in $S_{6n}$. Denote by
$$\textnormal{fix}_\omega(\sigma,\tau)=\#\left\{\begin{matrix}\text{fixed points of the}\\ \text{permutation }\tilde{\Phi}(\sigma,\tau)(\omega)\end{matrix}\right\}.$$
Then  $$\textnormal{fix}_\omega:\mathcal{E}_n\to\mathbb{Z}^{\geq 0}$$ 
gives a random variable on $\mathcal{E}_n$. In this subsection, we will calculate its expected value $$\mathbb{E}_n[\textnormal{fix}_\omega]=\frac{\sum\limits_{(\sigma,\tau)\in\mathcal{E}_n}\textnormal{fix}_\omega(\sigma,\tau)}{|\mathcal{E}_n|}.$$
First of all, we have the following lemma.
  \begin{lemma}\label{l-equal}
  Assume $n\geq 1$. Then for any reduced word $\omega\in\textbf{F}_2$ and $(\sigma,\tau)\in\mathcal{E}_n$,
  \begin{align*}
     \textnormal{fix}_\omega(\sigma,\tau) =\#\left\{\begin{matrix}\text{$X-$labeled graph homomorphism }\\ f:\Gamma_X(\omega)\to\Gamma(\sigma,\tau)\end{matrix}\right\}.
  \end{align*}
  \end{lemma}
  \begin{proof}
  Assume $\omega=e_1e_2\cdots e_k$, where $e_i\in\{x_1^{\pm},x_2^{\pm}\}$ and $k$ is the word length of $\omega$.
      For $1\leq i\leq k$, set 
$$\xi_i=\begin{cases}
    \sigma^{\pm}&\text{ if }e_i=x_1^{\pm};
    \\
    \tau^{\pm}&\text{ if }e_i=x_2^{\pm}.
\end{cases}$$
Denote by $\rho_0=id$ and $$\rho_i=\xi_{k-i+1}\cdots\xi_k.$$
It is obvious that $\rho_i\in S_{6n}\ (0\leq i\leq k)$ and $\rho_k=\xi_{1}\cdots\xi_k$. Fix $1\leq j\leq 6n$. Since $\rho_{i+1}(j)=\xi_{k-i}(\rho_i(j))\ (0\leq i\leq k-1)$, from the construction of $\Gamma(\sigma,\tau)$ there
exists an edge $$v_{\rho_i(j)}\overset{e_{k-i}}{\to}v_{\rho_{i+1}(j)}$$ that joins $v_{\rho_i(j)}$ and $v_{\rho_{i+1}(j)}$ in $\Gamma(\sigma,\tau)$. It follows that there is a consecutive path $P_j$
$$v_j\overset{e_k}{\to}v_{\rho_1(j)}\overset{e_{k-1}}{\to}\cdots\overset{e_2}{\to}v_{\rho_{k-1}(j)}\overset{e_1}{\to} v_{\rho_{k}(j)}$$
in $\Gamma(\sigma,\tau)$.
Recall that $\Gamma_X(\omega)$ is a closed path 
$$u_1\overset{e_k}{\to}u_2\overset{e_{k-1}}{\to}\cdots\overset{e_2}{\to}u_k\overset{e_1}{\to}u_1.$$
It follows that
\begin{align*}
    \rho_k(j)=j
    &\Longleftrightarrow P_j\text{ is a closed path}\\
    &\Longleftrightarrow\begin{matrix}\text{there is a unique $X-$labeled graph homomorphism}\\ f:\Gamma_X(\omega)\to\Gamma(\sigma,\tau) \text{ such that $f(u_1)=v_j$.}\end{matrix}
\end{align*}
Notice that
$$\tilde{\Phi}(\sigma,\tau)(\omega)=\rho_k=\xi_1\xi_2\cdots\xi_k,$$
it follows that
\begin{align*}
    \textnormal{fix}_\omega(\sigma,\tau)&=\#\{\text{fixed points of }\rho_k\} \\
    &=\#\left\{\begin{matrix}\text{$X-$labeled graph homomorphism }\\ f:\Gamma_X(\omega)\to\Gamma(\sigma,\tau)\end{matrix}\right\}.
\end{align*}
The proof is complete.
  \end{proof}
\begin{exam*}
We consider an example (see Figure \ref{fig:00}) for $\omega=x_1x_2x_1x_2^2$ and 
$$\sigma=(15)(26)(34),\ \tau=(123)(456).$$
Then it is direct to check that $$\tilde{\Phi}(\sigma,\tau)(\omega)=id.$$
Hence for each $1\leq i\leq 6$, $i$ is a fixed point of $\tilde{\Phi}(\sigma,\tau)(\omega)$. Meanwhile, for each $1\leq i\leq 6$, there exists a unique $X-$labeled graph homomorphism $h_i:\Gamma_1\to \Gamma_2$ such that $h_i(u_1)=v_i$. 
\end{exam*}

    From now on, we always assume that a word $\omega\in\textbf{F}_2$ is of form
    $$\omega=x_1x_2^{i_1}\cdots x_1x_2^{i_k},$$
    where $k\geq 1$ and $i_j\in\{1,2\}\ (1\leq j\leq k)$.
  
  \begin{Cons*}[completion of $\Gamma_X(\omega)$]
     Based on $\Gamma_X(\omega)$, we construct an associated $X-$labeled graph $\overline{\Gamma_X(\omega)}$ as follows:
\begin{enumerate}
    \item for any edge $v_1\overset{x_1}{\to}v_2$ in $\Gamma_X(\omega)$, add an edge $v_2\overset{x_1}{\to}v_1$ to $\Gamma_X(\omega)$;
    \item for any consecutive edges $$u_1\overset{x_1}{\to}v_1\overset{x_2}{\to}v_2\overset{x_1}{\to}u_2$$ in $\Gamma_X(\omega)$, add a new vertex $w$, and two new edges $$v_2\overset{x_2}{\to}w, \ \omega\overset{x_2}{\to}v_1$$ to $\Gamma_X(\omega)$;
    \item for any consecutive edges $$u_1\overset{x_1}{\to}v_1\overset{x_2}{\to}v_2\overset{x_2}{\to}v_3\overset{x_1}{\to}u_2$$ in $\Gamma_X(\omega)$, add a new edge $v_3\overset{x_2}{\to}v_1$ to $\Gamma_X(\omega)$.
\end{enumerate} 
  \end{Cons*}

\begin{figure}[h]
\includegraphics[width=.8\textwidth]{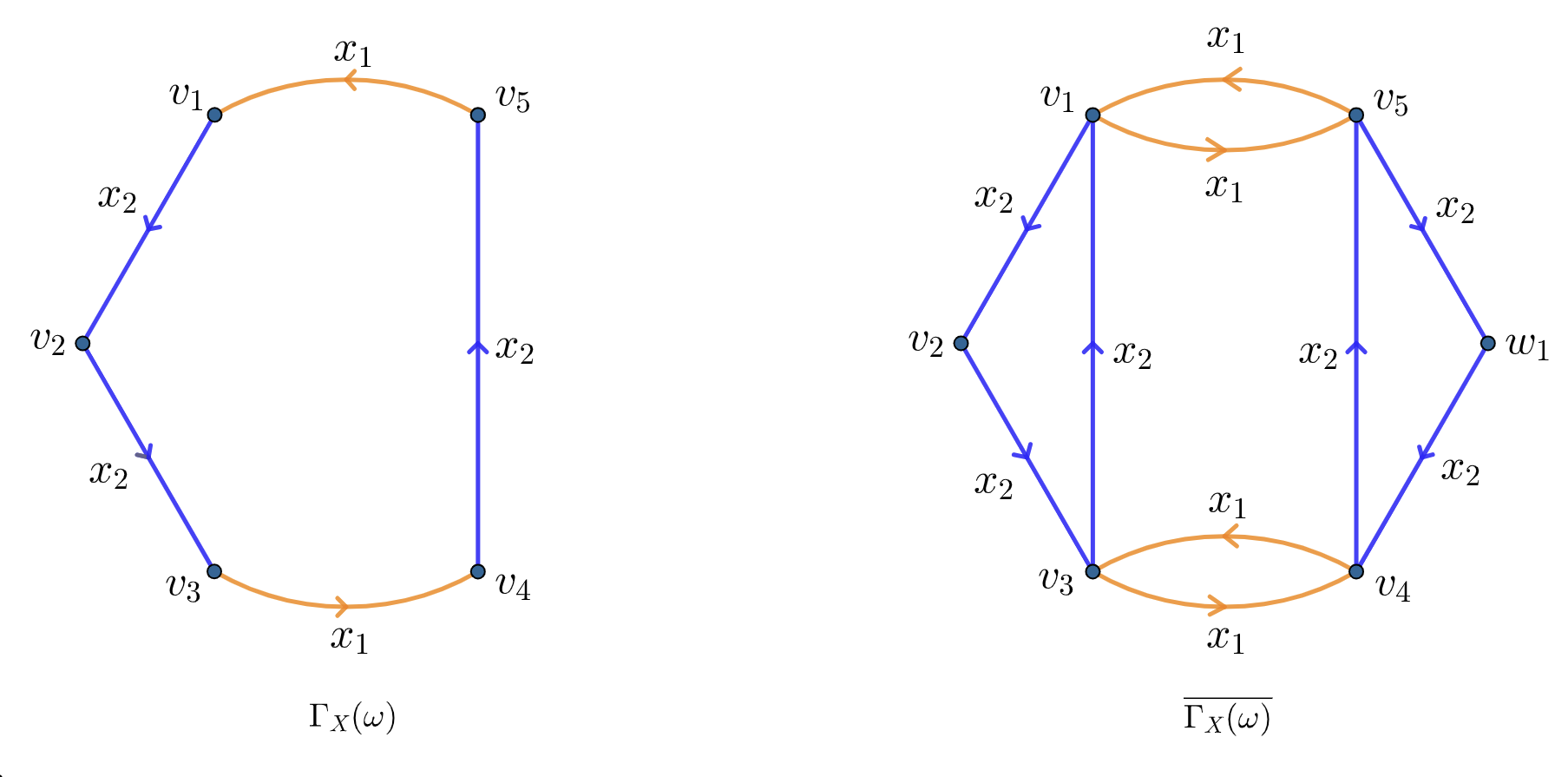}
{\caption{Completion of $\Gamma_X(\omega)$}
\label{fig:02}}
\end{figure}
\noindent One may see Figure \ref{fig:02} for an example of $\omega=x_1x_2x_1x_2^2$.
It is obvious that for any $(\sigma,\tau)\in\mathcal{E}_n$ and $X-$labeled graph homomorphism $f:\Gamma_X(\omega)\to\Gamma(\sigma,\tau)$, $f$ could be uniquely extended to an $X-$labeled graph homomorphism $$\overline{f}:\overline{\Gamma_X(\omega)}\to\Gamma(\sigma,\tau).$$
Denote by 
\begin{equation}\label{def-Fo}
F(\omega)=\left\{(g,\Gamma);\ \begin{matrix}
     g:\overline{\Gamma_X(\omega)}\to\Gamma\text{ is a surjective $X-$labeled} \\ \text{homomorphism and $\Gamma$ is a subgraph }\\ \text{of $\Gamma(\sigma,\tau)$ for some }(\sigma,\tau)\in\mathcal{E}_n\ (n\geq 1)
 \end{matrix}\right\}.
 \end{equation}
 For any two elements $(g_1,\Gamma_1),\ (g_2,\Gamma_2)\in F(\omega)$, we say that $(g_1,\Gamma_1)$ and $(g_2,\Gamma_2)$ are \emph{equivalent} if there exists an $X-$labeled graph isomorphism $\xi:\Gamma_1\to\Gamma_2$ such that $g_2=\xi\circ g_1$, i.e. the following diagram commutes:
  $$\begin{tikzcd}[column sep=scriptsize, row sep=scriptsize]
     & & \Gamma_1\arrow{dd}{\xi}\\
    &\overline{\Gamma_{X}(\omega)}\arrow{ru}{g_1}\arrow{rd}[swap]{g_2}&\\
    & &\Gamma_2
    \end{tikzcd}.
    $$
Denote by $$\mathcal{I}(\omega)=\textit{  the set of all equivalent classes in $F(\omega)$.}$$ 

\noindent For any $X-$labeled graph homomorphism $$\overline{f}:\overline{\Gamma_X(\omega)}\to\Gamma(\sigma,\tau),$$ 
let $\Gamma$ be the image $\overline{f}\left(\overline{\Gamma_X(\omega)}\right)\subset \Gamma(\sigma,\tau)$ that induces a unique injective $X-$labeled homomorphism $\iota:\Gamma\to\Gamma(\sigma,\tau)$ such that $\overline{f}=\iota\circ \overline{f}$. Moreover, $\left(\overline{f}, \overline{f}\left(\overline{\Gamma_X(\omega)}\right)\right)$ corresponds to a unique element in $\mathcal{I}(\omega)$ denoted by $(g,\Gamma)$. Thus, together with Lemma \ref{l-equal}, we have that for any fixed $(\sigma,\tau)\in\mathcal{E}_n$ and $\omega\in\textbf{F}_2$,
\begin{equation}
\begin{aligned}\label{e-f}
  \textnormal{fix}_\omega(\sigma,\tau)&=\#\left\{\begin{matrix}\text{$X-$labeled graph homomorphism }\\ f:\Gamma_X(\omega)\to\Gamma(\sigma,\tau)\end{matrix} \right\}\\
   &=\#\left\{\begin{matrix}\text{$X-$labeled graph homomorphism }\\ \overline{f}:\overline{\Gamma_X(\omega)}\to\Gamma(\sigma,\tau) \end{matrix}\right\}\\
   &=\sum\limits_{(g,\Gamma)\in\mathcal{I}(\omega)}\#\left\{\begin{matrix}\text{injective $X-$labeled graph}\\ \text{homomorphism }  \iota:\Gamma\to\Gamma(\sigma,\tau)\end{matrix}\right\}.
\end{aligned}
\end{equation}
\subsubsection{Injective $X-$labeled graph homomorphisms}\label{subsec-inj}
We first consider injective $X-$labeled graph homomorphisms. For any $(\sigma,\tau)\in\mathcal{E}_n\ (n\geq 1)$ and $(g,\Gamma)\in\mathcal{I}(\omega)$, set
$$\mathcal{N}_{\Gamma}(\sigma,\tau)=\left\{\begin{matrix}\text{injective $X-$labeled graph }\\ \text{homomorphism }\iota:\Gamma\to\Gamma(\sigma,\tau)\end{matrix}\right\}$$
and 
$$N_\Gamma(\sigma,\tau)=\#\mathcal{N}_\Gamma(\sigma,\tau).$$
Then $$N_\Gamma:\mathcal{E}_n\to\mathbb{Z}^{\geq 0}$$ is a random variable.

\begin{Def*}
    \begin{enumerate}
        \item An $X-$labeled graph $H_1$ is called an $x_1-$cycle if
        $$V(H_1)=\{u_1,u_2\}\text{ and }E(H_1)=\{u_1\overset{x_1}{\to}u_2,u_2\overset{x_1}{\to}u_1\}.$$
        \item An $X-$labeled graph is called an $x_2-$cycle if
        $$V(H_2)=\{u_1,u_2,u_3\}\text{ and }E(H_2)=\{u_1\overset{x_2}{\to}u_2,u_2\overset{x_2}{\to}u_3,u_3\overset{x_2}{\to}u_1\}.$$
    \end{enumerate}
\end{Def*}
One may see Figure \ref{fig:cycle} for illustrations.
\begin{figure}[h]
\includegraphics[width=.8\textwidth]{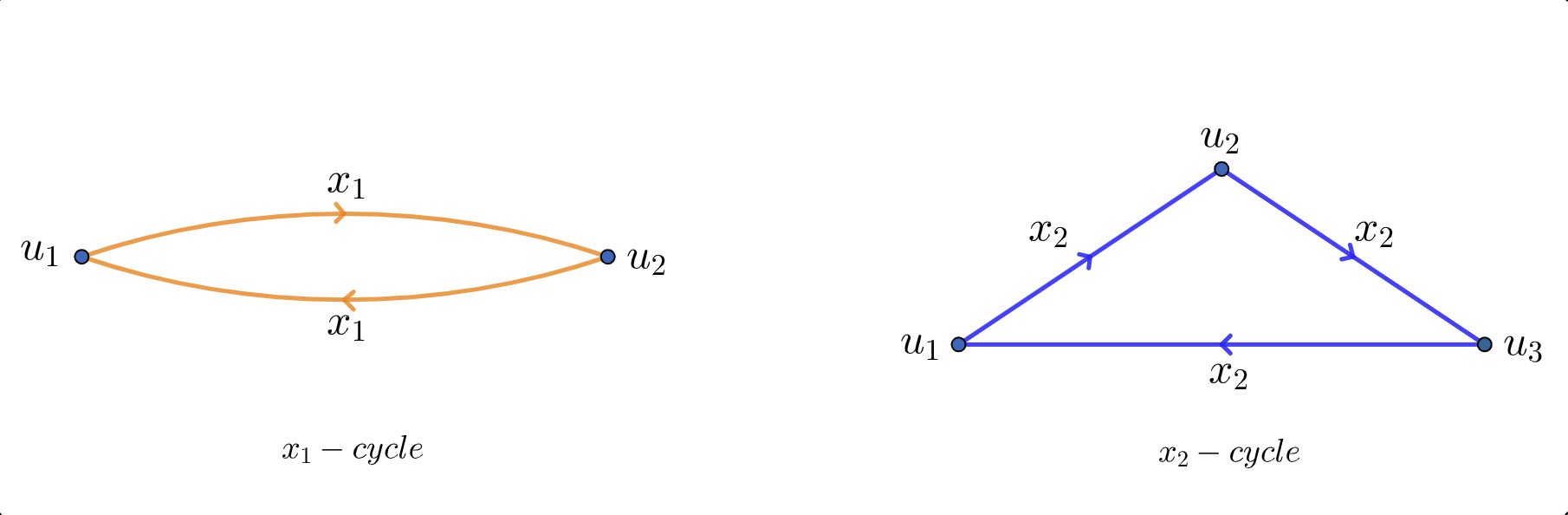}
{\caption{Two cycles}}
\label{fig:cycle}
\end{figure}

Recall that the word $\omega\in\textbf{F}_2$ is of form  
$$\omega=x_1x_2^{i_1}\cdots x_1x_2^{i_k}$$
where $k\geq 1,\ i_j\in\{1,2\}\ (1\leq j\leq k)$. For any element $(g,\Gamma)\in\mathcal{I}(\omega)$, since $g$ maps $x_i-$cycles in $\overline{\Gamma_X(\omega)}$ to $x_i-$cycles in $\Gamma(\sigma, \tau)$ for $i\in\{1,2\}$, it follows that $\Gamma=g\left(\overline{\Gamma_X(\omega)}\right)$ is a union of certain $x_1-$cycles and $x_2-$cycles. Moreover, it satisfies that
\begin{enumerate}
    \item for any vertex $v\in V(\Gamma)$, it is contained in some an $x_2-$cycle of $\Gamma$;
    \item for any $x_2-$cycle of $\Gamma$, at least two of its vertices are contained in some $x_1-$cycles of $\Gamma$.
\end{enumerate}
Assume that $\Gamma$ has $p$ $x_1-$cycles and $q$ $x_2-$cycles. From condition (i) above, we have
$$|V(\Gamma)|=3q.$$
From condition (ii) above, we have $p\geq q$. See Figure \ref{fig:oper} for an example in which $p=4$ and $q=3$. Define $$\eta(\Gamma)\overset{\mathrm{def}}{=}p-q\geq 0.$$
Then we have the following lemma.
\begin{proposition}\label{l-count}
    With the same assumptions as above, if $|V(\Gamma)|=o\left(n^\frac{1}{2}\right)$, then we have
    $$\mathbb{E}_n(N_\Gamma)=\frac{1}{(6n)^{\eta(\Gamma)}}\left(1+O\left(\frac{|V(\Gamma)|^2}{n}\right)\right),$$
    where the implied constant is independent of $n$ and $\Gamma$.  
\end{proposition}
\begin{proof}
   For any $(\sigma,\tau)\in\mathcal{E}_n$ and $\iota\in\mathcal{N}_\Gamma(\sigma,\tau)$, it    naturally induces an injective map
   $$\iota_{v}:V(\Gamma)\to \{v_1,...,v_{6n}\}.$$
   Since $|V(\Gamma)|=3q$, it follows that there are
   \begin{align}\label{e-inj-num}
   \prod\limits_{i=0}^{3q-1}(6n-i)
   \end{align}
   choices of such injective maps. Now we fix an injective map $$\zeta:V(\Gamma)\to\{v_1,...,v_{6n}\}.$$
   Assume $\iota\in\mathcal{N}_\Gamma(\sigma,\tau)$ such that $\iota_v=\zeta$. Then for any vertices $u_1\neq u_2\in V(\Gamma)$ and edge $u_1\overset{e}{\to}u_2\in E(\Gamma)\ (e\in\{x_1,x_2\})$, there exists an edge 
   $$\zeta(u_1)\overset{e}{\to}\zeta(u_2)\in E(\Gamma(\sigma,\tau)).$$
   Since $\zeta$ is injective and $\Gamma$ has $p$ $x_1-$cycles and $q$ $x_2-$cycles, it follows that there are $p$ $x_1-$cycles and $q$ $x_2-$cycles in $\Gamma(\sigma,\tau)$ that are determined by $\zeta$. Recall that there are $3n$ $x_1-$cycles and $2n$ $x_2-$cycles in $\Gamma(\sigma,\tau)$, similar to the proof of Part (1) of Lemma \ref{l-des} there are $$\frac{(6n-2p)!}{(3n-p)!2^{3n-p}}\times\frac{(6n-3q)!}{(2n-q)!3^{2n-q}}$$
   choices to determine the remaining $(3n-p)$ $x_1-$cycles and $(2n-q)$ $x_2-$cycles. Hence we have 
    $$\#\left\{\iota;\ \begin{matrix}
        \iota\in\mathcal{N}_\Gamma(\sigma,\tau)\text{ for some }(\sigma,\tau)\in\mathcal{E}_n\\
        \text{ such that } \iota_v=\zeta
    \end{matrix}\right\}=\frac{(6n-2p)!}{(3n-p)!2^{3n-p}}\times\frac{(6n-3q)!}{(2n-q)!3^{2n-q}}.$$
    This, together with \eqref{e-inj-num} and Part $(1)$ of Lemma \ref{l-des}, implies that
    \begin{equation}
    \begin{aligned}\label{e-inj-1}
        \mathbb{E}_n(N_\Gamma)&=\frac{\sum\limits_{(\sigma,\tau)\in\mathcal{E}_n}N_\Gamma(\sigma,\tau)}{|\mathcal{E}_n|} \\ &=\frac{\prod\limits_{i=0}^{3q-1}(6n-i)\times\frac{(6n-2p)!}{(3n-p)!2^{3n-p}}\times\frac{(6n-3q)!}{(2n-q)!3^{2n-q}}}{\frac{(6n)!}{(3n)!2^{3n}}\times\frac{(6n)!}{(2n)!3^{2n}}}\\
        &=\frac{\prod\limits_{i=0}^{p-1}(3n-i)\times\prod\limits_{i=0}^{q-1}(2n-i)}{\prod\limits_{i=0}^{2p-1}(6n-i)}\times 2^p3^q.
    \end{aligned}
    \end{equation}
    For any vertex $v\in V(\Gamma)$, it is contained in at most one $x_1-$cycle, and hence
    \begin{align}\label{e-inj-2}
        p\leq \frac{|V(\Gamma)|}{2}.
    \end{align}
    For $k=o\left(m^\frac{1}{2}\right)$ and $m>0$ large enough, it is known that
    $$\prod\limits_{i=0}^{k-1}(m-i)=m^k\left(1+O\left(\frac{k^2}{m}\right)\right).$$
    This together with \eqref{e-inj-1} implies that
    \begin{align*}
        \mathbb{E}_n(N_\Gamma)&=\frac{(3n)^p\times(2n)^q}{(6n)^{2p}}\times 2^p3^q\left(1+O\left(\frac{p^2+q^2}{n}\right)\right)\\
        &=\frac{1}{(6n)^{\eta(\Gamma)}}\left(1+O\left(\frac{|V(\Gamma)|^2}{n}\right)\right),
    \end{align*}
    where we apply $|V(\Gamma)|=3q$ and \eqref{e-inj-2} in the last equation. The proof is complete.
\end{proof}

The following lemma is the same as \cite[Lemma 2.2]{MPvH25} except an explicit expression for the constant $C^\prime$ in \cite[Lemma 2.2]{MPvH25}.
\begin{lemma}\label{l-poly}
    Let $P,\ Q$ be real polynomials of degree at most $m\in\mathbb{N}$, and let 
    $$\Phi(t)=\frac{P(t)}{Q(t)}.$$
    Assume that there is a constant $C>1$ such that 
    \begin{align}\label{e-cond}
    \left|P\left(\frac{1}{n}\right)\right|\leq C\text{ and }C^{-1}\leq Q\left(\frac{1}{n}\right)\leq C
    \end{align}
    for all $n\in\mathbb{N}$ with $n\geq (Cm)^C$. Denote by $C^\prime=4C+5$, then for all $l\leq Cm$,
    $$\sup\limits_{t\in\left[0,(C^\prime m)^{-C^\prime}\right]}\left|\Phi^{(l)}(t)\right|\leq (C^\prime m)^{C^\prime l}.$$
\end{lemma}
\begin{proof}
    From the proof of \cite[Lemma 2.2]{MPvH25}, we have 
    $$\left|\Phi^{(l)}(t)\right|\leq 2^l(2C)^{2l+2}(2Cm)^{4Cl}l!$$
    for any $t\in[0,(2Cm)^{-2C-1}]$. Hence for any $l\leq Cm$, 
    \begin{align*}
       \sup\limits_{t\in\left[0,(2C m)^{-2C-1}\right]}\left|\Phi^{(l)}(t)\right|&\leq 2^l(2C)^{2l+2}(2Cm)^{4Cl}l!\\
       &\leq (2C)^{2l+2}(2Cm)^{(4C+1)l}\\
       &\leq (2Cm)^{(4C+5)l}
    \end{align*}
    which implies that
    $$\sup\limits_{t\in\left[0,(C^\prime m)^{-C^\prime}\right]}\left|\Phi^{(l)}(t)\right|\leq\sup\limits_{t\in\left[0,(2C m)^{-2C-1}\right]}\left|\Phi^{(l)}(t)\right|\leq (C^\prime m)^{C^\prime l}.$$
    The proof is complete.
\end{proof}
Now we prove a more precise version of Proposition \ref{l-count} for a fixed word $\omega\in\textbf{F}_2$.
\begin{proposition}\label{l-poly-1}
    Assume word $\omega=x_1x_2^{i_1}\cdots x_1x_2^{i_k}\in\textbf{F}_2$, where $k\geq 1$ and $i_j\in\{1,2\}\ (1\leq j\leq k)$. Also assume  $(g,\Gamma)\in\mathcal{I}(\omega)$, then
    \begin{enumerate}
   \item   there exists a sequence $\{a_i(\Gamma)\}_{i\geq 0}\subset\mathbb{R}$ such that $a_0(\Gamma)=1$ and
   $$\mathbb{E}_n(N_\Gamma)=\frac{1}{(6n)^{\eta(\Gamma)}}\sum\limits_{i=0}^\infty \frac{a_i(\Gamma)}{(6n)^i}.$$
   \item For all $m\geq 2k$, $l\leq m$ and $6n\geq (17m)^{17}$,
    \begin{align*}
    \left|\mathbb{E}_n(N_\Gamma)-\frac{1}{(6n)^{\eta(\Gamma)}}\sum\limits_{i=0}^{l-1}\frac{a_i(\Gamma)}{(6n)^i}\right|\leq\frac{(17m)^{17l}}{l!(6n)^{l+\eta(\Gamma)}}.
\end{align*}
\end{enumerate}
\end{proposition}
\begin{proof}
   (1). Assume that $\Gamma$ has $p$ $x_1-$cycles and $q$ $x_2-$cycles. Since $(g,\Gamma)\in\mathcal{I}(\omega)$, we have
    $$p,q\leq k\text{ and }|V(\Gamma)|=3q.$$
    From \eqref{e-inj-1} in the proof of Lemma \ref{l-count}, we have
    \begin{align*}
    \mathbb{E}_n(N_\Gamma)&=\frac{\prod\limits_{i=0}^{p-1}(3n-i)\times\prod\limits_{i=1}^{q-1}(2n-i)}{\prod\limits_{i=0}^{2p-1}(6n-i)}\times 2^p3^q\\
    &=\frac{1}{(6n)^{\eta(\Gamma)}}\times\frac{\prod\limits_{i=0}^{p-1}\left(1-\frac{2i}{6n}\right)\times\prod\limits_{i=1}^{q-1}\left(1-\frac{3i}{6n}\right)}{\prod\limits_{i=0}^{2p-1}\left(1-\frac{i}{6n}\right)}.
    \end{align*}
     Set 
$$P(t)=\prod\limits_{i=0}^{p-1}\left(1-2it\right)\times\prod\limits_{i=0}^{q-1}\left(1-3it\right)
\text{ and } 
Q(t)=\prod\limits_{i=0}^{2p-1}\left(1-it\right).$$
Then the function $\frac{P}{Q}(t)$ could be represented as $$\frac{P}{Q}(t)=\sum\limits_{i=0}^\infty a_i(\Gamma)t^i$$
for a sequence $\{a_i(\Gamma)\}_{i\geq 0}\subset\mathbb{R}$. Hence
$$a_0(\Gamma)=\frac{P}{Q}(0)=1$$
and
$$\mathbb{E}_n(N_\Gamma)=\frac{1}{(6n)^{\eta(\Gamma)}}\times \frac{P}{Q}\left(\frac{1}{6n}\right)=\frac{1}{(6n)^{\eta(\Gamma)}}\sum\limits_{i=0}^\infty\frac{a_i(\Gamma)}{(6n)^i}.$$
(2). Assume $m\geq 2k\geq 2p$. For 
 $6n\geq (3m)^3$, we have
$$\left|P\left(\frac{1}{6n}\right)\right|<3$$
and 
\begin{align*}
3& >Q\left(\frac{1}{6n}\right)=\prod\limits_{i=0}^{2p-1}\left(1-\frac{i}{6n}\right)\\ &\geq \prod\limits_{i=0}^{2p-1}\left(1-\frac{i}
{(3m)^{3}}\right)
\geq \prod\limits_{i=0}^{2p-1}\left(1-\frac{1}{i^2}\right)> \frac{1}{3}.
\end{align*}
Apply Lemma \ref{l-poly} to $P(t),\ Q(t)$ and $C=3$, we have for any $l\leq 3m$,  
$$\sup\limits_{t\in \left[0,(17m)^{-17}\right]}\left|\left(\frac{P}{Q}\right)^{(l)}(t)\right|\leq (17m)^{17l}.$$
It follows that for $6n\geq (17 m)^{17}$ and $l\leq m$, there exists $\xi\in[0,\frac{1}{6n}]$ such that
\begin{align*}
    \left|\mathbb{E}_n(N_\Gamma)-\frac{1}{(6n)^{\eta(\Gamma)}}\sum\limits_{i=0}^{l-1}\frac{a_i(\Gamma)}{(6n)^i}\right|&=\frac{1}{(6n)^{\eta(\Gamma)}}\left|\frac{1}{l!(6n)^l}\left(\frac{P}{Q}\right)^{(l)}(\xi)\right|\\
    &\leq \frac{(17m)^{17l}}{l!(6n)^{l+\eta(\Gamma)}}.
\end{align*}
The proof is complete.
\end{proof}

\subsubsection{Surjective $X-$labeled graph homomorphisms}\label{subsec-surj}
Now we consider surjective $X-$labeled graph homomorphisms. Assume $\omega\in\textbf{F}_2$ is of form $$\omega=x_1x_2^{i_1}\cdots x_1x_2^{i_k}=e_1e_2\cdots e_l,$$ where $k\geq 1$, $i_j\in\{1,2\}\ (1\leq j\leq k)$, $2k\leq l\leq 3k$, and $e_j\in\{x_1,x_2\}\ (1\leq j\leq l)$. Recall that $\Gamma_X(\omega)$ is a closed path with its edges labeled as follows:
$$v_1\overset{e_l}{\to}v_2\overset{e_{l-1}}{\to}\cdots\overset{e_2}{\to} v_l\overset{e_1}{\to} v_1.$$ 
So for each $1\leq j\leq k$, $x_2^{i_{k+1-j}}$ corresponds to certain consecutive edges in $\Gamma_X(\omega)$ of form
$$w_1\overset{x_1}{\to}w_2\overset{x_2}{\to}w_3\overset{x_1}{\to}w_4\ \text{when }i_{k+1-j}=1,$$
or
$$w_1\overset{x_1}{\to}w_2\overset{x_2}{\to}w_3\overset{x_2}{\to}w_4\overset{x_1}{\to}w_5\ \text{when }i_{k+1-j}=2,$$
hence corresponds to an $x_2-$cycle $\Delta_j$ in $\overline{\Gamma_X(\omega)}$. For each $1\leq j\leq k$, there is an $x_1-$cycle connecting $\Delta_j$ and $\Delta_{j+1}$ (set $\Delta_{k+1}=\Delta_1$).

Now we relabel the vertices in $\overline{\Gamma_X(\omega)}$ as follows: $x_2^{i_{k+1-j}}\ (1\leq j\leq k)$ corresponds to an $x_2-$cycle $\Delta_j$ of $\overline{\Gamma_X(\omega)}$, label its vertices by $v_{2j-1},\ v_{2j}$ and $u_{j}$ so that 
\begin{enumerate}
    \item the two vertices $v_{2j}$ and $v_{2j+1}$ $(1\leq j\leq k)$ are joined by an $x_1-$cycle of $\overline{\Gamma_X(\omega)}$, (set $v_{2k+1}=v_1$);
    \item the vertex $u_j\ (1\leq j\leq k)$ is not contained in any $x_1-$cycle of $\overline{\Gamma_X(\omega)}$.
\end{enumerate}
One may see Figure \ref{fig:label} for an example of $\omega=x_1x_2^2x_1x_2x_1x_2$.
\begin{figure}[h]
\includegraphics[width=.8\textwidth]{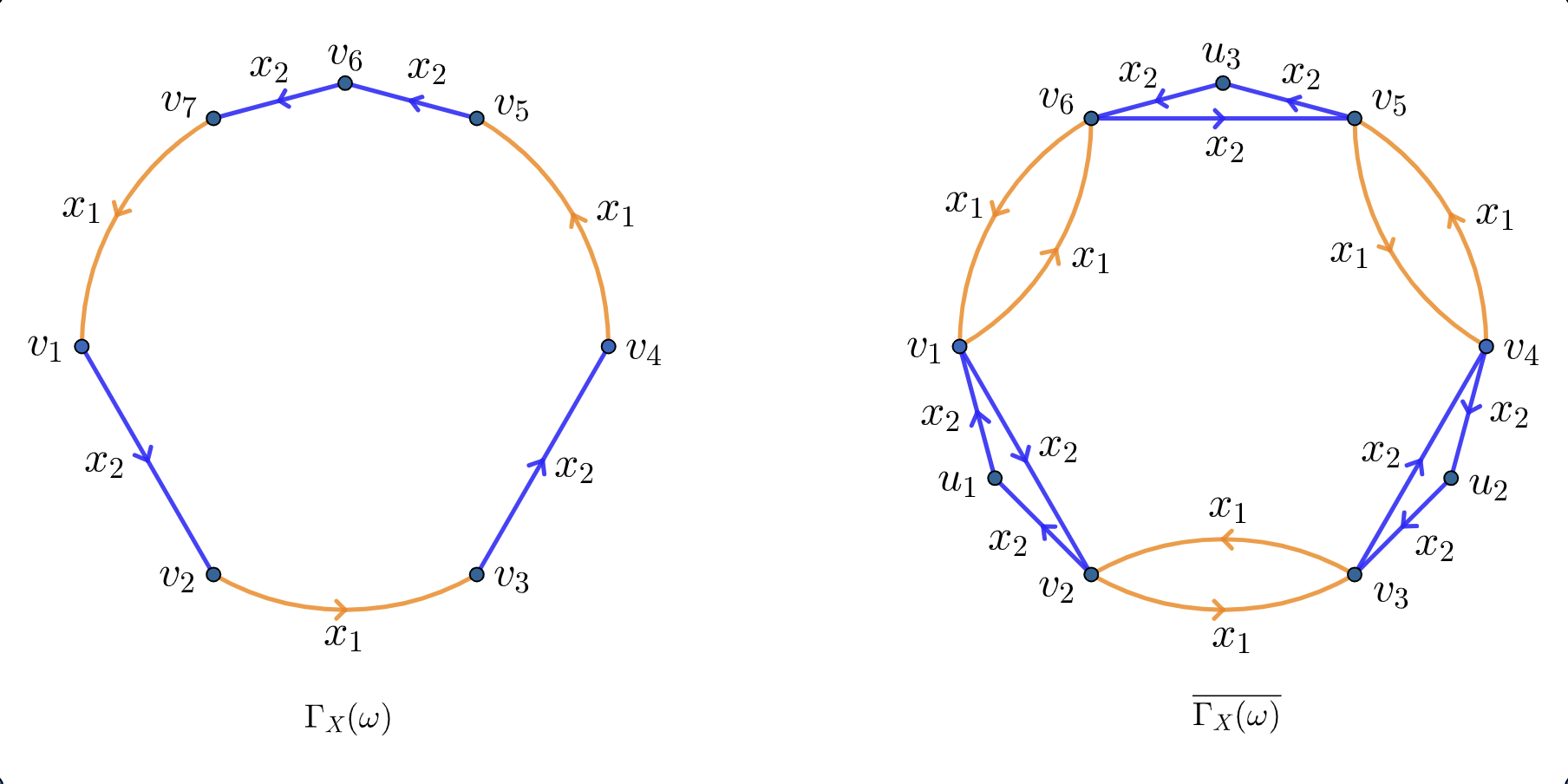}
{\caption{Relabel $\overline{\Gamma_X(\omega)}$ when $\omega=x_1x_2^2x_1x_2x_1x_2$}}
\label{fig:label}
\end{figure}

Now we give an algorithm to output all possible elements in $\mathcal{I}(\omega)$.
\begin{Algo*}[Algorithm for all outputs in $\mathcal{I}(\omega)$]
First, we have the initial data $\mathcal{A}_1=(\Gamma_1,G_1,f_1,a_1)$, where 
\begin{itemize}
    \item $\Gamma_1$ is the first $x_2-$cycle $\Delta_1$ in $\overline{\Gamma_X(\omega)}$;
    \item $G_1$ is an $x_2-$cycle $w_1\overset{x_2}{\to}w_2\overset{x_2}{\to}w_3\overset{x_2}{\to}w_1$;
    \item $f_1:\Gamma_1\to G_1$ is the unique $X-$labeled graph homomorphism such that $$f(v_1)=w_1;$$
    \item $a_1=0$.
\end{itemize}
Suppose that we have $\mathcal{A}_t=(\Gamma_t,G_t,f_t,a_t)\ (1\leq t\leq k-1)$ after the $(t-1)-$th step. In the $t-$th step, we construct $\mathcal{A}_{t+1}=(\Gamma_{t+1},G_{t+1},f_{t+1},a_{t+1})$ from $\mathcal{A}_t$ as follows. Let $\Gamma_{t+1}$ be the subgraph of $\overline{\Gamma_X(\omega)}$ which consists of $x_2-$cycles $\{\Delta_1,...,\Delta_{t+1}\}$ and $x_1-$cycles that join $\Delta_i$ and $\Delta_{i+1}$ for all $1\leq i\leq t$, i.e. $\Gamma_{t+1}$ is obtained by adding an $x_1-$cycle and an $x_2-$cycle to $\Gamma_t$. Then we construct $G_{t+1}$. After relabeling $\overline{\Gamma_X(\omega)}$, 
assume $$f_t(v_{2t})=w_j,$$ where $w_j$ is a vertex in $G_t$,
\begin{enumerate}
    \item if $\omega_j$ is not contained in any $x_1-$cycle of $G_t$, then
    pick one of the following two operations \emph{randomly}:
    \begin{itemize}
   \item Operation I: Take a new $x_2-$cycle $\Delta$
   $$w_{3s+1}\overset{x_2}{\to}w_{3s+2}\overset{x_2}{\to}w_{3s+3}\overset{x_2}{\to}w_{3s+1},$$
   where $|V(G_t)|=3s$. Let $G_{t+1}$ be the $X-$labeled graph which is obtained from $G_t$ by adding an $x_2-$cycle $\Delta$ and an $x_1-$cycle that joins $\omega_j$ and $\omega_{3s+1}$. Then $f_t$ could be uniquely extended to an $X-$labeled graph homomorphism $f_{t+1}:\Gamma_{t+1}\to G_{t+1}$. Set $a_{t+1}=a_t$.

   \item Operation II: Take a vertex $\omega_l\neq \omega_j\in V(G_t)$ such that $w_l$ is not contained in any $x_1-$cycle of $G_t$. Let $G_{t+1}$ be the $X-$labeled graph which is obtained from $G_t$ by adding an $x_1-$cycle that joins $w_j$ and $w_l$. Then $f_t$ could be uniquely extended to an $X-$labeled graph homomorphism $f_{t+1}:\Gamma_{t+1}\to G_{t+1}$. Set $a_{t+1}=a_t+1$.
   \end{itemize}
    \item  If $\omega_j$ is contained in an $x_1-$cycle of $G_t$, we then set $G_{t+1}=G_t$, and $f_t$ could be uniquely extended to an $X-$labeled graph homomorphism $$f_{t+1}:\Gamma_{t+1}\to G_{t+1}.$$ Set $a_{t+1}=a_t$.
\end{enumerate}

 \noindent After the $(k-1)-$th step, we obtain $\mathcal{A}_k=(\Gamma_{k},G_k,f_k,a_k)$. Now we finish the last step. Assume $$f_k(v_{2k})=w_r,$$ where $w_r$ is a vertex of $G_k$, there are three possible outputs as follows:
\begin{enumerate}
    \item If neither  $w_r$ nor $w_1$ is contained in any $x_1-$cycle of $G_k$, denote by $G_{k+1}$ the $X-$labeled graph obtained from $G_k$ by adding an $x_1-$cycle that joins $\omega_r$ and $\omega_1$ (i.e. we make Operation II). Then $f_k$ could be uniquely extended to a surjective $X-$labeled graph homomorphism $f_{k+1}:\overline{\Gamma_X(\omega)}\to G_{k+1}$. Set $a_{k+1}=a_k+1$.
    \item If $w_r$ and $w_1$ are contained in the same $x_1-$cycle of $G_k$, set $G_{k+1}=G_k$. Then $f_k$ could be uniquely extended to a surjective $X-$labeled graph homomorphism $f_{k+1}:\overline{\Gamma_X(\omega)}\to G_{k+1}$. Set $a_{k+1}=a_k$.
    \item For all remaining cases, set $\mathcal{A}_{k+1}=\emptyset.$
\end{enumerate}
Finally, we obtain $\mathcal{A}_{k+1}=\left(\overline{\Gamma_X(\omega)},G_{k+1},f_{k+1},a_{k+1}\right)$ or $\emptyset$.
\end{Algo*}

\begin{rem*}
It is clear that for any output $\mathcal{A}_{k+1}=\left(\overline{\Gamma_X(\omega)},G_{k+1},f_{k+1},a_{k+1}\right)$ of $\omega$, we have
$$a_{k+1}\leq k.$$
\end{rem*}

\begin{figure}[h]
\includegraphics[width=\textwidth]{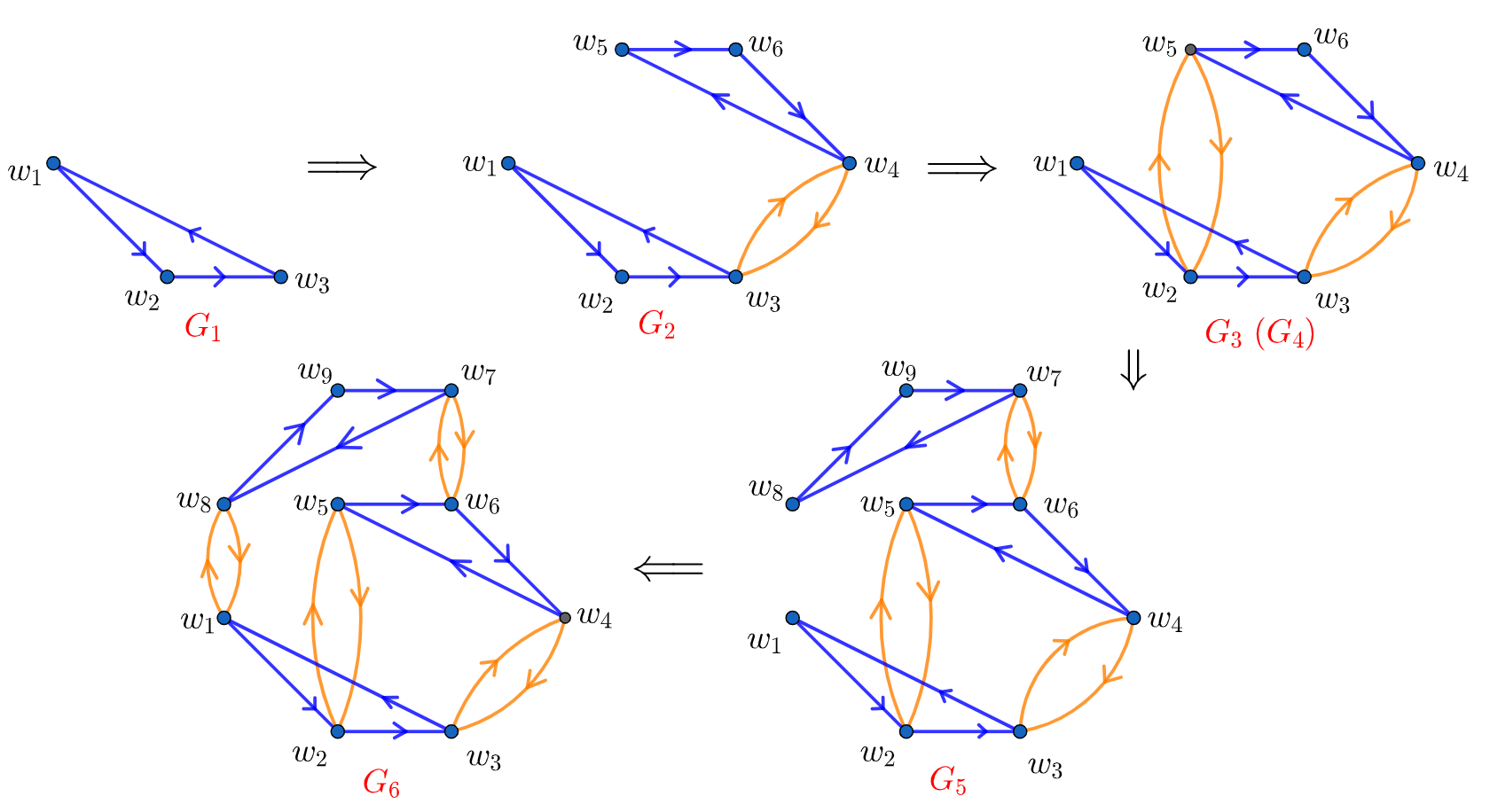}
{\caption{An output for $\omega=x_1x_2x_1x_2^2x_1x_2x_1x_2x_1x_2^2$}}
\label{fig:oper}
\end{figure}
\begin{exam*}
Now we consider an example for $\omega=x_1x_2x_1x_2^2x_1x_2x_1x_2x_1x_2^2$ containing all possible operations above (see Figure \ref{fig:oper}). Assume that the vertices of $\overline{\Gamma_X(\omega)}$ have been relabeled. First, we have 
$\mathcal{A}_1=(\Gamma_1,G_1,f_1,a_1=0)$, where the $X-$labeled graph homomorphism $f_1$ is uniquely determined by 
$$f_1(v_1)=w_1,\ f_1(v_2)=w_3,\ f_1(u_1)=w_2.$$

\noindent Step 1: $w_3$ is not contained in any $x_1-$cycle of $G_1$. Make Operation I, we obtain a graph $G_2$ from $G_1$ by adding a new $x_2-$cycle
$$w_{4}\overset{x_2}{\to}w_{5}\overset{x_2}{\to}w_{6}\overset{x_2}{\to}w_{4},$$
and joining $w_3$ and $w_4$ with an $x_1-$cycle. Set $a_2=0$. The $X-$labeled graph homomorphism $f_1$ could be uniquely extended to $f_2:\Gamma_2\to G_2$ such that
$$f_2(v_3)=w_4,\ f_2(v_4)=w_5,\ f_2(u_2)=w_6.$$

\noindent Step 2: $w_5$ is not contained in any $x_1-$cycle of $G_2$. Make Operation II, we obtain a graph $G_3$ from $G_2$ by joining $w_5$ and $w_2$ with an $x_1-$cycle (we remark here that this is not the unique way, e.g. one may add an $x_1-$cycle between $w_5$ and $w_1$). Set $a_3=1$. The $X-$labeled graph homomorphism $f_2$ could be uniquely extended to $f_3:\Gamma_3\to G_3$ such that 
$$f_3(v_5)=w_2,\ f_3(v_6)=w_3,\ f_3(u_3)=w_1.$$

\noindent Step 3: $w_3$ is contained in an $x_1-$cycle of $G_3$. Set $G_4=G_3$ and $a_4=1$. The $X-$labeled graph homomorphism $f_3$ could be uniquely extended to $f_4:\Gamma_4\to G_4$ such that
$$f_4(v_7)=w_4,\ f_4(v_8)=w_6,\ f_4(u_4)=w_5.$$

\noindent Step 4: $w_6$ is not contained in any $x_1-$cycle of $G_4$. Make Operation I, we obtain a graph $G_5$ from $G_4$ by adding a new $x_2-$cycle
$$w_{7}\overset{x_2}{\to}w_{8}\overset{x_2}{\to}w_{9}\overset{x_2}{\to}w_{7},$$
 and joining $w_6$ and $w_7$ with an $x_1-$cycle. Set $a_5=1$. The $X-$labeled graph homomorphism $f_4$ could be uniquely extended to $f_5:\Gamma_5\to G_5$ such that
$$f_5(v_9)=w_7,\ f_5(v_{10})=w_8,\ f_5(u_5)=w_9.$$

\noindent Step 5: Neither $w_8$ nor $w_1$ is contained in any $x_1-$cycle of $G_5$. We obtain a graph $G_6$ from $G_5$ by joining $w_8$ and $w_1$ with an $x_1-$cycle. Set $a_6=2$.
 The $X-$labeled graph homomorphism $f_5$ could be uniquely extended to
 $$f_6:\overline{\Gamma_X(\omega)}\to G_6.$$

\noindent Finally, we get an output $$\mathcal{A}_6=\left(\overline{\Gamma_X(\omega)},G_6,f_6,2\right).$$
\end{exam*}

Recall that $\omega\in\textbf{F}_2$ could be represented as
$$\omega=x_1x_2^{i_1}\cdots x_1x_2^{i_k},$$
where $k\geq 1$ and $i_j\in\{1,2\}\ (1\leq j\leq k)$.
Denote by
$$\mathfrak{D}(\omega)=\left\{\text{all possible output data $\mathcal{A}=\left(\overline{\Gamma_X(\omega)},G,f,a\right)$ for }\omega\right\}.$$
For any $\mathcal{A}=\left(\overline{\Gamma_X(\omega)},G,f,a\right)$ and $\mathcal{A}^\prime=\left(\overline{\Gamma_X(\omega)},G^\prime,f^\prime,a^\prime\right)$ in $\mathfrak{D}(\omega)$, we say $\mathcal{A}$ and $\mathcal{A}^\prime$ are \emph{equivalent} if $(f,G)$ and $(f^\prime,G^\prime)$ are equivalent in $F(\omega)$ given by \eqref{def-Fo}. Denote by
$$\mathfrak{F}(\omega)=\text{ the set of all equivalent classes in }\mathfrak{D}(\omega)$$
and
$$\mathfrak{F}_h(\omega)=\left\{\mathcal{A}=\left(\overline{\Gamma_X(\omega)},G,f,a\right)\in\mathfrak{F}(\omega);\ a=h\right\}\text{ for }h\geq 1.$$

Then we have the following two lemmas.
\begin{lemma}\label{l-eta}
Assume $h\geq 1$ and $\omega\in\textbf{F}_2$ is of form
$$\omega=x_1x_2^{i_1}\cdots x_1x_2^{i_k},$$
where $k\geq 1$ and $i_j\in\{1,2\}\ (1\leq j\leq k)$, then we have
\begin{enumerate}
    \item  $|\mathfrak{F}_h(\omega)|\leq k^{2h-1}$. \vspace{5pt}
    \item Assume $\mathcal{A}=\left(\overline{\Gamma_X(\omega)},G,f,h\right)\in\mathfrak{F}_h(\omega)$, then $$\eta(G)= h-1.$$
\end{enumerate}
\end{lemma}
\begin{proof}
     Recall from the algorithm above the initial data $\mathcal{A}_1=(\Gamma_1,G_1,f_1,a_1)$, and after the $t-$th step ($1\leq t\leq k$), we obtain $$\mathcal{A}_{t+1}=(\Gamma_{t+1},G_{t+1},f_{t+1},a_{t+1}).$$
     Consider the processes that produce elements in $\mathfrak{F}_h(\omega)$. During such processes, we make Operation II exactly $h$ times. For any $1\leq t\leq k-1$, 
     \begin{enumerate}
         \item if at the $t-$th step, we make Operation II, then there are at most $k$ different choices for constructing a new graph $G_{t+1}$ from $G_{t}$;
         \item otherwise, the new graph $G_{t+1}$ is uniquely determined by $G_t$.
     \end{enumerate}
      Moreover, there exists some $1\leq t_0\leq k$ such that at the $t_0-$th step, we make Operation II and join $w_1$ with some $w_i$ by an $x_1-$cycle. Since there are at most $\binom{k}{h}$ different choices to determine the remaining $h$ steps at which we make Operation II, it follows that
     \begin{align*}
    |\mathfrak{F}_h(\omega)|\leq \binom{k}{h}\times k^{h-1}\leq k^{2h-1},
     \end{align*}
     which yields the desired bound in Part $(1)$.
     
     Now we prove Part $(2)$.  Consider a process that outputs $\mathcal{A}$, i.e. $\mathcal{A}=\mathcal{A}_{k+1}$. Also, assume that the graph $G_{t}$ has $p_t$ $x_1-$cycles and $q_t$ $x_2-$cycles for $1\leq t \leq (k+1)$. Denote by $$\eta_t=p_t-q_t.$$ By the definition of our algorithm, it is clear that
    $$p_1=0,\ q_1=1\text{ and }\eta_1=-1.$$
    For any $1\leq t\leq k-1$,   
    \begin{enumerate}
    \item if at the $t-$th step, we make Operation I, then 
    $$p_{t+1}=p_t+1,\ q_{t+1}=q_t+1\text{ and }\eta_{t+1}=\eta_t;$$
    \item if at the $t-$th step, we make Operation II, then 
    $$p_{t+1}=p_t+1,\ q_{t+1}=q_t\text{ and }\eta_{t+1}=\eta_t+1;$$
    \item for the remaining case, we have
    $$ p_{t+1}=p_t,\ q_{t+1}=q_t\text{ and }\eta_{t+1}=\eta_t.$$
\end{enumerate}
Consider the $k-$th step, i.e. the final step, 
\begin{enumerate}
\item[(a)] if neither $w_j$ nor $w_1$ is contained in any $x_1-$cycle of $G_k$, then we make Operation II and
$$ p_{k+1}=p_k+1,\ q_{k+1}=q_k\text{ and }\eta_{k+1}=\eta_k+1;$$
\item[(b)] if $w_j$ and $w_1$ are contained in the same $x_1-$cycle of $G_k$, then  
$$p_{k+1}=p_k,\ q_{k+1}=q_k\text{ and }\eta_{k+1}=\eta_k.$$
\end{enumerate}
Since we make Operation II exactly $h$ times, it follows that $\eta(G)=\eta_{k+1}=h-1$.

The proof is complete.
\end{proof}
\begin{lemma}\label{l-zeta}
   Assume $\omega\in\textbf{F}_2$ is of form
$$\omega=x_1x_2^{i_1}\cdots x_1x_2^{i_k},$$
where $k\geq 1$ and $i_j\in\{1,2\}\ (1\leq j\leq k)$. Assume $\mathcal{A}=\left(\overline{\Gamma_X(\omega)},G,f,1\right)\in\mathfrak{F}_1(\omega)$, then $$G=\overline{\Gamma_X(\nu)},$$ where $\nu\in\textbf{F}_2$ is a reduced word such that $\omega=\nu^q$
    for some $q\geq 1$.

\end{lemma}
\begin{proof}
     Consider a process that outputs $\mathcal{A}$, i.e. $\mathcal{A}=\mathcal{A}_{k+1}$. 
     We assume after the $t-$th $(0\leq t\leq k)$ step, we obtain $$\mathcal{A}_{t+1}=(\Gamma_{t+1},G_{t+1},f_{t+1},a_{t+1}),$$
    and $f_{t+1}(v_{2t+2})=w_{j_{t+1}}\in V(G_{t+1})$.
     Since $$\mathcal{A}=\left(\overline{\Gamma_X(\omega)},G,f,1\right)\in\mathfrak{F}_1(\omega),$$ we only make one Operation II at some $s-$th step where $1\leq s\leq k$. Now we claim that \emph{at the $s-$th step, $w_{j_s}$ and $w_1$ is joined by an $x_1-$cycle}. Otherwise, we find that $w_{j_s}$ and $w_l$ is joined by an $x_1-$cycle for some $l\neq 1$, then $w_1$ is not contained in any $x_1-$cycle of $G_k$, and hence we have to join $w_1$ and some vertex in $G_k$ with an $x_1-$cycle at the $k-$th step. It follows that we make Operation II at least twice, which leads to a contradiction. Then we have 
    \begin{enumerate}
        \item for $1\leq t\leq s-1$, we make Operation I at each $t-$th step (i.e $G_{t+1}$ is obtained from $G_t$ by adding a new $x_1-$cycle and a new $x_2-$cycle);
        \item for $s+1\leq t\leq k$, $w_{j_{t}}$ is contained in an $x_1-$cycle of $G_{t}$. Hence $G_t=G_{t+1}$.
    \end{enumerate}
    Take $$\nu=x_1x_2^{i_{k+1-s}}\cdots x_1x_2^{i_k},$$ then from conditions (i) and (ii) above, we have $$G=G_{s+1}=\overline{\Gamma_X(\nu)} \textnormal{ \ and \ } \omega=\nu^q$$ for some $q\geq 1$. The proof is complete.
\end{proof}
\subsubsection{Number of fixed points}
In this subsection, we calculate the expected value $\mathbb{E}_n[\textnormal{fix}_\omega]$ for any fixed $\omega\in\textbf{F}_2$. We prove two propositions.
\begin{proposition}\label{p-poly-free}
     Assume $\omega\in\textbf{F}_2$ is a fixed word of form
     $$\omega=x_1x_2^{i_1}\cdots x_1x_2^{i_k},$$
where $k\geq 1$ and $i_j\in\{1,2\}\ (1\leq j\leq k)$.
    Then there exist a sequence $\{v_i(\omega)\}_{i\geq 0}\subset \mathbb{R}$  such that 
    $$\left|\mathbb{E}_n\left[\textnormal{fix}_\omega\right]-\sum\limits_{i=0}^{m-1}\frac{v_i(\omega)}{(6n)^i}\right|\leq\frac{(18m)^{18m}}{(6n)^m}$$
    for all $m\geq 2k$ and $6n\geq 18 m^{18}$.
\end{proposition}
\begin{proof}
      Assume $m\geq 2k$ and $6n\geq 17m^{17}$. From \eqref{e-f} we know that
     \[\mathbb{E}_n\left[\textnormal{fix}_{\omega}\right]=\sum\limits_{(g,\Gamma)\in\mathcal{I}(\omega)}\mathbb{E}_n(N_{\Gamma}).\]
     Then we have
     \begin{equation}
        \begin{aligned}\label{e-tr-1}
            &\ \ \ \ \left|\mathbb{E}_n\left[\textnormal{fix}_{\omega}\right]-\sum\limits_{(g,\Gamma)\in\mathcal{I}(\omega)}\frac{1}{(6n)^{\eta(\Gamma)}}\sum\limits_{i=0}^{m-1-\eta(\Gamma)}\frac{a_i(\Gamma)}{(6n)^i}\right|\\
            &=\left|\sum\limits_{(g,\Gamma)\in\mathcal{I}(\omega)}\mathbb{E}_n(N_{\Gamma})-\sum\limits_{(g,\Gamma)\in\mathcal{I}(\omega)}\frac{1}{(6n)^{\eta(\Gamma)}}\sum\limits_{i=0}^{m-1-\eta(\Gamma)}\frac{a_i(\Gamma)}{(6n)^i}\right| \\
    &\leq\sum\limits_{(g,\Gamma)\in\mathcal{I}(\omega)}\left|\mathbb{E}_n(N_{\Gamma})-\frac{1}{(6n)^{\eta(\Gamma)}}\sum\limits_{i=0}^{m-1-\eta(\Gamma)}\frac{a_i(\Gamma)}{(6n)^i}\right|,
        \end{aligned}
        \end{equation}
        where in the summations above, we set
        $$\sum\limits_{i=0}^{m-1-\eta(\Gamma)}\frac{a_i(\Gamma)}{(6n)^i}=0$$
        for the case $\eta(\Gamma)\geq m$. Now applying Proposition \ref{l-poly-1} to $(g,\Gamma)\in\mathcal{I}(\omega)$ and $l=m-\eta(\Gamma)$, we have
        \begin{align}\label{e-tr-2}
            \left|\mathbb{E}_n(N_{\Gamma})-\frac{1}{(6n)^{\eta(\Gamma)}}\sum\limits_{i=0}^{m-1-\eta(\Gamma)}\frac{a_i(\Gamma)}{(6n)^i}\right|\leq \frac{(17m)^{17(m-\eta(\Gamma))}}{(6n)^{m}}.
        \end{align}
        From Lemma \ref{l-eta} we have for $h\geq 0$,
        $$\#\{(g,\Gamma)\in\mathcal{I}(\omega);\ \eta(\Gamma)=h\}=|\mathfrak{F}_{h+1}(\omega)|\leq k^{2h+1}.$$
        It then follows that for $m\geq 2k$ and $6n\geq 18m^{18}$,
        \begin{equation}
        \begin{aligned}\label{e-tr-3}
            \sum\limits_{(g,\Gamma)\in\mathcal{I}(\omega)}\frac{(17m)^{17(m-\eta(\Gamma))}}{(6n)^{m}}&=\sum\limits_{h=0}^\infty\sum\limits\limits_{\substack{(g,\Gamma)\in\mathcal{I}(\omega)\\ \eta(\Gamma)=h}}\frac{(17m)^{17(m-h)}}{(6n)^{m}}\\
            &\leq\sum\limits_{h=0}^\infty\frac{(17m)^{17(m-h)}}{(6n)^{m}}\times k^{2h+1}\\
            &\leq \frac{(18m)^{18m}}{(6n)^{m}}.
        \end{aligned}
        \end{equation}
        Take  
        \begin{equation}\label{vje}
             v_j(\omega)=\sum\limits_{(g,\Gamma)\in\mathcal{I}(\omega)}a_{j-\eta(\Gamma)}(\Gamma)\text{ for } j\geq 0,
        \end{equation}
        where in the summation above, $a_i(\Gamma)=0$ for $i<0$. Together with \eqref{e-tr-1}, \eqref{e-tr-2} and \eqref{e-tr-3}, one may conclude that  $\{v_j(\gamma)\}_{j\geq 0}$ is as desired. The proof is complete.
\end{proof}

\begin{proposition}\label{p-coeff}
    Assume that $\omega\in\textbf{F}_2$ is of form $$\omega=x_1x_2^{i_1}x_1x_2^{i_2}...x_1x_2^{i_k}$$
        for some $k\geq 1$ and $i_j\in\{1,2\}$.
     Also assume $\omega=\omega_0^\mu\in\textbf{F}_2$, where $\mu\geq 1$ and $\omega_0$ is a primitive word. Then the coefficient $v_0(\omega)$ defined in Proposition \ref{p-poly-free} satisfies
     $$v_0(\omega)=d(\mu)\overset{\textnormal{def}}{=}\#\{\text{divisors of }\mu\}.$$
\end{proposition}
\begin{proof}
First, it follows from \eqref{vje} that
\[v_{0}(\omega)=\sum\limits_{\substack{(g,\Gamma)\in\mathcal{I}(\omega)\\\eta(\Gamma)=0}}a_0(\Gamma).\]
From Part $(1)$ of Proposition \ref{l-poly-1}, we know that $a_0(\Gamma)=1$. Then from Part $(2)$ of Lemma \ref{l-eta}, we have
    \begin{align*}
        v_{0}(\omega)=|\mathfrak{F}_1(\omega)|.
    \end{align*}
    Assume $\mathcal{A}=\left(\overline{\Gamma_X(\omega)},G,f,1\right)\in\mathfrak{F}_1(\omega)$. 
    From Lemma \ref{l-zeta}, there exists a reduced word $\nu\in\textbf{F}_2$ and $q\geq 1$ such that
     $$G=\overline{\Gamma_X(\nu)}\text{ and }\omega=\nu^q.$$
     Since $\omega=\omega_0^\mu$ and $\omega_0$ is primitive, it follows that
     $$q\text{ is a divisor of $\mu$ and }\nu=\omega_0^{\frac{\mu}{q}}.$$ 
     Hence $$v_0(\omega)=|\mathfrak{F}_1(\omega)|\leq d(\mu).$$
On the other hand, if $q$ is a divisor of $\mu$, then there exists a natural surjective $X-$labeled graph homomorphism $$f_q:\overline{\Gamma_X(\omega)}\to\overline{\Gamma_X(\omega_0^q)}.$$
The element $\mathcal{A}(q)=\left(\overline{\Gamma_X(\omega)},\overline{\Gamma_X(\omega_0^q)},f_q,1\right)\in\mathfrak{F}_1(\omega)$.
It is obvious that for any two divisors $q_1\neq q_2$ of $\mu$, $\mathcal{A}(q_1)$ and $\mathcal{A}(q_2)$ are not equivalent; hence they are two different elements in $\mathfrak{F}_1(\omega)$. It follows that $$v_0(\omega)=|\mathfrak{F}_1(\omega)|\geq d(\mu).$$ The proof is complete.
\end{proof}
\begin{rem*}
    We thank Doron Puder for informing us that he and Zimhoni, in their work \cite{PZ24}, studied a more general model involving two independent random permutations $\sigma$ and $\tau$, where $\sigma^2=\textnormal{id}$ and $\tau^3=\textnormal{id}$, where $\tau$ is allowed to have nonempty fixed points. Although the overall behavior of these two models differs significantly, certain similarities remain. Interested readers may refer to \cite{PZ24} for more details. 
    \end{rem*}

\section{Strong convergence property of $\textnormal{PSL}(2,\mathbb{Z})$}\label{s-scp}
      In this section, we first prove two propositions that are analogous to Assumptions 1.3 and 1.4 in \cite{MPvH25}. Based on such two propositions and the argument in \cite{MPvH25}, we will finish the proof of Theorem \ref{mt-1} by first proving the strong convergence property of $\textnormal{PSL}(2,\mathbb{Z})$. For strong convergence, one may refer to e.g. \cite{HT05, BC19, CGVTvH24, CGVvH24, MdlS24, LM25} and a recent survey \cite{Hand25} for more related topics.
    \subsection{Group and $C^\star-$algebra}
    In this subsection, we recall certain related notation of the $C^{\star}-$algebra that may be referred to \cite{MPvH25} for more details.
    Consider the group algebra
    $$\mathbb{C}\left[\textnormal{PSL}(2,\mathbb{Z})\right]=\left\{\sum\limits_{\gamma\in\textnormal{PSL}(2,\mathbb{Z})}a_{\gamma}\gamma;\ a_\gamma\in\mathbb{C},\ a_{\gamma}\neq 0\text{ for only finitely many }\gamma's\right\}.$$
    There is a natural involution on the group algebra $\mathbb{C}[\textnormal{PSL}(2,\mathbb{Z})]$, i.e. 
    $$\alpha^{\star}=\sum\limits_{\gamma\in\textnormal{PSL}(2,\mathbb{Z})}\overline{a_{\gamma}}\gamma^{-1} \text{ for any } \alpha=\sum\limits_{\gamma\in\textnormal{PSL}(2,\mathbb{Z})}a_\gamma \gamma.$$
    An element $\alpha$ is called \emph{self-adjoint} if $\alpha=\alpha^\star$. Recall that the group $\textnormal{PSL}(2,\mathbb{Z})$ has a unitary representation called the regular representation:
    $$\lambda:\textnormal{PSL}(2,\mathbb{Z})\to\mathcal{U}(\ell^2(\textnormal{PSL}(2,\mathbb{Z}))),\ \lambda(\gamma)[f](x)=f(\gamma^{-1}x)$$
    for any $f\in\ell^2\left(\textnormal{PSL}(2,\mathbb{Z})\right)$ and $x\in \textnormal{PSL}(2,\mathbb{Z})$.
    The regular representation could also be linearly extended onto $\mathbb{C}\left[\textnormal{PSL}(2,\mathbb{Z})\right]$, i.e. for $\alpha=\sum\limits_{\gamma\in\textnormal{PSL}(2,\mathbb{Z})}a_\gamma \gamma$, 
    $$\lambda(\alpha)[f](x)=\sum\limits_{\gamma\in\textnormal{PSL}(2,\mathbb{Z})}a_\gamma f(\gamma^{-1}x).$$
    The reduced $C^\star-$algebra $C^\star_{\textnormal{red}}[\textnormal{PSL}(2,\mathbb{Z})]$ is the norm-closure of $$\{\lambda(\alpha);\alpha\in\mathbb{C}[\textnormal{PSL}(2,\mathbb{Z})]\}.$$
    Denote by $\rho$ the canonical trace on $C^\star_{\textnormal{red}}[\textnormal{PSL}(2,\mathbb{Z})]$, i.e. $$\rho(\varphi)=\langle\delta_e,\varphi\delta_e\rangle\text{ for } \varphi\in C^\star_{\textnormal{red}}[\textnormal{PSL}(2,\mathbb{Z})],$$ where $e$ is the unitary element in $\textnormal{PSL}(2,\mathbb{Z})$, and the function $\delta_\gamma\in\ell^2(\textnormal{PSL}(2,\mathbb{Z}))\ (\gamma\in\textnormal{PSL}(2,\mathbb{Z}))$ is defined by
        $$\delta_\gamma(x)=\begin{cases}1&\text{ if }x=\gamma;\\ 0 &\text{otherwise}.\end{cases}$$
    \subsection{Estimate of traces}
    Recall that for any $\gamma\in\textnormal{PSL}(2,\mathbb{Z})$, it could be represented as a word with only letters $b$ and $c$, where 
    $$b=\begin{pmatrix}
        0 & 1\\ -1 &0
    \end{pmatrix}\text{ and }c=\begin{pmatrix}
        0 & 1 \\ -1 &1
    \end{pmatrix}.$$
    Denote by $|\gamma|$ the minimal length for $\gamma$ to be represented as a word with only letters $b$, $c$ (without $c^{-1}$).
    For $n\geq 1$ and  $(\sigma,\tau)\in\mathcal{E}_n$, $$\Phi(\sigma,\tau):\textnormal{PSL}(2,\mathbb{Z})\to S_{6n}$$ is the group homomorphism that satisfies
    $$\Phi(\sigma,\tau)(b)=\sigma\text{ and }\Phi(\sigma,\tau)(c)=\tau.$$
    For $\gamma\in\textnormal{PSL}(2,\mathbb{Z})$, define
    $$\textnormal{fix}_\gamma(\sigma,\tau)\overset{\textnormal{def}}{=}\#\left\{\begin{matrix}
        \text{ fixed points of }\\
        \text{permutation }\Phi(\sigma,\tau)(\gamma)
    \end{matrix}\right\}.$$
    Then $$\textnormal{fix}_\gamma:\mathcal{E}_n\to\mathbb{Z}^{\geq 0}$$ is a random variable. For any $n\geq 1$, denote by $\mathcal{U}(n)$ the group of $n\times n$ complex unitary matrices. The $6n-$dimensional linear representation of $S_{6n}$ by permutation matrices has a $(6n-1)-$dimensional irreducible subrepresentation 
    $$\textnormal{std}:S_{6n}\to\mathcal{U}(6n-1)$$
    obtained by removing non-zero invariant vectors. Assume
    $$\Pi(\sigma,\tau)=\textnormal{std}\circ\Phi(\sigma,\tau):\textnormal{PSL}(2,\mathbb{Z})\to\mathcal{U}(6n-1).$$
    For any $(\sigma,\tau)\in\mathcal{E}_n$ and $\gamma\in\textnormal{PSL}(2,\mathbb{Z})$, define $$\textnormal{Tr }\Pi(\sigma,\tau)(\gamma)\overset{\textnormal{def}}{=}\text{trace of unitary matrix }\Pi(\sigma,\tau)(\gamma)$$ and
    $$\textnormal{tr }\Pi(\sigma,\tau)(\gamma)\overset{\textnormal{def}}{=}\frac{1}{6n}\textnormal{Tr }\Pi(\sigma,\tau)(\gamma).$$
     Let $\textnormal{tr }\Pi_\gamma:\mathcal{E}_n\to\mathbb{R}$ be the random variable that satisfies 
    $$\textnormal{tr }\Pi_\gamma(\sigma,\tau)=\textnormal{tr }\Pi(\sigma,\tau)(\gamma).$$

    The following proposition is analogous to \cite[Assumption 1.3]{MPvH25}.
    
    \begin{proposition}\label{p-count}
    There exists a sequence $\{u_i(\gamma)\}_{i\geq 0}$ of real numbers for all $\gamma\in\textnormal{PSL}(2,\mathbb{Z})$ such that $u_0(\gamma)=\rho(\lambda(\gamma))$ and
    $$\left|\mathbb{E}_n\left[\textnormal{tr }\Pi_\gamma\right]-\sum\limits_{i=0}^{m-1}\frac{u_i(\gamma)}{(6n)^i}\right|\leq\frac{(18m)^{18m}}{(6n)^m}$$
    for all $m\geq |\gamma|+1$ and $6n\geq 18m^{18}$.
    \end{proposition}
    \begin{proof}
     Recall that the matrices 
     $$b=\begin{pmatrix}
         0 & 1 \\ -1 & 0
     \end{pmatrix} \text{ and }c=\begin{pmatrix}
         0 & 1 \\ -1 & 1
     \end{pmatrix}$$
     are generators of $\textnormal{PSL}(2,\mathbb{Z})$.
        For any $\gamma\in\textnormal{PSL}(2,\mathbb{Z})$, it satisfies one of the following three conditions:
        \begin{enumerate}
            \item $\gamma$ is the unitary element $e$;
            \item $\gamma$ is conjugate to one of $\{b,c,c^2\}$;
            \item $\gamma$ is conjugate to the element
        $$\gamma_0=bc^{i_1}\cdots bc^{i_k},$$
        where $k\geq 1$ and $i_j\in\{1,2\}\ (1\leq j\leq k)$.
        \end{enumerate}
         Now we prove it case by case.
        For the first case, i.e. $\gamma=e$, we have for any $(\sigma,\tau)\in\mathcal{E}_n$,
        \begin{align*}
        \textnormal{tr }\Pi_e(\sigma,\tau)&=\frac{1}{6n}\textnormal{Tr }\Pi(\sigma,\tau)(e)=\frac{1}{6n}(\textnormal{fix}_e(\sigma,\tau)-1)=1-\frac{1}{6n}.
        \end{align*}
        Hence,
        $$\mathbb{E}_n\left[\textnormal{tr }\Pi_e\right]=1-\frac{1}{6n},$$ it suffices to take $$u_0(e)=\rho(\lambda(e))=1,\ u_1(e)=-1\text{ and }u_i(e)=0\ (i\geq 2).$$
         For the second case, we have for any $(\sigma,\tau)\in\mathcal{E}_n$,
        \begin{align*}
        \textnormal{tr }\Pi_\gamma(\sigma,\tau)&=\frac{1}{6n}\textnormal{Tr }\Pi(\sigma,\tau)(\gamma)=\frac{1}{6n}(\textnormal{fix}_\gamma(\sigma,\tau)-1)=-\frac{1}{6n}.
        \end{align*}
        Hence,
        $$\mathbb{E}_n\left[\textnormal{tr }\Pi_\gamma\right]=-\frac{1}{6n},$$
        it suffices to take 
        $$u_0(\gamma)=\rho(\lambda(\gamma))=0,\ u_1(\gamma)=-1\text{ and }u_i(\gamma)=0\ (i\geq 2).$$
        For the third case, we have $|\gamma|\geq 2k$. Denote by 
        $$\omega=x_1x_2^{i_1}\cdots x_1x_2^{i_k}\in\textbf{F}_2.$$
        Since for any $(\sigma,\tau)\in\mathcal{E}_n$,
        $$\tilde{\Phi}(\sigma,\tau)(\omega)=\Phi(\sigma,\tau)(\gamma_0)=\sigma\tau^{i_1}\cdots \sigma\tau^{i_k},$$  it follows that
        \begin{equation}
        \begin{aligned}\label{e-pp-1}
            \mathbb{E}_n\left[\textnormal{tr } \Pi_\gamma\right]&=\mathbb{E}_n\left[\textnormal{tr } \Pi_{\gamma_0}\right]\\
            &=\frac{1}{6n}\left(\mathbb{E}_n\left[\textnormal{fix}_{\gamma_0}\right]-1\right)=\frac{1}{6n}\left(\mathbb{E}_n\left[\textnormal{fix}_{\omega}\right]-1\right).
           \end{aligned}
           \end{equation}
            From \eqref{e-pp-1} and Proposition \ref{p-poly-free}, we have for any $(m-1)\geq |\gamma|\geq 2k$ and $6n\geq 18m^{18}$,
           \begin{align*}
               &\ \ \ \ \left|\mathbb{E}_n\left[\textnormal{tr }\Pi_\gamma\right]-\frac{1}{6n}\left(\sum\limits_{i=0}^{m-2}\frac{v_i(\omega)}{(6n)^i}-1\right)\right|\\
               &=\frac{1}{6n}\left|\mathbb{E}_n[\textnormal{fix}_{\omega}]-\sum\limits_{i=0}^{m-2}\frac{v_i(\omega)}{(6n)^i}\right|\leq\frac{(18m)^{18m}}{(6n)^m}.
           \end{align*}
         It suffices to take $u_0(\gamma)=\rho(\lambda(\gamma))=0$, $u_1(\gamma)=v_0(\omega)-1$ and 
         $$u_j(\gamma)=v_{j-1}(\omega)\ (j\geq 2).$$
         The proof is complete.
    \end{proof} 
    
    We have the following lemma.
    \begin{lemma}
       Let $$u_1(\cdot):\textnormal{PSL}(2,\mathbb{Z})\to\mathbb{R}$$ be the function defined in Proposition \ref{p-count}. Then we have
    \begin{enumerate}
        \item if $\gamma$ is the unitary element or it is conjugate to one  of $\{b,c,c^2\}$, then $$u_1(\gamma)=-1.$$
        \item If $\gamma$ is conjugate to $\gamma_0^\mu\ (\mu\geq 1)$, where
        $$\gamma_0=bc^{i_1}\cdots bc^{i_k}$$
        for some $k\geq 1$, $i_j\in\{1,2\}\ (1\leq j\leq k)$ and $\gamma_0$ is primitive.
        Then $$u_1(\gamma)=d(\mu)-1,$$
        where $d(\mu)$ is the number of divisors of $\mu$.  
    \end{enumerate}
    \end{lemma}
    \begin{proof}
        Part $(1)$ is contained in the proof of Proposition \ref{p-count}. For Part $(2)$, denote by 
        $$\omega_0=x_1x_2^{i_1}\cdots x_1x_2^{i_k}\text{ and }\omega=\omega_0^\mu.$$
        By the definition of $u_1$ in Proposition \ref{p-count}, $u_1(\gamma)=v_0(\omega)-1$. It then follows from Proposition \ref{p-coeff} that
        $$u_1(\gamma)=v_0(\omega)-1=d(\mu)-1.$$
        The proof is complete.
    \end{proof}
    
    \subsection{Estimate of operator norms}
     The function $$u_1:\textnormal{PSL}(2,\mathbb{Z})\to\mathbb{C}$$ defined in Proposition \ref{p-count} could be linearly extended to the group algebra $\mathbb{C}[\textnormal{PSL}(2,\mathbb{Z})]$.
     Then we have the following proposition, which is analogous to \cite[Assumption 1.4]{MPvH25}.
    \begin{proposition}\label{p-bound}
        For every self-adjoint $\alpha\in\mathbb{C}[\textnormal{PSL}(2,\mathbb{Z})]$, we have
        $$\limsup\limits_{p\to\infty}\left|u_1(\alpha^p)\right|^{\frac{1}{p}}\leq ||\lambda(\alpha)||,$$
        where the right hand side is the operator norm.
    \end{proposition}
    Before showing Proposition \ref{p-bound}, we prove the following pure algebraic fact for $\textnormal{PSL}(2,\mathbb{Z})$, which will be applied to prove Proposition \ref{p-bound}.
    Assume that $\mathcal{S}\subset\textnormal{PSL}(2,\mathbb{Z})$ is a finite subset and $\omega$ is a not-necessary-reduced word that could be represented as $$\omega=\omega_1\omega_2\cdots\omega_p,$$ where $\omega_i\in\mathcal{S}$. For any $1\leq s\leq t\leq p$, denote by
    $$\omega_{[s,t]}=\omega_s\cdots \omega_{t}\text{ and }\omega_{[s,t)}=\omega_s\cdots\omega_{t-1},$$
    if $s=t$, set $\omega_{[s,t)}$ as the unitary element in $\textnormal{PSL}(2,\mathbb{Z})$.
    
    \begin{lemma}\label{l-decom}
        Fix a finite subset $\mathcal{S}\subset\textnormal{PSL}(2,\mathbb{Z})$ such that $\max\limits_{\gamma\in\mathcal{S}}|\gamma|\leq d$, and let $\omega$ be any not-necessary-reduced word of length $p$ with letters in $\mathcal{S}$, 
        \begin{enumerate}
            \item if $\omega$ is conjugate to one of $\{b,c,c^2\}$ as an element in $\textnormal{PSL}(2,\mathbb{Z})$, then there exist $1\leq t_1\leq t_2\leq p$,  elements $\xi,v_1,v_2\in\textnormal{PSL}(2,\mathbb{Z})$, and $h\in\{b,c,c^2\}$ with $|v_i|\leq d$ $(i=1,2)$ such that
            $$\omega_{[1,t_1)}v_1\equiv \xi,\ v_1^{-1}\omega_{[t_1,t_2)}v_2\equiv h,\ v_2^{-1}\omega_{[t_2,p]}\equiv \xi^{-1}.$$
            \item If $\omega$ is conjugate to $h_0^m\ (m\geq 2)$ as an element in $\textnormal{PSL}(2,\mathbb{Z})$ for some $h_0\in\textnormal{PSL}(2,\mathbb{Z})$, where $h_0$ is primitive and 
            $$h_0=bc^{i_1}\cdots bc^{i_k}$$
            for some $k\geq 1$ and $i_j\in\{1,2\}\ (1\leq j\leq k)$,
            then there exist $1\leq t_1\leq t_2\leq t_3\leq t_4\leq p$ and elements $\xi,h,v_1,...,v_4\in\textnormal{PSL}(2,\mathbb{Z})$ with $|v_i|\leq d+1\ (1\leq i\leq 4)$ such that
            $$\begin{matrix}
             \omega_{[1,t_1)}v_1\equiv \xi,& v_1^{-1}\omega_{[t_1,t_2)}v_2\equiv v_2^{-1}\omega_{[t_2,t_3)}v_3\equiv h,\\
                v_3^{-1}\omega_{[t_3,t_4)}v_4\equiv h^{m-2},& v_4^{-1}\omega_{[t_4,p]}\equiv \xi^{-1}.
            \end{matrix}$$
        \end{enumerate}
        Here, $``\equiv"$ means that two words are the same element in $\textnormal{PSL}(2,\mathbb{Z})$.
    \end{lemma}
    \begin{proof}
        Firstly by our assumption one may write
        \[\omega=\omega_1\omega_2\cdots\omega_p,\]
        where each $\omega_i=e_i^1\cdots e_i^{l_i}\in\mathcal{S}$ with $1\leq l_i\leq d$ and $e_i^j \in \{b,c\}$. So we write 
        $$\omega_1\omega_2\cdots\omega_p=\omega=e_1e_2\cdots e_l,$$
        where $l=\sum_{i=1}^{p}l_i\leq pd$ and $e_i\in\{b,c\}\ (1\leq i\leq l)$. For any non-unitary element $\theta\in\textnormal{PSL}(2,\mathbb{Z})$, it could be represented as 
        $$\theta=f_1\cdots f_t,$$
        where $f_i\in\{b,c,c^2\}$ and $\{f_i,f_{i+1}\}=\{b,c\}$ or $\{b,c^2\}$ for each $1\leq i\leq t-1$. We call it a $(b,c)-$representation of $\theta$.  In remaining part of the proof, for $f=b,\ c$ and $c^2$, we set $f^{-1}=b,\ c^2$ and $c$ respectively.
        
       (1). We first consider the case that $\omega$ is conjugate to $b$. Then there exists $\xi_0\in\textnormal{PSL}(2,\mathbb{Z})$  such that $\omega\equiv \xi_0 b \xi_0^{-1}$. Assume 
       $$\xi_0=f_1\cdots f_t$$
       is a $(b,c)-$representation of $\xi_0$. 
       If $f_t=b$, then $f_{t-1}=c$ or $c^2$ and  
       $$\omega\equiv f_1\cdots f_tbf_t^{-1}\cdots f_1^{-1}\equiv f_1\cdots f_{t-1}b f_{t-1}^{-1}\cdots f_1^{-1},$$
       where the word in the right hand side is reduced. Take $\xi=f_1\cdots f_{t-1}$, then $\omega\equiv \xi b \xi^{-1}$, and there exist $1\leq j_1\leq j_2\leq l$ such that 
        $$e_1e_2\cdots e_{j_1}\equiv \xi\text{ and }e_1e_2\cdots e_{j_2}\equiv \xi b.$$
        Hence 
        $$e_{j_1+1}\cdots e_{j_2}\equiv b\text{ and }e_{j_2+1}\cdots e_l\equiv \xi^{-1}.$$
        Assume $e_{j_1}$ and $e_{j_2}$ are contained in $\omega_{t_1}$ and $\omega_{t_2}$ $(1\leq t_1\leq t_2\leq p)$ respectively. Write 
        $$\omega_{t_1}=e_{s_1}e_{s_1+1}\cdots e_{k_1} \textit{ \ and \ } \omega_{t_2}=e_{s_2}e_{s_2+1}\cdots e_{k_2}$$
        for some $s_1\leq j_1\leq k_1$ and $s_2\leq j_2\leq k_2$.        Take $$v_1=e_{s_1}e_{s_1+1}\cdots e_{j_1} \text{ and }v_2=e_{s_2}e_{s_2+1}\cdots e_{j_2}.$$ Then we have 
        \begin{enumerate}
            \item $|v_1|\leq|w_{t_1}|\leq d\text{ and }|v_2|\leq|w_{t_2}|\leq d$;
            \item $\omega_{[1,t_1)}v_1\equiv\omega_1\cdots \omega_{t_1-1}e_{s_1}\cdots e_{j_1}\equiv e_1\cdots e_{s_1-1} \cdot e_{s_1}\cdots e_{j_1}\equiv \xi$;
            \item $
        v_1^{-1}\omega_{[t_1,t_2)}v_2\equiv (e_{s_1}\cdots e_{j_1})^{-1}\omega_{t_1}\cdots \omega_{t_2-1}(e_{s_2}\cdots e_{j_2})\equiv b;$
        \item $v_2^{-1}\omega_{[t_2,p]}\equiv (e_{s_2}\cdots e_{j_2})^{-1}\omega_{t_2}\cdots \omega_p\equiv e_{j_2+1}\cdots e_l\equiv \xi^{-1}$.
        \end{enumerate}
        It follows that $1\leq t_1\leq t_2\leq p$ and $\xi,\ v_1,\ v_2\in\textnormal{PSL}(2,\mathbb{Z})$ and $h=b$ are desired. 
        
        If $f_t=c$ or $c^2$, take $\xi=\xi_0$. Then 
        $$\omega\equiv f_1\cdots f_tbf_t^{-1}\cdots f_1^{-1},$$
        where the word in the right hand side is reduced. With similar argument to above, one may find desired $1\leq t_1\leq t_2\leq p$ and $\xi,\ v_1,\ v_2\in\textnormal{PSL}(2,\mathbb{Z})$ and $h\in\{b,c,c^2\}$.

        The proof of Part $(1)$ is complete.
        
        (2). By assumption there exists $\xi_0\in\textnormal{PSL}(2,\mathbb{Z})$ such that $$\omega\equiv \xi_0 h_0^m\xi_0^{-1}$$ where
        $$h_0=f_1\cdots f_t$$
        is a $(b,c)-$representation of $h_0$, in which each $f_s\in\{b,c,c^{2}\}$ and $f_1=b$, $f_t=c$ or $c^{2}$. Take a $(b,c)-$representation of $\xi_0$ whose last letter could be one of $\{b,c,c^2\}$. We only consider the case that the last letter is $b$; the proofs for the cases of $c$ and $c^2$ are similar, that we leave to interested readers. Recall that $b^2 \equiv 1$. Since $f_1=b$ and the last letter of $\xi_0$ is $b$, there are certain cancellations in $\xi_0 h_0^m\xi_0^{-1}.$ So up to maximal cancellations, $\xi_0$ has a $(b,c)-$representation with one of the following forms,
        \begin{enumerate}
            \item $\xi_0=\left(f_t^{-1}\cdots f_1^{-1}\right)^q$ for some $q\geq 0$;
            \item $\xi_0=f_j^{-1}\cdots f_1^{-1}\left(f_t^{-1}\cdots f_1^{-1}\right)^q$ for some $1\leq j\leq t-1$ and $q\geq 0$;
            \item $\xi_0=g_1\cdots g_if_j^{-1}\cdots f_1^{-1}\left(f_t^{-1}\cdots f_1^{-1}\right)^q$ for some $i\geq 1$, $1\leq j\leq t-1$ and $q\geq 0$; moreover, $g_i \neq f_{j+1}^{-1}$. In this case, it is not hard to see that $g_i=f_{j+1}=c$ or $c^2$.
        \end{enumerate}
         We prove the lemma when $\xi_0$ is of form (c) above, i.e.,  
         $$\xi_0=g_1\cdots g_if_j^{-1}\cdots f_1^{-1}\left(f_t^{-1}\cdots f_1^{-1}\right)^q;$$
         the proofs are similar for the other two cases. Then we have
        \begin{align}\label{e-word}
        e_1e_2\cdots e_l=\omega \equiv g_1\cdots g_i f_{j+1}\cdots f_t\left(f_1\cdots f_t\right)^{m-1}f_1\cdots f_j g_i^{-1}\cdots g_1^{-1}.
        \end{align}
        Take $\xi=g_1\cdots g_i$ and $h=f_{j+1}\cdots f_t f_1\cdots f_j$. So we have 
        \[e_1e_2\cdots e_l=\omega \equiv \xi h^m \xi^{-1}.\]
        \begin{itemize}
        \item If $g_i=f_{j+1}=c$, then the word in the right hand side of \eqref{e-word} is reduced. It follows that there exist $1\leq j_1\leq j_2\leq j_3\leq j_4\leq l$ such that
        $$\begin{matrix}e_1\cdots e_{j_1}\equiv \xi, & e_1\cdots e_{j_2}\equiv \xi h,
        \\
        e_1\cdots e_{j_3}\equiv \xi h^2, & e_1\cdots e_{j_4}\equiv \xi h^m.\end{matrix}$$
        Then the claim follows from a similar argument to that of Part $(1)$.

        \item If $g_i=f_{j+1}=c^2$, then $\omega$ could be written as a reduced word of form 
        $$\omega\equiv g_1\cdots g_{i-1}cf_{j+2}\cdots f_t\left(f_1\cdots f_t\right)^{m-1}f_1\cdots f_j g_i^{-1}\cdots g_1^{-1}.$$
        Note that
        \[ g_1\cdots g_{i-1}cf_{j+2}\cdots f_t \cdot f_1\cdots f_j \equiv \xi h.\]
        It follows that there exist $1\leq j_1\leq j_2\leq j_3\leq j_4\leq l$ such that
        $$\begin{matrix}e_1\cdots e_{j_1}c\equiv g_1\cdots g_{i-1}c^2= \xi, & \ e_1\cdots e_{j_2}\equiv \xi h,
        \\
        e_1\cdots e_{j_3}\equiv \xi h^2, & e_1\cdots e_{j_4}\equiv \xi h^m.\end{matrix}$$
        Then the claim follows from a similar argument to that of Part $(1)$.
        \end{itemize}
        
        The proof of Part $(2)$ is also complete.
    \end{proof}
   
    The following lemma will be applied later.
    \begin{lemma}\label{l-sc}
       Assume $\alpha=\sum\limits_{\gamma\in\textnormal{PSL}(2,\mathbb{Z})}a_{\gamma}\gamma\in\mathbb{C}\left[\textnormal{PSL}(2,\mathbb{Z})\right]$ has all coefficients $a_\gamma\geq 0$. Then for any $u_i\in\textnormal{PSL}(2,\mathbb{Z})\ (1\leq i\leq 4)$ and $m_1,m_2\geq 0$, the following inequalities hold:
     $$\sum\limits_{\xi\in\textnormal{PSL}(2,\mathbb{Z})}\langle\delta_{u_1\xi u_2},\lambda(\alpha)^{m_1}\delta_e\rangle\langle\delta_{u_3\xi u_4},\lambda(\alpha)^{m_2}\delta_e\rangle\leq ||\lambda(\alpha)||^{m_1+m_2}$$
     and
     $$\sum\limits_{\xi\in\textnormal{PSL}(2,\mathbb{Z})}\langle\delta_{u_1\xi u_2},\lambda(\alpha)^{m_1}\delta_e\rangle\langle\delta_{u_3\xi^{-1}u_4},\lambda(\alpha)^{m_2}\delta_e\rangle\leq ||\lambda(\alpha)||^{m_1+m_2}.$$
    \end{lemma}
    \begin{proof}
        Since for any $t\geq 0$,
        $$\sum\limits_{\gamma\in\textnormal{PSL}(2,\mathbb{Z})}\langle \delta_\gamma,\lambda(\alpha)^t\delta_e\rangle^2=\left|\left|\lambda(\alpha)^t\delta_e\right|\right|^2\leq ||\lambda(\alpha)||^{2t},$$
        together with Cauchy-Schwarz inequality, it follows that
        \begin{align*}
            &\ \ \ \ \sum\limits_{\xi\in\textnormal{PSL}(2,\mathbb{Z})}\langle\delta_{u_1\xi u_2},\lambda(\alpha)^{m_1}\delta_e\rangle\langle\delta_{u_3\xi u_4},\lambda(\alpha)^{m_2}\delta_e\rangle\\
            &\leq\sqrt{\left(\sum\limits_{\xi\in\textnormal{PSL}(2,\mathbb{Z})}\langle\delta_{u_1\xi u_2},\lambda(\alpha)^{m_1}\delta_e\rangle^2\right)\left(\sum\limits_{\xi\in\textnormal{PSL}(2,\mathbb{Z})}\langle\delta_{u_3\xi u_4},\lambda(\alpha)^{m_2}\delta_e\rangle^2\right)}\\
     &\leq\sqrt{||\lambda(\alpha)||^{2m_1}||\lambda(\alpha)||^{2m_2}}=||\lambda(\alpha)||^{m_1+m_2}.
        \end{align*}
        
        Similarly, we can also prove the second inequality.

        The proof is complete.
    \end{proof}
    
    Now we start to prove Proposition \ref{p-bound}.
   
    \begin{proof}[Proof of Proposition \ref{p-bound}]
        Fix a self-adjoint $\alpha=\sum\limits_{\gamma}a_{\gamma}\gamma$ with $$\max\limits_{a_{\gamma}\neq 0}|\gamma|=d.$$ 
        According to \cite[Theorem A]{Jo90}, if groups $A$ and $B$ both satisfy the rapid decay property, then so does their free product $A\star B$.
        Since $\textnormal{PSL}(2,\mathbb{Z})\simeq \mathbb{Z}_2\star\mathbb{Z}_3$ and any finite group satisfies the rapid decay property, it follows that $\textnormal{PSL}(2,\mathbb{Z})$ satisfies the rapid decay property. From \cite[Proposition 6.3]{MdlS24}, we can assume that $\alpha$ has positive coefficients, and after a normalization, we assume $\sum\limits_{\gamma}a_{\gamma}=1$. Thus we can write $\alpha^p=\mathbb{E}[\gamma_1\cdots\gamma_p]$ for every $p\in\mathbb{N}$, where the $\gamma_i$'s are i.i.d random variables over $\textnormal{PSL}(2,\mathbb{Z})$ with $$\mathbb{P}\left(\gamma_i=\gamma\right)=a_\gamma.$$
        Then we have
        \begin{equation}\label{e-b-1}
        \begin{aligned}
        u_1(\alpha^p)&=-\mathbb{E}\left[1_{\gamma_1\cdots\gamma_p\equiv e}\right]-\sum\limits_{h_1}\mathbb{E}\left[1_{\gamma_1\cdots\gamma_p\equiv h_1}\right]\\
        &+\sum\limits_{m=2}^{pd}(d(m)-1)\sum\limits_{h_m}\mathbb{E}\left[1_{\gamma_1\cdots\gamma_p\equiv h_m}\right],
        \end{aligned}
        \end{equation}
        where $e$ is the unitary element in $\textnormal{PSL}(2,\mathbb{Z})$, $h_1$ is taken over all elements in $\textnormal{PSL}(2,\mathbb{Z})$ which are conjugate to one of $\{b,c,c^2\}$, $h_m\ (m\geq 2)$ is taken over all elements in $\textnormal{PSL}(2,\mathbb{Z})$ which are conjugate to some $h^m$, where $h$ is a primitive word of form $$h=bc^{i_1}\cdots bc^{i_k}$$
        for some $k\geq 1$, $i_j\in\{1,2\}\ (1\leq j\leq k)$. Notice that for any $\gamma\in\textnormal{PSL}(2,\mathbb{Z})$,
        $$1_{\gamma_1\cdots\gamma_p\equiv \gamma}=\langle\delta_\gamma,\lambda(\gamma_1\cdots\gamma_p)\delta_e\rangle$$
        and
        $$\mathbb{E}\left[1_{\gamma_1\cdots\gamma_p\equiv \gamma}\right]=\langle\delta_\gamma,\lambda(\alpha)^p\delta_e\rangle.$$
        
        For the first term of \eqref{e-b-1}, we have
        \begin{align}\label{e-b-4}
           \mathbb{E}\left[1_{\gamma_1\cdots\gamma_p\equiv e}\right]=\langle \delta_e,\lambda(\alpha)^p\delta_e\rangle \leq ||\lambda(\alpha)||^p.
        \end{align}

        For the second term of \eqref{e-b-1}, from Lemma \ref{l-decom}, we have 
        \begin{align*}
            \sum\limits_{h_1}1_{\gamma_1\cdots\gamma_p\equiv h_1}&\leq \sum\limits_{1\leq t_1\leq t_2\leq p}\ \sum\limits_{\substack{\xi,v_1,v_2\in\textnormal{PSL}(2,\mathbb{Z})\\ h\in\{b,c,c^2\} \\|v_1|,\ |v_2|\leq d}}1_{\gamma_1\cdots\gamma_{t_1-1}\equiv \xi v_1^{-1}}1_{\gamma_{t_1}\cdots\gamma_{t_2}\equiv v_1hv_2^{-1}}\\
            &\times1_{\gamma_{t_2+1}\cdots\gamma_p\equiv v_2\xi^{-1}}.
        \end{align*}
        It follows that
        \begin{align*}
          \sum\limits_{h_1}\mathbb{E}\left[1_{\gamma_1\cdots\gamma_p\equiv h_1}\right]&\leq\sum\limits_{1\leq t_1\leq t_2\leq p}\ \sum\limits_{\substack{\xi,v_1,v_2\in\textnormal{PSL}(2,\mathbb{Z})\\ h\in\{b,c,c^2\}\\|v_1|,|v_2|\leq d}}\langle\delta_{\xi v_1^{-1}},\lambda(\alpha)^{t_1-1}\delta_e\rangle\\
          &\times \langle\delta_{v_1hv_2^{-1}},\lambda(\alpha)^{t_2-t_1}\delta_e\rangle\langle\delta_{v_2\xi^{-1}},\lambda(\alpha)^{p+1-t_2}\delta_e\rangle.
        \end{align*}
        Apply Lemma \ref{l-sc}, one may conclude that for any fixed $1\leq t_1\leq t_2\leq p$,  $v_1,\ v_2\in\textnormal{PSL}(2,\mathbb{Z})$ and $h\in\{b,c,c^2\}$,
        \begin{align*}
            &\ \ \ \ \sum\limits_{\xi\in\textnormal{PSL}(2,\mathbb{Z})}\langle\delta_{\xi v_1^{-1}},\lambda(\alpha)^{t_1-1}\delta_e\rangle\langle\delta_{v_1hv_2^{-1}},\lambda(\alpha)^{t_2-t_1}\delta_e\rangle\langle\delta_{v_2\xi^{-1}},\lambda(\alpha)^{p+1-t_2}\delta_e\rangle\\
            &\leq||\lambda(\alpha)||^{t_2-t_1}\times\left(\sum\limits_{\xi\in\textnormal{PSL}(2,\mathbb{Z})}\langle\delta_{\xi v_1^{-1}},\lambda(\alpha)^{t_1-1}\delta_e\rangle\langle\delta_{v_2\xi^{-1}},\lambda(\alpha)^{p+1-t_2}\delta_e\rangle\right)\\
            &\leq||\lambda(\alpha)||^{t_2-t_1}\times ||\lambda(\alpha)||^{p-t_2+t_1} =||\lambda(\alpha)||^p.
        \end{align*}
        There are at most $p^2$ choices for $1\leq t_1\leq t_2\leq p$, at most $3^{2d}$ choices for $v_1,v_2\in\textnormal{PSL}(2,\mathbb{Z})$, and $3$ choices for $h\in\{b,c,c^2\}$. Hence,
        \begin{align}\label{e-b-2}
        \sum\limits_{h_1}\mathbb{E}\left[1_{\gamma_1\cdots\gamma_p\equiv h_1}\right]\leq 3p^2\times 3^{2d}||\lambda(x)||^p.
        \end{align}
        
        For the third term of \eqref{e-b-1}, apply Lemma \ref{l-decom} and a similar argument as above, we have that for fixed $m\geq 2$, 
        \begin{align*}
            \sum\limits_{h_m}\mathbb{E}\left[1_{\gamma_1\cdots\gamma_p\equiv h_m}\right]&\leq\sum\limits_{\substack{1\leq t_1\leq t_2\\  \leq t_3\leq t_4\leq p}}\ \sum\limits_{\substack{\xi,h,v_1,...,v_4\in\textnormal{PSL}(2,\mathbb{Z})\\|v_1|,...,|v_4|\leq d+1}}\langle\delta_{\xi v_1^{-1}},\lambda(\alpha)^{t_1-1}\delta_e\rangle\\
          &\times \langle\delta_{v_1hv_2^{-1}},\lambda(x)^{t_2-t_1}\delta_e\rangle\langle\delta_{v_2hv_3^{-1}},\lambda(\alpha)^{t_3-t_2}\delta_e\rangle\\ &\times\langle\delta_{v_3h^{m-2}v_4^{-1}},\lambda(\alpha)^{t_4-t_3}\delta_e\rangle\langle\delta_{v_4 \xi^{-1}},\lambda(\alpha)^{p+1-t_4}\delta_e\rangle
        \end{align*}
        Apply Lemma \ref{l-sc}, one may conclude that for fixed $1\leq t_1\leq t_2\leq t_3\leq t_4\leq p$ and $v_1,...,v_4\in\textnormal{PSL}(2,\mathbb{Z})$,
        \begin{align*}
           &\ \ \ \ \sum\limits_{\xi,h\in\textnormal{PSL}(2,\mathbb{Z})} \langle\delta_{\xi v_1^{-1}},\lambda(\alpha)^{t_1-1}\delta_e\rangle\times \langle\delta_{v_1hv_2^{-1}},\lambda(\alpha)^{t_2-t_1}\delta_e\rangle\langle\delta_{v_2hv_3^{-1}},\lambda(\alpha)^{t_3-t_2}\delta_e\rangle\\ &\times\langle\delta_{v_3h^{m-2}v_4^{-1}},\lambda(\alpha)^{t_4-t_3}\delta_e\rangle\langle\delta_{v_4 \xi^{-1}},\lambda(\alpha)^{p+1-t_4}\delta_e\rangle\\
           &\leq ||\lambda(\alpha)||^{t_4-t_3}\times\left(\sum\limits_{\xi\in\textnormal{PSL}(2,\mathbb{Z})}\langle\delta_{\xi v_1^{-1}},\lambda(\alpha)^{t_1-1}\delta_e\rangle\langle\delta_{v_4 \xi^{-1}},\lambda(\alpha)^{p+1-t_4}\delta_e\rangle\right)\\
           &\times\left(\sum\limits_{h\in\textnormal{PSL}(2,\mathbb{Z})}\langle\delta_{v_1hv_2^{-1}},\lambda(\alpha)^{t_2-t_1}\delta_e\rangle\langle\delta_{v_2hv_3^{-1}},\lambda(\alpha)^{t_3-t_2}\delta_e\rangle\right)\\
           &\leq ||\lambda(\alpha)||^{t_4-t_3}\times ||\lambda(\alpha)||^{p-t_4+t_1}\times||\lambda(\alpha)||^{t_3-t_1}=||\lambda(\alpha)||^p.
        \end{align*}
        There are at most $p^4$ choices for $1\leq t_1\leq t_2\leq t_3\leq t_4\leq p$ and at most $3^{4d+4}$ choices for $v_1,v_2,v_3,v_4\in\textnormal{PSL}(2,\mathbb{Z})$. It follows that
        \begin{align}\label{e-b-3}
            \sum\limits_{m=2}^{pd}(d(m)-1)\sum\limits_{h_m}\mathbb{E}\left[1_{\gamma_1\cdots\gamma_p\equiv h_m}\right]\leq (pd)^2p^4\times 3^{4d+4}||\lambda(\alpha)||^p.
        \end{align}
        
        From \eqref{e-b-1}, \eqref{e-b-4}, \eqref{e-b-2} and \eqref{e-b-3}, one may complete the proof.
    \end{proof}
    \subsection{Proofs of strong convergence property and Theorem \ref{mt-1}} In this subsection, we complete the proof of Theorem \ref{mt-1}. We first prove the strong convergence property of $\textnormal{PSL}(2,\mathbb{Z})$.

    Recall that for any $(\sigma,\tau)\in\mathcal{E}_n$
    $$\Pi(\sigma,\tau)=\textnormal{std}\circ \Phi(\sigma,\tau):\textnormal{PSL}(2,\mathbb{Z})\to\mathcal{U}(6n-1),$$
    and it could be extended to $\mathbb{C}[\textnormal{PSL}(2,\mathbb{Z})]$ naturally. 
    For any real polynomial $h$ and $x\in\mathbb{C}[\textnormal{PSL}(2,\mathbb{Z})]$, $h\left(\Pi(\sigma,\tau)(x)\right)$ is a matrix. Denote by 
    $$\textnormal{tr }h\left(\Pi(\sigma,\tau)(x)\right)=\frac{1}{6n}\textnormal{Tr }h\left(\Pi(\sigma,\tau)(x)\right).$$
    Then $\textnormal{tr }h(\Pi_x):\mathcal{E}_n\to\mathbb{R}$ defined by 
    $$\textnormal{tr }h(\Pi_x)(\sigma,\tau)=\frac{1}{6n}\textnormal{Tr }h\left(\Pi(\sigma,\tau)(x)\right)$$
    is a random variable on $\mathcal{E}_n$. For any $x=\sum\limits_{\gamma\in\textnormal{PSL}(2,\mathbb{Z})}a_\gamma \gamma$, define 
    $$|x|=\max\limits_{a_\gamma\neq 0}|\gamma|\text{ and }m(x)=\#\{\gamma;\ a_{\gamma}\neq 0\}.$$
    \noindent From Proposition \ref{p-count} and an identical proof of \cite[Lemma 3.1]{MPvH25}, we have the following property.
    \begin{proposition}\label{l-rational}
       Fix  a self-adjoint element and Let $K=||x||_{  C^\star(\textnormal{PSL}(2,\mathbb{Z}))}$. Then 
       $$|u_1(h(x))|\leq 2(108q|x|)^{216}||h||_{[-K,K]}$$
       and
      $$\left|\mathbb{E}_n[\textnormal{tr }h(\Pi_x)]-\rho(h(\lambda(x)))-\frac{1}{6n}u_1(h(x))\right|\leq\frac{8(108q|x|)^{216}}{(6n)^2}||h||_{[-K,K]}$$
      for all $n\geq 1$ and every real polynomial $h$ of degree at most $q$.
      \end{proposition}
    \begin{rem*}
      
       From \eqref{e-f} and \eqref{e-inj-1} we know that for any $\omega\in\textbf{F}_2$,
       $$\mathbb{E}_n[\textnormal{fix}_\omega]=\sum\limits_{(g,\Gamma)\in\mathcal{I}(\omega)}\mathbb{E}_n(N_{\Gamma})=\sum\limits_{(g,\Gamma)\in\mathcal{I}(\omega)}\frac{1}{(6n)^{\eta(\Gamma)}}\times\frac{\prod\limits_{i=0}^{p-1}\left(1-\frac{2i}{6n}\right)\times\prod\limits_{i=1}^{q-1}\left(1-\frac{3i}{6n}\right)}{\prod\limits_{i=0}^{2p-1}\left(1-\frac{i}{6n}\right)}$$
       is a rational function of $\frac{1}{6n}$.
       One may also deduce Proposition \ref{l-rational} by directly applying the methods in \cite{CGVTvH24} without using Proposition \ref{p-count}. In this paper, the important function $u_1$ in Proposition \ref{p-bound}, which is essential in the proof of Theorem \ref{t-prob}, is defined in Proposition \ref{p-count}. We are grateful to Will Hide, Doron Puder and Ramon van Handel for kindly pointing this to us.
    \end{rem*}
    From Proposition \ref{l-rational}, the proof of \cite[Corollary 3.2]{MPvH25} gives the following result.
    \begin{corollary}\label{cor-trace}
       Fix a self-adjoint $x\in\mathbb{C}[\textnormal{PSL}(2,\mathbb{Z})]$ and $K=||x||_{ C^\star(\textnormal{PSL}(2,\mathbb{Z}))}$. Then $\nu(h)\overset{\textnormal{def}}{=}u_1(h(x))$ extends a compactly supported distribution, and 
        $$\left|\mathbb{E}_n\left[\textnormal{tr }h(\Pi_x)\right]-\rho(h(\lambda(x)))-\frac{1}{6n}\nu(h)\right|\leq\frac{32(108|x|)^{216}}{(6n)^2}||f^{(217)}||_{[0,2\pi]}$$
        for all $n\geq 1$, $h\in C^\infty(\mathbb{R})$ and $f(\theta)=h(K\cos\theta)$.
    \end{corollary}

Now we prove the following theorem which is analogous to \cite[Theorem 3.3]{MPvH25}.
    \begin{theorem}\label{t-prob}
For any self-adjoint $x\in \C[\textnormal{PSL}(2,\mathbb{Z})]$, $n\geq1$ and $\epsilon>0$, 
    \[ \mathrm{Prob}_n \Big( (\sigma,\tau)\in\mathcal{E}_n;\ \big\|\Pi(\sigma,\tau)(x)\big\| \geq (1+\epsilon)\big\| \lambda(x)\big\| \Big) \leq \frac{C|x|^{217}m(x)^{217}}{n\epsilon^{217}}, \] 
    where $C>0$ is a universal constant.
\end{theorem}
\begin{proof}
    Assume $K=||x||_{C^\star(\textnormal{PSL}(2,\mathbb{Z}))}$ and take $\delta=\epsilon||\lambda(x)||$. Then the proof of \cite[Lemma 4.10]{CGVTvH24} gives that for any $m\geq 1$ and $\epsilon>0$, there exists an $h\in C^\infty(\mathbb{R})$ with values in $[0,1]$ and $f(\theta)=h(K\cos\theta)$ such that
    \begin{align}\label{e-supp}
    \text{ $h(z)=0$ for $|z|\leq ||\lambda(x)||+\frac{\delta}{2}$ and $h(z)=1$ for $|z|\geq ||\lambda(x)||+\delta$}
        \end{align}
        and
        \begin{align}\label{e-est}
       \text{$||f^{(m+1)}||_{[0,2\pi]}\leq 16(4m)^m\left(\frac{2K}{\delta}\right)^{m+\frac{1}{2}}$.}
        \end{align}
        From Proposition \ref{p-bound} and \cite[Lemma 4.9]{CGVTvH24}, we have $$\textnormal{supp }\nu\subset[-||\lambda(x)||,||\lambda(x)||],$$ which together with \eqref{e-supp}, \eqref{e-est} and Corollary \ref{cor-trace} implies $$\rho(h(\lambda(x)))=\nu(h)=0,$$ and
    \begin{align*}
        &\ \ \ \ \mathrm{Prob}_n \left( (\sigma,\tau)\in\mathcal{E}_n;\ \big\|\Pi(\sigma,\tau)(x)\big\| \geq  \big\|\lambda(x)\big\|+\delta \right)\\
        &\leq \mathbb{E}_n\left[\textnormal{Tr\ } h( \Pi(\sigma,\tau))(x)\right]\\
        &\leq 6n\times\frac{32(108|x|)^{216}}{(6n)^2}\times 16\times864^{216}\left(\frac{2K}{\delta}\right)^{216+\frac{1}{2}}\\ &=O\left(\frac{1}{n}\left(\frac{|x|K}{\delta}\right)^{217}\right).
    \end{align*}
       Similarly to \cite[Equation (3.3)]{MPvH25}, for $x=\sum\limits_{\gamma} a_\gamma\gamma$, we have
       $$K=||x||_{C^\star(\textnormal{PSL}(2,\mathbb{Z}))}\leq \sum\limits_{\gamma}|a_{\gamma}|\leq m(x)||\lambda(x)||.$$
       It follows that
    $$\mathrm{Prob}_n \left( (\sigma,\tau)\in\mathcal{E}_n;\ \big\|\Pi(\sigma,\tau)(x)\big\| \geq  (1+\epsilon)\big\|\lambda(x)\big\| \right)\leq\frac{C|x|^{217}m(x)^{217}}{n\epsilon^{217}},$$
    where $C>0$ is a universal constant. The proof is complete.
\end{proof}

Based on Theorem \ref{t-prob}, the proof of \cite[Theorem 3.4]{MPvH25} gives the following theorem, which is a matrix coefficient version of Theorem \ref{t-prob}.
   \begin{theorem}\label{t-prob-matrix}
       For any self-adjoint $x\in M_d(\mathbb{C})\otimes\mathbb{C}[\textnormal{PSL}(2,\mathbb{Z})]$, $n\geq 1$ and $\epsilon>0$, we have 
       $$\textnormal{Prob}_n\left((\sigma,\tau)\in\mathcal{E}_n;\ ||[\textnormal{id}\otimes\Pi(\sigma,\tau)](x)||\geq (1+\epsilon)||[\textnormal{id}\otimes\lambda](x)||\right)\leq\frac{dC|x|^{217}m(x)^{217}}{n\epsilon^{217}},$$
       where $C>0$ is a universal constant.
   \end{theorem}
 From Theorem \ref{t-bmrep} and Theorem \ref{t-prob-matrix}, we obtain the following key result.
\bt\label{mt-oc}
There exists a universal constant $c>0$ such that
 \begin{equation}\label{e-gap}
      \begin{aligned}\lim\limits_{n\to\infty}\textnormal{Prob}_n\left((\sigma,\tau)\in\mathcal{E}_n;\ \begin{matrix}\textnormal{Spec}(\Delta_{\mathbb{H}/\textnormal{PSL}(2,\mathbb{Z})})\cap\left(0,\frac{1}{4}-\frac{c}{\log n} \right) \vspace{5pt}\\ 
      =\textnormal{Spec}(\Delta_{S^O(\sigma,\tau)})\cap\left(0,\frac{1}{4}-\frac{c}{\log n}\right)\end{matrix}\right)=1. \nonumber
       \end{aligned}\end{equation}  
\et

\bp
We leave the proof  into Appendix A.
\ep

     Finally, we are ready to prove Theorem \ref{mt-1}. 
   \begin{proof}[Proof of Theorem \ref{mt-1}]
      Since $\lambda_1(\mathbb{H}/\textnormal{PSL}(2,\mathbb{Z}))>\frac{1}{4}$ (see e.g. \cite[Theorem 3.38]{Be-book}),       
       it follows from Theorem \ref{mt-oc} that 
       \begin{align}\label{e-cusp}\lim\limits_{n\to\infty}\textnormal{Prob}_{\textnormal{BM}}^n\left((\Gamma,\mathcal{O})\in\mathcal{F}_n^\star;\ \textnormal{Spec}(\Delta_{S^O(\Gamma,\mathcal{O})})\cap\left(0,\frac{1}{4}-\frac{c}{\log n}\right)=\emptyset\right)=1. \nonumber
       \end{align}
       This together with Proposition \ref{p-compare} and \eqref{pro-genus} completes the proof.
   \end{proof} 
\appendix
\section{Proof of Theorem \ref{mt-oc}}
We are very grateful to Will Hide for the helpful discussions on this part.

For any complete hyperbolic surface $X$, since the spectrum of $\Delta_X$ is a subset of $\mathbb{R}_{\geq 0}$, one can define the heat operator $\exp(-t\Delta_X):L^2(X)\to L^2(X)$ by functional calculus for $t\geq 0$. There exists a unique smooth function $H_X(t,x,y)\in C^\infty(\mathbb{R}_+\times X\times X)$, called the heat kernel, such that for all $f\in L^2(X)$ and $x\in X$,
$$\exp(-t\Delta_X)f(x)=\int_X H_X(t,x,y)f(y)dy.$$
From \cite[Lemma 7.4.26]{Bu-book}, for any $x,\ y\in\mathbb{H}$, there exists a universal constant $C>0$ such that
\begin{equation}
\begin{aligned}\label{heat-est}
    H_{\mathbb{H}}(t,x,y)\leq Ct^{-1}\exp\left(-\frac{d_{\mathbb{H}}(x,y)^2}{8t}\right).
\end{aligned}
\end{equation}
One may refer to \cite[Chapter 7]{Bu-book} for more details of the heat kernel.

For simplicity, we write $\textnormal{PSL}(2,\mathbb{Z})$ as $\Gamma$ in this appendix. Assume $\mathcal{M}=\mathbb{H}/\Gamma$ is the modular surface. Then for any $(\sigma,\tau)\in\mathcal{E}_n$,
$$L^2(S^O(\sigma,\tau))\simeq L^2_{\text{new}}(S^O(\sigma,\tau))\oplus L^2(\mathcal{M})$$
and
$$L^2_{\text{new}}(S^O(\sigma,\tau))\simeq L^2(F)\otimes V_{6n}^0,$$
where 
$$V_{6n}^0=\left\{(x_1,...,x_{6n})\in\mathbb{R}^{6n};\ \sum\limits_{i=1}^{6n}x_i=0\right\}.$$
Then we have $\exp\left(-t\Delta_{S^O(\sigma,\tau)}\right):L^2_{\text{new}}(S^O(\sigma,\tau))\to L^2_{\text{new}}(S^O(\sigma,\tau))$. 

Now we start to estimate the operator norm $\left|\left|\exp\left(-t\Delta_{S^O(\sigma,\tau)}\right)\right|\right|_{L^2_{\text{new}}(S^O(\sigma,\tau))}$. Fix a point $o\in\mathcal{M}$ and define
$$\mathcal{C}_R=\{x\in \mathcal{M};\ d(x,o)>R\}.$$
For any $(\sigma,\tau)\in\mathcal{E}_n$, let $\mathcal{C}_R^{(\sigma,\tau)}$ be a lift of $\mathcal{C}_R$ in $S^O(\sigma,\tau)$ and $\chi_{\mathcal{C}_R^{(\sigma,\tau)}}:S^O(\sigma,\tau)\to\mathbb{R}$ is defined by
$$\chi_{\mathcal{C}_R^{(\sigma,\tau)}}(x)=\begin{cases}
    1 & x\in \mathcal{C}_R^{(\sigma,\tau)};\\ 0 & x\notin \mathcal{C}_R^{(\sigma,\tau)}.
\end{cases}$$
Then we may write 
$$\exp\left(-t\Delta_{S^O(\sigma,\tau)}\right)=\exp\left(-t\Delta_{S^O(\sigma,\tau)}\right)\chi_{\mathcal{C}_R^{(\sigma,\tau)}}+\exp\left(-t\Delta_{S^O(\sigma,\tau)}\right)\left(1-\chi_{\mathcal{C}_R^{(\sigma,\tau)}}\right).$$
For the first part, the proof of \cite[Proposition 3.6]{Moy25} yields
\begin{proposition}\label{p-moy}
    For some $R,\ t>0$ large enough and $r>0$ such that $R\geq 2r$, consider the operator
    $$\exp\left(-t\Delta_{S^O(\sigma,\tau)}\right)\chi_{\mathcal{C}_R^{(\sigma,\tau)}}.$$
    Then 
    $$\exp\left(-t\Delta_{S^O(\sigma,\tau)}\right)\chi_{\mathcal{C}_R^{(\sigma,\tau)}}=\chi_{\mathcal{C}_r^{(\sigma,\tau)}}\exp\left(-t\Delta_{\mathcal{C}_r^{(\sigma,\tau)}}\right)\chi_{\mathcal{C}_R^{(\sigma,\tau)}}+\mathcal{R},$$
    where $\Delta_{\mathcal{C}_r^{(\sigma,\tau)}}$ is the Dirichlet Laplacian in $\mathcal{C}_r^{(\sigma,\tau)}$ and the error term satisfies 
    $$||\mathcal{R}||_{L^2(S^O(\sigma,\tau))}\leq C_1(1+\sqrt{t})\exp\left(C_1R-\frac{c_1R^2}{t}\right)$$
    for two universal constants $C_1,\ c_1>0$.
\end{proposition}

Now we consider the part $\exp\left(-t\Delta_{S^O(\sigma,\tau)}\right)\left(1-\chi_{\mathcal{C}_R^{(\sigma,\tau)}}\right)$.
Let $\pi:\mathbb{H}\to\mathcal{M}$ be a universal covering, assume
$$\Tilde{\mathcal{C}}_R=\pi^{-1}(\mathcal{C}_R)$$
and $F$ is the Dirichlet domain of $\mathcal{M}$ about $o$. 
The operator $$\exp\left(-t\Delta_{S^O(\sigma,\tau)}\right)\left(1-\chi_{\mathcal{C}_R^{(\sigma,\tau)}}\right)$$ acts on $L^2_{\text{new}}(S^O(\sigma,\tau))$ is equivalent to that the operator
$$\sum\limits_{\gamma\in\Gamma}a_{\gamma}\otimes\Pi(\sigma,\tau)(\gamma^{-1})$$
acts on $L^2(F)\otimes V_{6n}^0$, where $a_\gamma:L^2(F)\to L^2(F)$ is an integral operator with kernel 
$$a_\gamma(x,y)=H_{\mathbb{H}}(t,x,\gamma y)(1-\chi_{\Tilde{\mathcal{C}}_R})(y).$$
Write $a_\gamma=a_\gamma^{(1)}+a_\gamma^{(2)}$, where $a_\gamma^{(1)}$ and $a_{\gamma}^{(2)}$ are integral operators with kernels $$a_\gamma(x,y) 1_{d(x,y)\leq R}\text{ and }a_\gamma(x,y) 1_{d(x,y)>R}$$ respectively. 
Similar to the proof of \cite[Section 4.5]{Moy25}, one may obtain that
\begin{equation}
\begin{aligned}\label{e-moy}
\left|\left|\sum\limits_{\gamma\in\Gamma}a_{\gamma}^{(2)}\otimes\Pi(\sigma,\tau)(\gamma^{-1})\right|\right|_{L^2(F)\otimes V_{6n}^0}\leq C_2\exp\left(C_2R-c_2\frac{R^2}{t}\right)
\end{aligned}
\end{equation}
for two universal constants $C_2,c_2>0$. 

For the remaining part, we assume $t\geq 1$ 
and
$$S_R=\left\{\gamma\in\Gamma;\ a_\gamma^{(1)}\neq 0\right\}.$$
Together with \eqref{heat-est} and similar to the proofs in \cite{HM23} or \cite{HMTcov}, we have the following two lemmas.
\begin{lemma}\label{l-1}
    There exists a universal constant $C>0$ such that 
    $$\#S_R\leq Ce^{4R},$$
and for any $\gamma\in S_R$, it has word length $|\gamma|\leq Ce^{4R}$.
\end{lemma}
 
\begin{lemma}\label{l-2}
    For any $k\in\mathbb{N}$, there exist a universal constant $C>0$, a finite dimensional subspace $W_k\subset L^2(F)$ of rank $d(k)\leq k\#S_R$ and operators $b_\gamma^{(k)}:W_k\to W_k$ such that for every $\gamma\in S_R$,
    $$\left|\left|a_\gamma^{(1)}-b_\gamma^{(k)}\right|\right|_{L^2(F)\to L^2(F)}\leq \frac{Ce^{4R}}{\sqrt{k}}.$$
\end{lemma}
Consider
 $$x=\sum\limits_{\gamma\in\Gamma}b_\gamma^{(k)}\otimes \gamma^{-1}\in M_{d(k)}(\mathbb{C})\otimes \mathbb{C}[\Gamma]$$ 
 and
 $$\tilde{x}=\sum\limits_{\gamma\in\Gamma}\tilde{b}_\gamma^{(k)}\otimes \gamma^{-1}\in M_{2d(k)}(\mathbb{C})\otimes \mathbb{C}[\Gamma],$$
where $d(k)$ is defined in Lemma \ref{l-2} and for $\gamma\in\Gamma$,
$$\tilde{b}_\gamma^{(k)}=\begin{pmatrix}
    0 & b_\gamma^{(k)}\\ \left(b_\gamma^{(k)}\right)^\star &0
\end{pmatrix}.$$
 Then $\tilde{x}$ is self-adjoint and for any unitary representation $\rho:\Gamma\to V$,
 \begin{align}\label{e-norm-rep}
 \left|\left|\sum\limits_{\gamma\in\textnormal{PSL}(2,\mathbb{Z})}b_\gamma^{(k)}\otimes\rho(\gamma^{-1})\right|\right|_{M_{d(k)}(\mathbb{C})\otimes V}=\left|\left|\sum\limits_{\gamma\in\textnormal{PSL}(2,\mathbb{Z})}\tilde{b}_\gamma^{(k)}\otimes\rho(\gamma^{-1})\right|\right|_{M_{2d(k)}(\mathbb{C})\otimes V}.
 \end{align}
 From  Lemma \ref{l-1} and Lemma \ref{l-2}, we have 
$$d(k)\leq k\#S(R)\leq Cke^{2R},$$
and
$$|\tilde{x}|=|x|\leq  Ce^{2R}\text{ and }m(\tilde{x})=m(x)\leq Ce^{2R}.$$
Take $$\epsilon=\left(\frac{m(x)^{217}|x|^{217}d(k)\log n}{n}\right)^{\frac{1}{217}}.$$ 
Apply Theorem \ref{t-prob-matrix} to $\tilde{x}$ and $\epsilon$, we have that with probability at least $1-O\left(\frac{1}{\log n}\right)$, a random element $(\sigma,\tau)\in\mathcal{E}_n$ satisfies 
\begin{align*}
&\ \ \ \ \left|\left|\sum\limits_{\gamma\in\textnormal{PSL}(2,\mathbb{Z})}\tilde{b}_\gamma^{(k)}\otimes\Pi(\sigma,\tau)(\gamma^{-1})\right|\right|_{M_{2d(k)}(\mathbb{C})\otimes V_{6n}^0}\\
&\leq (1+\epsilon)\left|\left|\sum\limits_{\gamma\in\Gamma}\tilde{b}_\gamma^{(k)}\otimes\lambda(\gamma^{-1})\right|\right|_{M_{2d(k)}(\mathbb{C})\otimes L^2(\Gamma)}\\
&\leq \left|\left|\sum\limits_{\gamma\in\Gamma}\tilde{b}_\gamma^{(k)}\otimes\lambda(\gamma^{-1})\right|\right|_{M_{2d(k)}(\mathbb{C})\otimes L^2(\Gamma)}\left(1+\left(\frac{Cke^{830 R}\log n}{n}\right)^{\frac{1}{217}}\right).
\end{align*}
Since $\Pi(\sigma,\tau)$ and $\lambda$ are both unitary representations, it follows from \eqref{e-norm-rep} that
\begin{equation}
\begin{aligned}\label{e-com-1}
&\ \ \ \ \left|\left|\sum\limits_{\gamma\in\textnormal{PSL}(2,\mathbb{Z})}b_\gamma^{(k)}\otimes\Pi(\sigma,\tau)(\gamma^{-1})\right|\right|_{M_{d(k)}(\mathbb{C})\otimes V_{6n}^0}\\
&\leq \left|\left|\sum\limits_{\gamma\in\Gamma}b_\gamma^{(k)}\otimes\lambda(\gamma^{-1})\right|\right|_{M_{d(k)}(\mathbb{C})\otimes L^2(\Gamma)}\left(1+\left(\frac{Cke^{830 R}\log n}{n}\right)^{\frac{1}{217}}\right).
\end{aligned}
\end{equation}
It then follows from Lemma \ref{l-1} and Lemma \ref{l-2} that
\begin{align*}
\left|\left|\sum\limits_{\gamma\in\Gamma}a_\gamma^{(1)}\otimes\Pi(\sigma,\tau)(\gamma^{-1})\right|\right|_{L^2(F)\otimes V_{6n}^0}\leq \left|\left|\sum\limits_{\gamma\in\Gamma}b_\gamma^{(k)}\otimes\Pi(\sigma,\tau)(\gamma^{-1})\right|\right|_{L^2(F)\otimes V_{6n}^0}+\frac{e^{8R}}{\sqrt{k}}
\end{align*}
and 
\begin{align*}
\left|\left|\sum\limits_{\gamma\in\Gamma}b_\gamma^{(k)}\otimes\lambda(\gamma^{-1})\right|\right|_{L^2(F)\otimes L^2(\Gamma)}\leq \left|\left|\sum\limits_{\gamma\in\Gamma}a_\gamma^{(1)}\otimes\lambda(\gamma^{-1})\right|\right|_{L^2(F)\otimes L^2(\Gamma)}+\frac{e^{8R}}{\sqrt{k}}.
\end{align*}
Together with \eqref{e-com-1}, it follows that 

\begin{align*}
    &\ \ \ \ \left|\left|\sum\limits_{\gamma\in\Gamma}a_\gamma^{(1)}\otimes\Pi(\sigma,\tau)(\gamma^{-1})\right|\right|_{L^2(F)\otimes V_{6n}^0}\\
    &\leq \left(\left|\left|\sum\limits_{\gamma\in\Gamma}a_\gamma^{(1)}\otimes\lambda(\gamma^{-1})\right|\right|_{L^2(F)\otimes L^2(\Gamma)}+\frac{e^{8R}}{\sqrt{r}}\right)\left(1+\left(\frac{Cke^{828 R}\log n}{n}\right)^{\frac{1}{217}}\right)+\frac{e^{8R}}{\sqrt{k}}.
\end{align*}
Notice that $L^2(F)\otimes L^2(\Gamma)\simeq L^2(\mathbb{H})$ and under this identification $$T=\sum\limits_{\gamma\in\Gamma}a_\gamma^{(1)}\otimes\lambda(\gamma^{-1}):L^2(\mathbb{H})\to L^2(\mathbb{H})$$ is an integral operator with kernel 
$$T(x,y)=H_{\mathbb{H}}(t,x,y)1_{d(x,y)\leq R}(1-\chi_{\Tilde{\mathcal{C}}_R})(y).$$
Since
$$0\leq T(x,y)\leq H_{\mathbb{H}}(t,x,y)=e^{-t\Delta_{\mathbb{H}}}(x,y),$$
it follows that 
$$\left|\left|T\right|\right|_{L^2(\mathbb{H})}\leq \left|\left|e^{-t\Delta_{\mathbb{H}}}\right|\right|_{L^2(\mathbb{H})}=e^{-\frac{1}{4}t}.$$
Hence we deduce that
\begin{equation}
\begin{aligned}\label{e-com-2}
   &\ \ \ \ \left|\left|\sum\limits_{\gamma\in\textnormal{PSL}(2,\mathbb{Z})}a_\gamma^{(1)}\otimes\Pi(\sigma,\tau)(\gamma^{-1})\right|\right|_{L^2(F)\otimes V_{6n}^0}\\
    &\leq \left(e^{-\frac{1}{4}t}+\frac{e^{8R}}{\sqrt{k}}\right)\left(1+\left(\frac{Cke^{830 R}\log n}{n}\right)^{\frac{1}{217}}\right)+\frac{e^{8R}}{\sqrt{k}}.
\end{aligned}
\end{equation}
Take $t=\epsilon_1\log n$, $R=\epsilon_2\log n$ and $k=\lceil n^a\rceil$, where $\epsilon_1,\ \epsilon_2,\ a\in\mathbb{R}^+$ with
 $$C_1\epsilon_2-\frac{c_1\epsilon_2^2}{\epsilon_1}\leq -3,\ C_2\epsilon_2-\frac{c_2\epsilon_2^2}{\epsilon_1}\leq -3$$
 and 
 $$32\epsilon_1,\ 32\epsilon_2\leq a,\ 830\epsilon_2+a\leq\frac{1}{2}.$$
 For $R$ large enough, $C_{R/2}^{(\sigma,\tau)}$ is a union of certain cusp neighborhoods and
 $$\inf\textnormal{Spec}\left(\Delta_{C_{R/2}^{(\sigma,\tau)}}\right)=\frac{1}{4}.$$
 Then together with Proposition \ref{p-moy}, we have
 \begin{equation}
     \begin{aligned}\label{e-norm-1}
         \left|\left|\exp(-t\Delta_{S^O(\sigma,\tau)})\chi_{\mathcal{C}_R^{(\sigma,\tau)}}\right|\right|_{L^2(S^O(\sigma,\tau))}\leq e^{-\frac{1}{4}t}+\frac{1}{n^2}.
     \end{aligned}
 \end{equation}
 From \eqref{e-moy}, we have
 \begin{align}\label{e-norm-2}
\left|\left|\sum\limits_{\gamma\in\Gamma}a_{\gamma}^{(2)}\otimes\Pi(\sigma,\tau)(\gamma^{-1})\right|\right|_{L^2(F)\otimes V_{6n}^0}\leq\frac{1}{n^2}.
 \end{align}
From \eqref{e-com-2}, we have with probability $1-O\left(\frac{1}{\log n}\right)$,
\begin{equation}
    \begin{aligned}\label{e-norm-3}
        \left|\left|\sum\limits_{\gamma\in\textnormal{PSL}(2,\mathbb{Z})}a_\gamma^{(1)}\otimes\Pi(\sigma,\tau)(\gamma^{-1})\right|\right|_{L^2(F)\otimes V_{6n}^0}\leq 3\left(e^{-\frac{1}{4}t}+\frac{1}{n^{\frac{a}{4}}}\right).
    \end{aligned}
\end{equation}
If $\Delta_{S^O(\sigma,\tau)}:L^2_{\text{new}}(S^O(\sigma,\tau))\to L^2_{\text{new}}(S^O(\sigma,\tau))$ has some discrete eigenvalues contained in $(0,\frac{1}{4})$. Assume $\lambda_1^{\text{new}}(S^O(\sigma,\tau))$ is the smallest one, then 
together with \eqref{e-norm-1}, \eqref{e-norm-2}, \eqref{e-norm-3} and the triangle inequality, one may deduce that with probability $1-O\left(\frac{1}{\log n}\right)$,
\begin{align*}
    e^{-t\lambda_1^{\text{new}}(S^O\left(\sigma,\tau)\right)}&=\left|\left|e^{-t\Delta_{S^O(\sigma,\tau)}}\right|\right|_{L^2_{\text{new}}(S^O(\sigma,\tau))}\\
    &\leq 4e^{-\frac{1}{4}t}+\frac{2}{n^2}+\frac{3}{n^{\frac{a}{4}}} \leq 5e^{-\frac{1}{4}t}
\end{align*}
which implies that 
$$\lambda_1^{\text{new}}(S^O\left(\sigma,\tau)\right)\geq \frac{1}{4}-\frac{\log 5}{\epsilon_1}\cdot\frac{1}{\log n}.$$
In summary, we may conclude the following result.
\begin{theorem}[=Theorem \ref{mt-oc}]
    There exists a universal constant $c>0$ such that
    $$\lim\limits_{n\to\infty}\textnormal{Prob}_n\left((\sigma,\tau)\in\mathcal{E}_n;\ \begin{matrix}
        \textnormal{Spec}\left(\Delta_{S^{O}(\sigma,\tau)}\right)\cap\left(0,\frac{1}{4}-\frac{c}{\log n}\right)\\
       \vspace{0.3cm} =\textnormal{Spec}\left(\Delta_{\mathbb{H}/\textnormal{PSL}(2,\mathbb{Z})}\right)\cap\left(0,\frac{1}{4}-\frac{c}{\log n}\right)
    \end{matrix}\right)=1.$$
\end{theorem}

\bibliographystyle{amsalpha}
\bibliography{ref}
\end{document}